\documentclass[a4paper,reqno]{amsart}
\usepackage[a4paper, marginparwidth=35pt, total={10cm, 10cm}]{geometry}
\usepackage{marginnote}
\usepackage{graphicx}
\usepackage{amssymb}
\usepackage{amscd}
\usepackage{amsthm}
\usepackage{amsmath}
\usepackage{enumerate}
\usepackage{float}
\usepackage{mathrsfs}
\usepackage{graphics}
\usepackage[all,cmtip]{xy}
\usepackage{fullpage}
\usepackage{fbox}
\usepackage{xcolor}
\usepackage{tikz}
\usetikzlibrary{patterns, hobby}
\usepackage{extarrows}
\usepackage{tikz-cd}
\usepackage{ulem}
\usetikzlibrary{tqft,calc}
\usetikzlibrary{matrix,arrows,decorations.pathmorphing}
\usetikzlibrary{arrows,decorations.markings, matrix, angles,quotes}

\usepackage[colorlinks=true,pagebackref,hyperindex]{hyperref}
\newcommand\myshade{85}
\colorlet{mylinkcolor}{violet}
\colorlet{mycitecolor}{red}
\colorlet{myurlcolor}{cyan}

\hypersetup{
	linkcolor  = mylinkcolor!\myshade!black,
	citecolor  = mycitecolor!\myshade!black,
	urlcolor   = myurlcolor!\myshade!black,
	colorlinks = true,
}

\newtheorem{Thm}{Theorem}[section]\newtheorem*{Thm*}{Theorem}

\newtheorem{Lem}[Thm]{Lemma}
\newtheorem{Cor}[Thm]{Corollary}

\newtheorem{Prop}[Thm]{Proposition}
\newtheorem{Prop-Def}[Thm]{Proposition-Definition}
\newtheorem*{Conj*}{Conjecture}

\theoremstyle{definition}
\newtheorem{Ex}[Thm]{Example}
\newtheorem{Def}[Thm]{Definition}
\newtheorem{Convention}[Thm]{Convention}
\newtheorem{Assu}[Thm]{Assumption}
\newtheorem{Rem}[Thm]{Remark}

\newcommand{\w}{\widetilde}
\newcommand{\ra}{\rightarrow}

\newcommand{\D}{\mathcal{D}}

\newcommand{\cal}{\mathcal}
\newcommand{\T}{\mathcal T}
\newcommand{\A}{\mathcal A}

	\newcommand{\bb}{\mathrm{b}}

	\newcommand{\h}{{\mathrm H}}

	\renewcommand{\dim}{{\rm dim}}

	\newcommand{\xra}{\xrightarrow}
	\newcommand{\Z}{{\mathbb Z}}

	\newcommand{\op}{\oplus}
	\newcommand{\bop}{\bigoplus}
	\newcommand{\ot}{\otimes}

	\newcommand{\Hom}{\operatorname{Hom}\nolimits}

	\newcommand{\End}{\operatorname{End}\nolimits}
	\newcommand{\RHom}{\mathbf{R}\strut\kern-.2em\operatorname{Hom}\nolimits}
	\newcommand{\RshHom}{\mathbf{R}\strut\kern-0em\mathscr{H}\strut\kern-.3em\operatorname{om}\nolimits}
	\newcommand{\shHom}{\mathscr{H}\strut\kern-.3em\operatorname{om}\nolimits}
	\newcommand{\shEnd}{\mathscr{E}\strut\kern-.3em\operatorname{nd}\nolimits}

	\newcommand{\za}{\alpha}
	\newcommand{\zb}{\beta}
	\newcommand{\zd}{\delta}
	\newcommand{\zD}{\Delta}
	
	\newcommand{\zg}{\gamma}
	
	\newcommand{\zl}{\lambda}
	\newcommand{\zs}{\sigma}

	\DeclareMathOperator{\moduleCategory}{{\mathsf{mod}}} \renewcommand{\mod}{\moduleCategory}
	
	\DeclareMathOperator{\thick}{\mathsf{thick}}
	\DeclareMathOperator{\per}{\mathsf{per}}

	\numberwithin{equation}{section}
	
	\definecolor{dark-green}{RGB}{14,150,2}
	\definecolor{red}{RGB}{250,0,0}

\begin{document}
		\normalem
		\title{A complete derived invariant and silting theory for graded gentle algebras}

		\author{Haibo Jin}
		\address{Haibo Jin, School of Mathematical Sciences,  Key Laboratory of Mathematics and Engineering Applications (Ministry of Education),   Shanghai Key laboratory of PMMP,
			East China Normal University,
			Shanghai 200241, China and Department of Mathematics, Universit\"at zu K\"oln, Weyertal 86-90, 50931 K\"oln, Germany}
		\email{hbjin@math.ecnu.edu.cn}
		
		\author{Sibylle Schroll}
		\address{Sibylle Schroll, Department of Mathematics, Universit\"at zu K\"oln, Weyertal 86-90, 50931 K\"oln, Germany and
			Institutt for matematiske fag, NTNU, N-7491 Trondheim, Norway}
		\email{schroll@math.uni-koeln.de}
		
		\author{Zhengfang Wang}
		\address{Zhengfang Wang, School of Mathematics, Nanjing University, Nanjing 210093, Jiangsu,
			PR China and Institut f\"ur Algebra und Zahlentheorie, Universit\"at Stuttgart, Pfaffen\-wald\-ring 57, 70569 Stuttgart, Germany}
		\email{zhengfangw@nju.edu.cn}

		\thanks{The first author was supported by the National Key R\&D Program of China (2024YFA1013801) and by Science and Technology Commission of Shanghai Municipality (No. 22DZ2229014). The third author was supported by the National Key R\&D Program of China (2024YFA1013803), the  NSFC (Nos. 13004005 and 12371043) and  the DFG (WA 5157/1-1). The first and second authors were also supported by  the DFG through the project SFB/TRR 191 Symplectic Structures in Geometry, Algebra and Dynamics (Projektnummer 281071066--TRR 191).
			\\
			On behalf of all authors, the corresponding author states that there is no conflict of interest. }

		\date{\today}
		

		
		\begin{abstract}
			
			We confirm a conjecture by Lekili and Polishchuk that the geometric invariants which they construct for homologically smooth graded (not necessarily proper) gentle algebras form a complete derived invariant. Hence, we obtain a complete invariant of triangle equivalences for partially wrapped Fukaya categories of graded surfaces with stops.  
			
			A key ingredient of the proof is the full description of homologically smooth graded  gentle algebras whose perfect derived categories admit silting objects. We also apply this to classify which graded gentle algebras admit pre-silting objects that are not partial silting. In particular, this allows us to construct a family of counterexamples to the question whether any pre-silting object in the derived category of a finite-dimensional algebra is partial silting.
		\end{abstract}

		\maketitle
		
		\tableofcontents

		\section{Introduction}
		
		Derived categories and more generally triangulated categories are important in many areas of mathematics such as algebraic geometry and representation theory of algebras. One important question arising in this context is when two of these categories are equivalent. For derived categories of finite dimensional algebras, tilting objects play an important role in answering this question, see for example \cite{Rickard89}. As a generalization of tilting objects, silting objects, introduced in \cite{KV88}, give a way to explicitly construct equivalences of categories in a more general setting, see for example \cite{Keller94,Ke98}. Since their introduction, silting objects have been found to have deep connections with other important mathematical subjects such as Bridgeland stability conditions, cluster theory, torsion theory, and $t$-structures, see, for example,  \cite{AI,AIR,CS20, IY,KoY,O, QW}. However, it is not always easy to construct silting objects or even to determine whether they exist. Indeed, in many cases silting objects may not exist.  
		For this reason, one would like to have other ways to determine whether two categories are equivalent or not. Complete derived invariants are such a way. Of particular interest are those invariants that can be calculated using, for example, geometric data associated to the category. 
		
		The main thrust of this paper is to study triangle equivalences in the following situation. Namely, we are interested in derived equivalences of graded gentle algebras, in particular, in the context of their relation to partially wrapped Fukaya categories of surfaces with stops \cite{HKK17, LP20}. 
		
		Gentle algebras first arose in the representation theory of finite dimensional algebras in \cite{AsH, AS87}. Their representation theory is governed by word combinatorics and has been well studied since their inception \cite{AS87}. 
		Remarkably, gentle algebras occur in many areas of mathematics such as cluster theory \cite{ABCP, Labardini}, N=2 gauge theories \cite{Cecotti} and most importantly for the purposes of this paper, they naturally appear in the context of homological mirror symmetry of two dimensional manifolds. Indeed, it has been shown in \cite{HKK17, LP20} that partially wrapped Fukaya categories of surfaces with stops are triangle equivalent to perfect derived categories of homologically smooth graded gentle algebras. More precisely, in \cite{HKK17} it is shown  that (graded) gentle algebras naturally appear as  endomorphism algebras of formal generators in  partially wrapped Fukaya categories associated to graded surfaces with stops. 
		Conversely, in  \cite{LP20, OPS18} it is shown that to any graded gentle algebra one may associate  a surface model in the form of  a graded surface with stops together with an admissible dissection corresponding to a  formal generator. %
		From the perspective of representation theory,  surface models provide a convenient and geometric approach to study derived categories of graded gentle algebras. 
		For instance, the indecomposable objects, morphisms and algebraic operations (e.g.\ mapping cones or mutations) between these objects correspond to natural geometric objects and operations in the associated  surface models, see \cite{CS22, HKK17, IQZ, OPS18, QZZ22}.

		In this paper, we study the silting theory for graded gentle algebras which are homologically smooth but not necessarily proper. In particular, we explicitly determine for which graded gentle algebras their perfect derived categories admit silting objects. This in turn allows us to fully characterise the existence of silting objects in terms of the surface associated to a graded gentle algebra and thus naturally allows a description of which partially wrapped Fukaya categories admit silting objects.

		As applications of this result,  
		we address two open questions.  As a first application we give a positive answer to a conjecture by Lekili and Polishchuk \cite[Remark 3.19]{LP20} giving a complete derived invariant for graded gentle algebras in terms of their surface models. This also gives a complete invariant for triangulated equivalence classes of partially wrapped Fukaya categories of graded surfaces with stops. 
		For the ungraded case a complete derived invariant based on a geometric interpretation of silting objects  is given in \cite{APS19, O19}.
		However, we point out that the  difficulty in passing from the ungraded case to the graded case is that in the graded case there might not exist a silting object in the perfect derived category  whereas in the ungraded case the algebra itself is always a silting (even tilting) object. As a second application, we provide a sufficient and necessary condition to determine for which  homologically smooth and proper graded gentle algebras all pre-silting objects in their  derived categories can be completed to silting objects or in other words, we determine for which graded gentle algebras all pre-silting objects are partial tilting (i.e.\ a direct summand of a silting object).
		Based on this, we construct an infinite family of counter examples, which includes  the example previously given in \cite{LZ23}, to the representation theoretic question of whether any pre-silting object in the derived category of a finite-dimensional algebra is  partial tilting.

		In the following we state our three main results more precisely.

		\subsection{Existence of silting objects}
		The notion of silting objects, introduced in \cite{KV88} as a generalization of tilting objects, plays an important role in the study of the representation theory of algebras. It is one of the standard tools for studying triangulated categories and it enables us to control equivalences of triangulated categories (see, for example, \cite{Keller94,Ke98, Rickard89}).

		Recall that an object $P$ in a triangulated category $\T$ is called \emph{pre-silting} if 
		$$
		\Hom_{\T}(P, P[i])=0,\quad \text{for each $i> 0$},
		$$ 
		and it is called \emph{silting} if, moreover, $\T$ coincides with  $\thick_{\T}(P)$, the smallest thick subcategory of $\T$ containing $P$ \cite{KV88}. Silting objects appear naturally in many triangulated categories. For instance, for a non-positive dg (i.e.\ non-positively graded differential) algebra $A$, the perfect derived category $\per(A)$ admits $A$ itself as a silting object. More generally,  if an algebraic, idempotent complete, triangulated category $\T$ admits a silting object then $\T$ is triangle equivalent to $\per(A)$ for some non-positive dg algebra $A$ \cite{Keller94,Keller06}. 
		
		The first result of this paper is to give a complete classification of the graded gentle algebras $A$ such that $\per(A)$ admits silting objects.

		\begin{Thm}[The contrapositive of Theorem \ref{Thm:existence_of_silting}]\label{Thm:1.1introduction}
			Let $A=\Bbbk Q/I$ be a homologically smooth graded gentle algebra.  Then $\per(A)$ admits no silting objects if and only if one of the following holds.
			\begin{enumerate}[\rm(a)]
				\item There exists a cycle $p=\alpha_1 \alpha_2\dotsb \alpha_n$ such that $\alpha_1 \alpha_{2}, \dotsc, \alpha_{n-1}\alpha_n, \alpha_n \alpha_1 \notin I$ and $|p|>0$. \label{Conditiona}
				\item As a graded algebra, $A$ is isomorphic to a graded gentle algebra given by the quiver
				\begin{align*}
					\xymatrix{ & 1 \ar@<1pc>[r]^-{\za} \ar@<-1pc>[r]^-{\zg}& 2 \ar[l]_-{\zb}} 
				\end{align*}
				with relations $\{\za\zb, \zb\zg  \}$ and satisfying  $|\za|+|\zb|=1$ and $|\zb|+|\zg|=1$. \label{Conditionb}
			\end{enumerate}
		\end{Thm}
		
		Condition \eqref{Conditiona} in Theorem \ref{Thm:1.1introduction} particularly implies that $A$ is infinite dimensional and  Condition \eqref{Conditionb} implies that $A$ is finite-dimensional. Theorem \ref{Thm:1.1introduction} generalizes the result in \cite[Theorem 5.9]{CJS}, which gives a partial classification of the homologically smooth and proper graded gentle algebras $A$ whose  $\per(A)$  admit  silting objects.

		As a corollary, we obtain the following result for homologically smooth and proper (i.e.\ finite-dimensional) graded gentle algebras.

		\begin{Cor}\label{Thm:existence_of_siltingA}
			Let $A=\Bbbk Q/I$ be a homologically smooth and proper graded gentle algebra. Then $\per(A)$ does not admit  silting objects if and only if $A$ is isomorphic to a graded gentle algebra given by the  quiver
			\begin{align*}
				\xymatrix{ & 1 \ar@<1pc>[r]^-{\za} \ar@<-1pc>[r]^-{\zg}& 2 \ar[l]_-{\zb}} 
			\end{align*}
			with relations $\{\za\zb, \zb\zg  \}$ and satisfying  $|\za|+|\zb|=1$ and $|\zb|+|\zg|=1$. 
			
		\end{Cor}
		
		In terms of partially wrapped Fukaya categories of graded surfaces with stops, Theorem \ref{Thm:1.1introduction} is translated into the following form. 
		
		\begin{Cor}[Corollary \ref{Cor:existence_of_silting}]\label{cor:existence-of-siltingA}
			Let ${\cal W} (S,M, \eta)$ be the partially wrapped Fukaya category of the graded surface with stops $(S,M, \eta)$.  Then  ${\cal W} (S,M, \eta)$  does not admit silting objects if and only if one of the following holds. 
			\begin{enumerate}[\rm(i)]
				\item There exists a non-stopped boundary component $\partial_iS$ (i.e.\ $\partial_i S \cap M = \emptyset$)  whose winding number is negative. 
				\item $S$ is the torus with one boundary component and one stop (i.e.\ $\# M =1$) such that the winding number of each non-separating simple closed curve is zero. 
			\end{enumerate}
		\end{Cor}
		The condition that the winding number of each non-separating simple closed curve is zero is equivalent to the invariant $\w{A}(\eta)$ being equal to zero, where we refer to Formula (\ref{align:wA}) in Subsection \ref{subsection:combinatorialinvariants}  for the definition of  $\w{\A}(\eta)$.

		The proof of Theorem \ref{Thm:1.1introduction}  uses the numerical invariants for line fields constructed in \cite{LP20}. Specifically, if a graded surface $(S, M, \eta)$ with stops does not satisfy Conditions  \eqref{Conditiona} and \eqref{Conditionb} of Theorem \ref{Thm:1.1introduction} then we can construct an explicit admissible dissection $\Delta$ on $S$ such that the associated gentle algebra $A(\Delta)$ is concentrated in non-positive degree, see Theorem \ref{Cor:siltingexistence}. For the converse, if a graded gentle algebra $A$ satisfies Condition \eqref{Conditiona} {\it or} \eqref{Conditionb} then  we use the invariants to show that  $\per(A)$ cannot admit silting objects, see Propositions \ref{Prop:no_silting} and  \ref{Prop:existence_of_siltingfornonproper}.

		We remark that  the perfect derived category $\per(A)$ of $A$ in Condition \eqref{Conditionb} of Theorem \ref{Thm:1.1introduction} also  does not admit simple-minded collections (SMCs), by combining Theroem \ref{Thm:1.1introduction} and \cite[Proposition 5.12]{CJS} (see also \cite{KaY, KV88}). Hence  $\per(A)$ does not admit bounded $t$-structures with length heart either. However, following recent work in \cite{BCMPZ}, it seems  that $\per(A)$ admits bounded $t$-structures which then necessarily do not have length hearts. 
		Still one could ask the question whether in this case  the Bridgeland stability manifold of $\per(A)$ is empty. By \cite[Theorem 5.3]{HKK17} this would also imply that the moduli space of $S$ is empty.

		\subsection{A complete derived invariant of graded gentle algebras and partially wrapped Fukaya categories}
		The  first application of  Theorem~\ref{Thm:1.1introduction} is to give a complete invariant for triangle equivalences of partially wrapped Fukaya categories of graded surfaces with stops in terms of their geometric surface models. This confirms a conjecture in \cite{LP20}, where the invariant was first given. We do this by working with the derived categories of the  associated graded gentle algebras and their surface models.

		Classifying algebras up to derived equivalence is a difficult task. Even for small classes of algebras that are closed under derived equivalence, this is difficult and often additional information such as geometric invariants is necessary such as in the case of ungraded gentle algebras and Brauer graph algebras, see, for example, \cite{Am, AG, APS19, O19, OZ}.  In the case of ungraded gentle algebras, using the surface model of their derived category described in \cite{OPS18}, a complete derived invariant is described in \cite{APS19,O19}. In the case of graded gentle algebras while partial results exist \cite{HKK17, LP20}, the construction of a complete derived invariant is still open. More precisely, building on \cite{HKK17}, for general graded gentle algebras, a sufficient condition for two algebras to be derived equivalent is given in \cite{LP20} in terms of winding numbers of simple closed curves. However, the converse is still open and is conjectured to be true in \cite[Remark 3.19]{LP20}. 
		
		The second result of our paper is to confirm the conjecture by showing the following result.
		
		\begin{Thm}[Theorem \ref{Thm:derivedeq}]
			\label{Thm:derivedeqA}
			Let $A$ and $B$ be two homologically smooth graded gentle algebras with associated surface models  $(S_{A}, M_{A}, \zD_{A}, \eta_{A})$ and
			$(S_{B}, M_{B}, \zD_{B}, \eta_{B})$,  respectively. Then $A$ and $B$ are derived equivalent if and only if there exists an orientation preserving homeomorphism $\varphi: S_{A}\ra S_{B}$ such that $\varphi(M_{A})=M_{B}$  and $\varphi_{*}(\eta_{A})$ is homotopic to $\eta_{B}$.
		\end{Thm}

		Silting objects play a key role in the construction of derived equivalences in representation theory as mentioned above. Regardless of this, another motivation to study silting objects for graded gentle algebras is the following remarkable geometric property (see Proposition \ref{Prop:behaviorsiltingobjects} and refer to \cite{OPS18} for the ungraded case)  in terms of  surface models: \begin{center}{\it Silting objects give rise to admissible dissections.} \end{center} 
		Based on this, we establish a direct link between derived equivalences and the surface models of the corresponding graded gentle algebras, which plays an essential role in the proof of Theorem \ref{Thm:derivedeqA}.

		By \cite{HKK17} and \cite{LP20}, Theorem \ref{Thm:derivedeqA} can be translated into the language of partially wrapped Fukaya categories and in this way it also fully characterizes the triangle equivalence classes of partially wrapped Fukaya categories of graded surfaces with stops.

		\begin{Cor}
			\label{Cor:derivedeqA}
			Let  $(S_{1}, M_{1}, \eta_{1})$ and
			$(S_{2}, M_{2}, \eta_{2})$ be two graded surfaces with stops. Then the partially wrapped Fukaya categories $\cal W(S_{1}, M_{1}, \eta_{1})$ and $\cal W(S_{2}, M_{2}, \eta_{2})$ are triangle equivalent if and only if there exists an orientation preserving homeomorphism $\varphi: S_{1}\ra S_{2}$ such that $\varphi(M_{1})=M_{2}$  and $\varphi_{*}(\eta_{1})$ is homotopic to $\eta_{2}$.
		\end{Cor}
		
		Combining  Theorem~\ref{Thm:derivedeqA} and \cite[Corollary 1.10]{LP20}, we show that the geometric invariants of the surface models of graded gentle algebras considered in \cite{LP20} form a complete derived invariant. More precisely, we have the following result, where we refer to Subsection  \ref{subsection:combinatorialinvariants} for more details and definitions.
		\begin{Cor}[Corollary \ref{Cor:complete}]
			Let $A$ and $B$ be two homologically smooth graded gentle algebras with associated surface models  $(S_{A}, M_{A},\eta_{A},  \zD_{A})$ and $(S_{B}, M_{B}, \eta_{B}, \zD_{B})$,  respectively. Then $A$ and $B$ are derived equivalent if and only if $S_A$ and $S_B$ have the same number $b$ of boundary components and there exists a numbering of these boundary components such that for each $0\leq i < b$, one has
			\[\#(M_{A}\cap\partial_{i} S_{A})= \#(M_{B}\cap\partial_{i} S_{B}),\]
			\[w_{\eta_{A}}(\partial_{i} S_{A})=w_{\eta_{B}}(\partial_{i} S_{B}),  \]
			and in addition,
			\begin{itemize} 
				\item if $g(S_{A})=g(S_{B})=1$, then $\w{\mathcal A}(\eta_{A})=\w{\mathcal A}(\eta_{B})$; 
				\item if $g(S_{A})=g(S_{B}) >1$, then $\zs(\eta_{A})=\zs(\eta_{B})$ and $\mathcal A(\eta_{A})=\mathcal A (\eta_{B})$ whenever the latter two invariants are defined.
			\end{itemize}
		\end{Cor}
		Here,  $\partial_{0}S, \dots, \partial_{b-1}S$ are the connected boundary components of the surface $S$ and $w_{\eta}(\partial_{i}S)$ is the winding number of the boundary component $\partial_{i}S$. We note that the invariants  $\w{\mathcal A}, \sigma, \mathcal A$ defined in Subsection \ref{subsection:combinatorialinvariants} are expressed in terms of the winding numbers of boundary components and  some non-separating simple closed curves.
		
		\subsection{Pre-silting versus  partial silting for graded gentle algebras } 
		
		As a second application of Theorem \ref{Thm:1.1introduction}, 
		we provide a family of counterexamples to the natural question in representation theory whether each pre-silting object in the derived category of a finite-dimensional algebra is partial silting.  We do this by completely classifying for which graded gentle algebras every pre-silting object can be completed to a silting object, see Theorem \ref{Thm:partialsiltingmainresultA} below. 
		
		Recall that a pre-silting object $P$ in a triangulated category $\T$ is called \emph{partial silting} if there exists $Q$ such that $P\op Q$ is silting in $\mathcal T$. 
		Assume $\T$ admits silting objects. 
		The question whether any pre-silting object in $\T$ is partial silting has been of great interest, especially, for the case that $\T=\per(\Lambda)$ for some finite-dimensional algebra $\Lambda$. 
		A  Bongartz-type lemma which states that any $2$-term pre-silting object is partial silting \cite{Aihara13,B} gives a positive answer to this question in that case. 
		It has also been shown to be true  for  piecewise hereditary algebras \cite{AI, BY} and for silting-discrete algebras \cite{AM}, including local algebras \cite{AI}, and representation-finite symmetric algebras \cite{Aihara13, AbH}.

		Using Theorem \ref{Thm:1.1introduction} and the technique of silting reduction,  we give a complete answer to the question above for the derived category of 
		any homologically smooth and proper graded gentle algebra.

		\begin{Thm}[Theorem \ref{Thm:partialsiltingmainresult}]\label{Thm:partialsiltingmainresultA}
			Let $A$ be a homologically smooth and proper graded gentle algebra. Let $(S, M, \eta)$ be a surface model of $A$. Then the following are equivalent
			\begin{enumerate}[\rm(i)]
				\item Every pre-silting object in $\per(A)$ is partial silting.
				\item Either $S$ is of genus $0$,  or $S$ is of genus $1$ such that  
				\[ \w{\A}(\eta)\not=\gcd\{w_{\eta}(\partial_{0}S)+2, \dots, w_{\eta}(\partial_{b-1}S)+2\}. \]
			\end{enumerate}
		\end{Thm}

		In particular, Theorem~\ref{Thm:partialsiltingmainresultA} shows that for almost all homologically smooth and proper graded gentle algebras $A$, $\per(A)$ admits pre-silting objects which are not partial silting, the exceptions being the graded gentle algebras associated to surfaces of genus $0$ and certain graded surfaces with stops of genus $1$. 
		
		This immediately gives us the following result for finite dimensional (ungraded) gentle algebras.
		
		\begin{Cor}[{Corollary \ref{cor:finitedimensionalgentle}}]
			Let $A$ be a finite dimensional (ungraded)  gentle algebra of finite global dimension. If the genus of the  surface of $A$ is strictly greater than $1$ then there exists a pre-silting object in $\per(A)$ which is not partial silting.
		\end{Cor}
		As a consequence, it is easy to construct finite dimensional algebras whose derived categories admit pre-silting objects which are not partial silting. For example, in   Example \ref{Ex:notpartial1}, we give an infinite family  of such finite dimensional algebras. We note that this family includes the example previously  given  in 
		\cite{LZ23}.

		\section*{Conventions}\label{Conventions}
		In this paper, all algebras will be assumed to be over a base field $\Bbbk$. 
		Arrows in a quiver are composed from left to right as follows: for arrows $a$ and $b$ we write $ab$ for the path from the source of $a$ to the target of $b$. However, we adopt the convention that maps are composed from right to left, that is if $f: X \to Y$ and $g: Y \to Z$ then $gf : X \to Z$.  
		
		Let $\T$ be a triangulated $\Bbbk$-category and let $P\in \T$. We denote by $\thick_{\T}(P)$ the smallest full  triangulated  subcategory of $\T$  containing  $P$  and closed under direct summands. For a dg algebra $A$,  we denote by $\per(A):=\thick_{\D(A)}(A)$ the \emph{perfect derived category} of $A$. We denote by $\D^{\bb}(A)$ the full subcategory of $\D(A)$ consisting of the objects $M$ whose total cohomology is finite-dimensional (that is, $\sum_n \dim_{\Bbbk}\; \h^{n}(M)<\infty$).  If $\per (A)\subset \D^{\bb}(A)$, we say $A$ is \emph{proper}. Note that a proper graded algebra (a dg algebra with trivial differential) is finite dimensional. If $A\in \per (A\ot_{\Bbbk}A^{\rm op})$, we say $A$ is \emph{homologically smooth}.
		Note that if $A$ is homologically smooth and proper, we have $\per(A)=\D^{\rm b}(A)$. All (differential graded) modules considered here are right modules.

		\section{Graded gentle algebras and partially wrapped Fukaya categories}

		\subsection{Graded gentle algebras and their surface model}\label{subsection:marked surfaces}
		This subsection provides a brief overview of the surface model associated with a graded gentle algebra. We begin by revisiting the definition of a graded gentle algebra.
		
		\begin{Def}\label{definition:gentle algebras}
			A   \emph{graded gentle algebra}  is a graded algebra of the form  $A= \Bbbk Q/I$  where $Q=(Q_0,Q_1)$ is a finite graded  quiver (i.e.\ each arrow is assigned an integer) and $I$ is an ideal  of $\Bbbk Q$ generated by paths of length two such that 
			\begin{enumerate}[\rm(1)]
				\item each vertex in $Q_0$ is the source of at most two arrows and the target of at most two arrows;
				
				\item for each arrow $\za$ in $Q_1$, there is at most one arrow $\zb$ such that  $0 \neq \za\zb\in I$ and at most one arrow $\zg$ such that  $0 \neq \zg\za\in I$. Furthermore,   there is at most one arrow $\zb'$ such that $\za\zb'\notin I$ and at most one arrow $\zg'$ such that $\zg'\za\notin I$.
				
			\end{enumerate}
		\end{Def}
		
		Throughout this paper, all graded gentle algebras are supposed to be homologically smooth (i.e.\ no cycles with full relations) unless stated otherwise.
		
		\begin{Def}\label{definition: marked surface}
			A \emph{graded surface with stops} is a triple $(S,M,  \eta)$ where
			\begin{enumerate}[\rm(1)]
				\item $S$ is a compact oriented surface with nonempty boundary  $\partial S=\sqcup_{i=0}^{b-1}\partial_i S$ where each $\partial_i S$ is a connected boundary component;
				\item 
				$M$ is a finite set of points (called {\it stops}) on $\partial S$ such that $M \neq \emptyset;$
				\item $\eta$ is a line field on $S$, i.e.\ a section of the projectivized tangent bundle $\mathbb{P}(T S)$. 
			\end{enumerate}
		\end{Def}
		
		\begin{Rem}
			If $\partial_iS \cap M = \emptyset$ then $\partial_iS$ is called a {\it boundary component without stops} or, for short, a {\it non-stopped boundary component}. Note that in  \cite{APS19, OPS18}  this corresponds to a green puncture. We call a boundary component with stops a {\it stopped boundary component}. 
			
			Since by assumption $M$ is a finite set there are no fully stopped boundary components in $S$, which correspond to red punctures in \cite{APS19, OPS18}.
		\end{Rem}
		
		\begin{Convention}
			
			Throughout the paper, we fix the orientation of a surface such that,  when drawn locally in the plane, the interior lies to the left as one follows the boundary.
			
			By a curve $\gamma$ in a graded surface with stops we mean a continuous map $\gamma \colon [0,1] \to S\backslash M$ such that $\gamma(t) \notin \partial S$ for $0< t < 1$. It is called {\it closed} if $\gamma(0)= \gamma(1)\notin  \partial S$ and {\it open} if $\gamma(0), \gamma(1)$ are distinct points  in $\partial S$.  All curves are considered up to homotopy. Furthermore, all  intersections of curves are  to be transversal.
		\end{Convention}

		\begin{Def}\label{definition:arcs}
			Let $(S,M, \eta)$ be a graded surface with stops.
			\begin{enumerate}[\rm(1)]
				\item An \emph{arc} is an embedded open curve (with distinct endpoints  in $\partial S \backslash M$), such that it is not isotopic to a boundary interval without stops. 
				
				\item An \emph{admissible dissection} $\zD$ on $S$ is a collection of  arcs  on $S$, which are pairwise non-intersecting in $S$, such that these arcs cut $S$ into polygons each of which contains exactly one stop, see Figure \ref{Fig:cutpolygon}.  We note that $\eta$ imposes an integer grading on the boundary segments and we may and will assume that  $\eta$ is transverse to each of the arcs in $\Delta$, compare Remark \ref{Rem:uptohomotopy} (2).
				\item An {\it admissible collection} is a collection of arcs on $S$ which can be completed into an admissible dissection. 
			\end{enumerate}
		\end{Def}
		
		\begin{figure}[ht]
			\begin{tikzpicture}[scale=0.4, transform shape]
				\begin{scope}[decoration={markings, mark=at position 0.5 with {\arrow{>}}}] 
					\foreach \x in {0,1,...,7}
					{
						\draw[thick, green] (\x*40-70:5) to (\x*40-30:5);
					}
					\foreach \x in {1,..., 7}
					{\begin{scope}[shift={(\x*40-70:5)}, rotate=\x*40+20]
							\draw[white, line width=5pt] (0,0) to (20:.5);
							\draw[white, line width=5pt] (0,0) to (160:.5);
							\draw[bend left, postaction={decorate}] (160:.5) to (20:.5);
						\end{scope}
					}
					\foreach \x in {1,..., 8}
					{\begin{scope}[shift={(\x*40-70:5)}, rotate=\x*40+20]
							\coordinate (a\x) at (160:1.5);
						\end{scope}
					}
					\draw[] (-70:5) to [curve through = {.. (-90:4)..}] (-110:5);
					\node[fill,circle, red, inner sep=0pt, minimum size=6pt] at (-90:4){};
				\end{scope}
			\end{tikzpicture}
			\caption{The polygon cut out by (green) arcs, where the non-coloured edges with orientation are boundary segments in $\partial S$. It contains one stop on  one of the boundary segments. }
			\label{Fig:cutpolygon}
		\end{figure}

		\begin{Rem}\label{Rem:admissiblecollection}
			Note that a collection of pairwise non-intersecting arcs forms an admissible collection if and only if these arcs do not enclose a subsurface $S'$ without stops on its boundary $\partial S'$, see \cite[Definition 1.9]{APS19}. 
			It is shown in \cite[Proposition 3.12]{APS19} that each admissible dissection of $S$ contains exactly $\# M+b+2g-2$ arcs, where $b$ is the number of boundary components and $g$ is the genus of $S$. 
		\end{Rem}

		\begin{Def}\label{definition: dissection}
			Let $(S,M, \eta)$ be a graded surface with stops.
			Let $\zD$ be an admissible dissection. We construct a graded gentle algebra  $A(\zD)=\Bbbk Q(\zD)/I(\zD)$ as follows.
			\begin{enumerate}[\rm(1)]
				\item The vertices of the quiver $Q(\zD)$ are in one-to-one correspondence with the arcs in $\zD$;
				\item  For two arcs $\ell_i$ and $\ell_j$ in $\zD$, an arrow $\za$ from $\ell_i$ to $\ell_j$ is given by a {\it minimal} embedded interval $c\subset (\partial S\backslash M)$ starting at  an endpoint of $\ell_i$ and ending at an endpoint of $\ell_j$, following  the orientation of the boundary. The degree of $\alpha$ is given by $-w_\eta(c)$. By minimal we mean that the interior of $c$ does not cross any arcs in $\Delta$. 
				\item  The ideal $I(\zD)$ is generated by the length two paths of the form $\za\zb$  for all minimal boundary segments  $\za:\ell_i\rightarrow\ell_j$ and  $\zb:\ell_j\rightarrow\ell_k$  intersecting $\ell_j$ at the different ends of $\ell_j$. See Figure \ref{figure:relations} for the general configuration to have relations. 
			\end{enumerate} 
		\end{Def} 
		
		Note that $A(\Delta)$ is homologically smooth since $M$ is a finite set (i.e.\ there are no fully stopped boundary components on $S$). It is proper if and only if each $\partial_i S\cap M$ is nonempty (i.e.\ there are no non-stopped boundary components), see e.g.\ \cite{OPS18, APS19, LP20, HKK17}.  
		
		\begin{figure}
			\centering
			\begin{tikzpicture}
				\begin{scope}[shift={(-6,0)}]
					\draw[thick] (-1,0) to [curve through = {.. (-0.5, 0.25)..(0,.35)..(0.5,0.25)..}] (1,0);
					\draw [green, thick] (-.5,.25) to (120:1.5);
					\draw [green, thick] (.5,.25) ..controls (1.5,.5).. (2.5,0.25);
					\draw [thick, decoration={markings, mark=at position 0.1 with {\arrow{>}}}, postaction={decorate}] (0,.35) to (0.01,.35) ;
					\node at (140:1) {\tiny$\ell_i$};
					\node at (1.5,0.6) {\tiny$\ell_j$};
					\node at (0, .6) {\tiny$\za$};
					\begin{scope}[shift={(3,0)}]
						\draw[thick] (-1,0) to [curve through = {.. (-0.5, 0.25)..(0,.35).. (0.5,0.25)..}] (1,0);\draw [green, thick] (0.5,.25) to (60:1.5);
						\draw [thick, decoration={markings, mark=at position 0.1 with {\arrow{>}}}, postaction={decorate}] (0,.35) to (0.01,.35) ;
						\node at (40:1) {\tiny$\ell_k$};
						\node at (0,.6) {\tiny$\zb$};
					\end{scope}
				\end{scope}
			\end{tikzpicture}
			\caption{The general local picture for relations $\alpha \beta$, where the two boundary components are not necessarily pairwise distinct. 
			}\label{figure:relations}
		\end{figure}

		\begin{Def}
			Let $A$ be a graded gentle algebra. We say that $(S, M, \eta, \zD)$ is a {\it surface model} of $A$ if $A(\zD)$ is isomorphic to $A$ as graded algebras. 
			
			Let $A$ and $A'$ be two graded gentle algebras. Let $(S, M, \eta, \zD)$ (resp.\ $(S', M', \eta', \zD')$) be a  surface model of $A$ (resp. $A'$). We say that the two surface models are {\it isomorphic} if there is an orientation preserving homeomorphism $\varphi\colon S \to S'$ such that $\varphi(M) = M'$ and $\varphi_*(\eta)$ is homotopic to $\eta'$. (Note that here we do not require any relation between $\varphi(\Delta)$ is and $\Delta'$, compare with Proposition \ref{Prop:models_same_algebra} below.)
		\end{Def}
		
		\begin{Rem}
			By \cite{LP20, OPS18} any graded gentle algebra admits at least one surface model. In particular, the surface may be obtained by thickening the ribbon graph associated to the (ungraded) gentle algebra, see \cite[Subsection 3.1]{Sc}. 
		\end{Rem}
		
		\subsection{Partially wrapped Fukaya categories} We recall the partially wrapped Fukaya category associated to a graded surface with stops as defined in \cite{HKK17}. 
		
		Let $(S, M, \eta)$ be a graded surface with stops. The partially wrapped Fukaya category ${\cal W} (S,M, \eta)$, which can be understood in terms of the global sections of a constructible cosheaf on a ribbon graph associated to $S$, is a triangulated category whose indecomposable objects are described by the isotopy classes of admissible curves (with indecomposable local system), see \cite[Theorem 4.3]{HKK17}. It is also shown in \cite{HKK17} that any admissible dissection $\Delta$ is a formal generator of $\cal W(S, M, \eta)$, so that $\cal W(S, M, \eta)$ is triangle equivalent to the perfect derived category $\per(A(\Delta))$ of the graded gentle algebra $A(\Delta)$, see Definition \ref{definition: dissection}. As a result, we obtain a geometric description of objects in $\per(A(\Delta))$. 
		
		In particular, the following holds, which can be used to construct derived equivalences of graded gentle algebras.

		\begin{Thm}[{\cite[Proposition 3.3]{HKK17}}]\label{thm:hkkequivalence}
			Let $\zD_{1}$ and $\zD_{2}$ be two admissible dissections of a graded surface with stops $(S,M,\eta)$.  Then $\per(A(\zD_{1}))$ is triangle equivalent to $\per(A(\zD_{2}))$. 
		\end{Thm}

		Now we give a characterization of a special type of admissible curves of  $S$, which will play an important role in the proofs of the main results in this paper. 
		\begin{Prop}\label{Prop:kxcurve}
			Let $(S, M, \eta)$ be a graded surface with stops. Assume that $\gamma$ is an admissible curve in $(S, M, \eta)$. Denote by $X_\gamma$ the corresponding indecomposable object   in $\mathcal W(S, M, \eta)$. Then $\End(X_\gamma) \simeq \Bbbk [x]$ if and only if $\gamma$ is an arc (particularly, no self-intersections) connecting a stopped boundary component to a non-stopped boundary component $\partial_iS$. In this case, $|x|=-w_\eta(\partial_iS).$
		\end{Prop}
		\begin{proof}
			Let us first prove the \lq if' part. We extend the arc $\gamma$ into an admissible  dissection $\Delta$ and denote by $e_i$ the idempotent of $A(\Delta)$ corresponding to $\gamma$. Then we have $\End(X_\gamma) \simeq e_i A(\Delta)e_i \simeq \Bbbk[x] $.
			
			Let us prove the \lq only if' part. Since  $\End(X_\gamma) \simeq \Bbbk [x]$ it follows from Remark \ref{remark:appendix} that $\gamma$ cannot be a closed curve and from \cite[Theorem 4.11]{QZZ22} that $\gamma$ has no self-intersections.
			In other words, $\gamma$ can only be a curve without self-intersections connecting two boundary components. There are three cases: If $\gamma$ connects two stopped boundary components then $\End(X_\gamma) \simeq \Bbbk$; If $\gamma$ connects two non-stopped boundary components then $\End(X_\gamma) \simeq \Bbbk [x, y]/(xy, yx).$ If $\gamma$ connects a stopped boundary component with a non-stopped boundary component then $\End(X_\gamma) \simeq \Bbbk [x]$. 
			
			In this case, by Definition \ref{definition: dissection}, we have $|x|=-w_\eta(\partial_iS)$.
		\end{proof}
		
		\begin{Cor}\label{Cor:kxcurve}
			Let $(S,M,\eta, \zD)$ and  $(S',M',\eta', \zD')$ be two graded surfaces with stops. Assume that there is a triangle equivalence $F:\mathcal W(S,M,\eta) \ra \mathcal W(S', M', \eta')$. Let $\gamma$ be an arc  in $S$ connecting a stopped boundary segment to a non-stopped boundary component, which corresponds to an indecomposable object $X_\gamma$ in $\mathcal W(S,M,\eta)$. Then the curve corresponding to $F(X_\gamma)$ is an arc of the same type (as $\gamma$)  in $(S',M',\eta')$.
		\end{Cor}

		\begin{proof}
			By Proposition \ref{Prop:kxcurve}, we have that  $\End(X_\gamma)\simeq \Bbbk[x]$. Since $F$ is a triangle equivalence it follows that $\End(F(X_\gamma)) \simeq \Bbbk [x]$. Then again by Proposition \ref{Prop:kxcurve} the curve corresponding to $F(X_\gamma)$ also connects a stopped boundary component to a non-stopped boundary component in $S'$.
		\end{proof}

		\section{Geometric invariants}\label{Preliminaries}
		In this section we recall the surface models and geometric invariants of graded gentle algebras based on \cite{HKK17,LP20, OPS18,APS19}.
		
		\subsection{Combinatorial geometric invariants of a line field} \label{subsection:combinatorialinvariants}
		In this subsection, we recall from \cite{LP20} the combinatorial invariants of a line field in a surface. For details see \cite[Section 1]{LP20}.

		Let $S$ be a compact oriented surface with boundary  $\partial S=\sqcup_{j=0}^{b-1}\partial_j S$. Let $\bar S$ be the closed surface obtained from $S$ by filling in each boundary component. The Poincar\'e duality for $\bar S$ induces a nondegenerate skew-symmetric pairing $\langle\cdot,\cdot\rangle$ on $\mathrm{H}_1(\bar S)\simeq \mathbb Z^{2g}$,  where $g=g(S)$ is the genus of $S$.

		We fix a collection of simple closed curves $s_1, \dotsc, s_g, t_1, \dotsc, t_g$ on $S$ such that their homotopy classes form a  symplectic  basis of  $\mathrm{H}_1(\bar S)$ such that $\langle s_i, s_j \rangle = 0 = \langle t_i, t_j \rangle$ and $\langle s_i, t_j \rangle = \delta_{i,j}$ for $1\leq i, j \leq g$. 
		For any line field $\eta$,  the following sequence of integers 
		\begin{align}\label{align:allthewindingnumbers}
			W_\eta:= \{w_{\eta}(\partial_0 S),\dotsc, w_{\eta}(\partial_{b-1} S), w_{\eta} (s_1), \dotsc, w_{\eta}(s_g), w_{\eta}(t_1),\dotsc, w_{\eta}(t_g)\}
		\end{align}
		determines the orbit of $\eta$ under the action of the mapping class group of $S$ by \cite{LP20}.  
		
		For a line field $\eta$, we  define the $\mathbb Z/2\mathbb Z$-valued invariant $\sigma(\eta)$,  which indicates whether the line field comes from a vector field (i.e.\ $\sigma(\eta) = 0$) 
		\begin{align}\label{align:sigmainvariant}
			\sigma(\eta) = \begin{cases}
				0 & \text{if  all the integers in $W_\eta$ are even,}\\
				1 & \text{otherwise.}
			\end{cases}
		\end{align}
		When $g=1$, we also define
		\begin{align}\label{align:wA}
			\w{\mathcal A}(\eta):=\gcd( w_{\eta}(s_1), w_{\eta}(t_1), w_{\eta}(\partial_{0}S)+2, \dots, w_{\eta}(\partial_{b-1}S)+2). 
		\end{align}
		Note that for integers $n_{1}, \dots, n_{k}$, the  $\gcd(n_1, \dotsc, n_k)$ is defined as the greatest common divisor of the absolute values of the $n_{i}$ and it is zero if and only if $n_i = 0$ for each $1\leq i \leq k$.
		
		\begin{Thm}[{\cite[Corollary 1.10]{LP20}}]
			\label{Prop:homotopic_linefield}
			Let $(S_{1}, M_{1}, \eta_{1})$ and $(S_{2}, M_{2}, \eta_{2})$ be two graded surfaces with stops which have the same number $b$ of boundary components. Then there exists an orientation preserving homeomorphism $\varphi: S_{1}\ra S_{2}$ such that $\varphi(M_{1})=M_{2}$ and $\varphi_{*}(\eta_{1})$ is homotopic to $\eta_{2}$ if and only if  
			there exists  a numbering of the boundary components of $S_{1}$ and $S_{2}$ such that for each boundary component one has
			\[\#(M_{1}\cap\partial_{j} S_{1})= \#(M_{2}\cap\partial_{j} S_{2}),\]
			\[w_{\eta_{1}}(\partial_{j} S_{1})=w_{\eta_{2}}(\partial_{j} S_{2}),  \]
			and in addition,
			\begin{itemize}
				\item if $g(S_{1})=g(S_{2})=1$, then $\w{\mathcal A}(\eta_{1})=\w{\mathcal A}(\eta_{2})$;
				\item if $g(S_{1})=g(S_{2})>1$, then one of the following three conditions occurs
				\begin{enumerate}[\rm(a)]
					\item $\sigma(\eta_1) = \sigma(\eta_2) =1$;
					\item $\sigma(\eta_1) = \sigma(\eta_2) = 0$ and there exists $0\leq i < b$ such that $w_{\eta_1}(\partial_i S) = w_{\eta_2}(\partial_i S) = 0 \ \mod 4$;
					\item $\sigma(\eta_1) = \sigma(\eta_2) = 0$, for each  $0\leq i < b$ we have $w_{\eta_1}(\partial_i S) = w_{\eta_2}(\partial_i S) = 2 \ \mod 4$, and the Arf invariants for $\eta_1$ and $\eta_2$ coincide, that is
					\[
					\mathcal A(\eta_1) = \mathcal A(\eta_2)
					\]
					where the Arf invariant is defined as follows 
					\begin{align}\label{align:Arfinvariant}
						\mathcal A(\eta_j) := \sum_{i=1}^g \left (\frac{1}{2} w_{\eta_j}(s_i) +1\right) \left(\frac{1}{2} w_{\eta_j}(t_i) +1\right) \ \mod 2,
					\end{align}
					for $j=1,2$.
				\end{enumerate}
			\end{itemize}
		\end{Thm}

		\begin{Rem}\label{Remark:linefieldinvariant}
			\begin{enumerate}
				\item The invariants $\sigma, \w{\mathcal A}$ and $\mathcal A$ are independent of the choice of the symplectic basis of $\mathrm H_1(\bar S)$ given by the simple closed curves $s_i, t_i$,  and $\w{\mathcal A}$ can also be calculated by 
				\[ \w{\mathcal A}(\eta)=\gcd(\{w_{\eta}(\zg)\mid \zg \mbox{ non-separating simple closed curve}\}),\]
				see  \cite[Lemma 2.6]{K} and \cite{LP20}. The invariant  $\w{\mathcal A}(\eta)$  for the torus with one boundary component already appears in earlier in \cite{Am} where  examples of gentle algebras with the same AG-invariant \cite{AAG08} that are not derived equivalent are considered and where the invariant $\w{\mathcal A}(\eta)$ is constructed to give a complete derived invariant.  
				\item Note that if the sequences $W_{\eta_1}$ and $W_{\eta_2}$ in \eqref{align:allthewindingnumbers} associated to two line fields $\eta_1$ and $\eta_2$ in a surface $S$ coincide,  then by Theorem \ref{Prop:homotopic_linefield} above $\eta_1$ and $\eta_2$ are homotopic up to the action of the mapping class group. The converse does not generally hold. But if $S$ is of genus $0$, then $\eta_1$ is homotopic to $\eta_2$ if and only if the two sequences $W_{\eta_1}$ and $W_{\eta_2}$ coincide up to a permutation.
				\item By the Poincar\'e--Hopf index theorem (see \cite{Hopf}), we have 
				\begin{align}\label{algin:PHindex}
					4-4g = \sum_{i=0}^{b-1}\left(w_\eta(\partial_iS)+2\right),
				\end{align}
				which is the only constraint for the sequence $W_\eta$, see \cite[Page 119]{LP20}. 
				That is, any sequence of integers as in \eqref{align:allthewindingnumbers} satisfying \eqref{algin:PHindex}  arises  from some line field  on $S$.
			\end{enumerate}
		\end{Rem}
		
		\subsection{The calculation of winding numbers via graded quivers}
		In this subsection, we express the winding number of any simple closed (non-contractible) curve in a surface model of a graded gentle algebra, in terms of the grading of the algebra.  As a result, we show in Proposition \ref{Prop:models_same_algebra} that any graded gentle algebra admits a unique surface model, up to isomorphism, which is essentially due to \cite{LP20, OPS18}.

		Let $A$ be a graded gentle algebra. Let $(S, M, \eta,\zD)$ be any surface model of $A$. Then the winding number $w_{\eta}(\gamma)$ of any simple closed curve $\gamma$ (not necessarily non-separating) can be calculated using the gradings of the arrows in $A$. Recall that the graded  admissible dissection $\Delta$ cuts the surface into polygons
		\begin{align} \label{align:cuttingpolygons}
			S\setminus \Delta = \cup_{f} P_f
		\end{align}
		where  each $P_f$  contains exactly one stop, see Figure \ref{figure:simpleclosedcurves}.  Note that these polygons bijectively correspond to the {\it forbidden threads} as defined in  \cite{AAG08}, see also \cite{LP20}.  
		
		Let $\gamma$ be a simple closed curve in $S$. We choose a representative of $\gamma$ so that it intersects with the  arcs in $\zD$ in minimal position. Let $\ell_1, \dotsc, \ell_k$ be those arcs in $\Delta$ crossed by $\gamma$ in order. Then they cut $\gamma$ into segments $\gamma_1, \dotsc, \gamma_k$, where $\gamma_i$ lies between $\ell_i$ and $\ell_{i+1}$. Here, the indices are considered mod $k$. We have (see e.g. \cite[Proposition 1.6]{APS19})
		\[
		w_\eta(\gamma) = \sum_{i=1}^k w_\eta(\gamma_i).
		\]
		Each segment $\gamma_i$ lies in a polygon $P_f$ and cuts $P_f$ into two subpolygons. Denote by $\alpha_1, \dotsc, \alpha_m$  the arrows in the subpolygon of $P_f$ that does not contain a stop.  Then we have 
		\begin{align}\label{align:windingnumbersegments}
			w_\eta(\gamma_i) = \begin{cases}
				1-m +\sum_{j=1}^m |\alpha_j|  & \text{if the stop of $P_f$ lies on the right of $\gamma_i$},\\
				-1+m - \sum_{j=1}^{m}|\alpha_j| &  \text{otherwise}.
			\end{cases}
		\end{align}
		Here, the assumption for the first (resp.\ second) equality means that the direction of $\gamma$ agrees with (resp.\ differs from) the orientation induced by the subpolygon which does not contain a stop. In particular, we have $w_\eta(\gamma^{-1}) = - w_\eta(\gamma)$ where $\gamma^{-1}$ is the inverse of $\gamma$ given by $\gamma^{-1}(t) = \gamma(1-t)$, for $t \in [0,1]$. 
		
		Let us look at the example in Figure \ref{figure:simpleclosedcurves}. According to  \eqref{align:windingnumbersegments} the winding numbers of the segments $\gamma_{i-1}, \gamma_i, \gamma_{i+1}$ can be expressed as follows 
		\begin{align*}
			w_{\eta}(\gamma_{i-1}) &=-2+|\za_1'|+|\za_2'|+|\za_3'|,\\
			w_{\eta}(\gamma_i) &= 3 - |\za_1|-|\za_2|-|\za_3| - |\za_4|, \\
			w_{\eta}(\gamma_{i+1}) &= -3 + |\za_1''| +|\za_2''| +|\za_3''|+ |\za_4''|.
		\end{align*}
		
		\begin{figure}[ht]
			\centering
			\begin{tikzpicture}
				\begin{scope}[scale=0.6]
					\foreach \x in {0,1,2,3}
					\draw [green, thick] (-9+6*\x, -1.5) to (-9+6*\x,1);
					\draw [green, thick] (-9,1) to (-6,3) to (-3,1) to (-1,3);
					\draw [green, thick] (1,3) to (3,1) to (5,3) to (7, 3) to (9, 1);
					\draw [green, thick] (-5,-3) to (-3,-1.5) to (-1,-3) to (1,-3) to (3,-1.5) to (5,-3);
					\draw [green, thick] (7, -3) to (9, -1.5);
					\draw [green, thick] (-9,-1.5) to (-7,-3);
					\draw [green, thick] (-1, 3) to (1, 3);
					\begin{scope}[shift={(-1,3)}]
						\draw [fill=white]  (0,0) circle (.4);
						\draw[white, line width=5pt] (40:.4) arc[start angle = 40, end angle =180, radius = .4];
						\draw[decoration={markings, mark=at position 0.55 with {\arrow{>}}}, postaction={decorate}](-80:.4) to (-81:.401);
					\end{scope}
					\begin{scope}[shift={(1,3)}]
						\draw [fill=white]  (0,0) circle (.4);
						\draw[white, line width=5pt] (10:.4) arc[start angle = 10, end angle =150, radius = .4];
						\node[fill,circle, red, inner sep=0pt, minimum size=3pt] at (-100:.4){};
					\end{scope}
					\begin{scope}[shift={(-3,1)}]
						\draw [fill=white]  (0,0) circle (.4);
						\draw[white, line width=5pt] (120:.4) arc[start angle = 120, end angle = 60, radius = .4];
						\draw[decoration={markings, mark=at position 0.55 with {\arrow{>}}}, postaction={decorate}](180:.4) to (178:.401);
					\end{scope}
					\node at (-3.8,.7) {\tiny $\alpha'_1$};
					\begin{scope}[shift={(-3,-1.5)}]
						\draw [fill=white]  (0,0) circle (.4);
						\draw[white, line width=5pt] (-60:.4) arc[start angle = -60, end angle = -120, radius = .4];
						\draw[decoration={markings, mark=at position 0.55 with {\arrow{>}}}, postaction={decorate}](20:.4) to (19:.401);
					\end{scope}
					\node at (-2.15,-1.35) {\tiny $\alpha_1$};
					\begin{scope}[shift={(3,1)}]
						\draw [fill=white]  (0,0) circle (.4);
						\draw[white, line width=5pt] (120:.4) arc[start angle = 120, end angle = 60, radius = .4];
						\draw[decoration={markings, mark=at position 0.55 with {\arrow{>}}}, postaction={decorate}](-19:.4) to (-20:.401);
					\end{scope}
					\node at (3.9,.8) {\tiny$\alpha''_4$};
					\begin{scope}[shift={(3,-1.5)}]
						\draw [fill=white]  (0,0) circle (.4);
						\draw[white, line width=5pt] (-60:.4) arc[start angle = -60, end angle = -120, radius = .4];
						\draw[decoration={markings, mark=at position 0.55 with {\arrow{>}}}, postaction={decorate}](180:.4) to (178:.401);
					\end{scope}
					\node at (2.2,-1.3) {\tiny $\alpha_4$};
					\draw [] (-7.3,-3.1) to [curve through = {(-7, -3).. (-6,-2.8)..(-5, -3)}] (-4.7,-3.1);
					\node[fill,circle, red, inner sep=0pt, minimum size=3pt] at (-6,-2.8){};
					\draw [] (4.7,-3.1) to [curve through = {(5, -3).. (6,-2.8)..(7, -3)}] (7.3,-3.1);
					\node[fill,circle, red, inner sep=0pt, minimum size=3pt] at (6,-2.8){};
					\draw[decoration={markings, mark=at position 0.55 with {\arrow{>}}}, postaction={decorate}](-9.5,0) to (-3,0);
					\draw [decoration={markings, mark=at position 0.45 with {\arrow{>}}}, postaction={decorate}] (-10.5,0) to (-9.45,0);
					\draw[decoration={markings, mark=at position 0.5 with {\arrow{>}}}, postaction={decorate}] (-3,0) to (3,0);
					\draw[decoration={markings, mark=at position 0.45 with {\arrow{>}}}, postaction={decorate}] (3,0) to (9.5,0);
					\draw [decoration={markings, mark=at position 0.45 with {\arrow{>}}}, postaction={decorate}] (9.35,0) to (10.5,0);
					\node at (-6,.4) {\tiny $\gamma_{i-1}$};
					\node at (0,.4) {\tiny $\gamma_{i}$};
					\node at (6,.4) {\tiny $\gamma_{i+1}$};
					\draw[very thick, white](-8.5,4/3) -- (-9,1);
					\draw[very thick, white](-9,.5) -- (-9,1);
					\draw[bend left, decoration={markings, mark=at position 0.5 with {\arrow{>}}}, postaction={decorate}] (-8.5, 4/3) to (-9, .5);
					\begin{scope}[shift={(-9,-1.5)}]
						\draw [fill=white]  (0,0) circle (.4);
						\draw[white, line width=5pt] (120:.4) arc[start angle = 120, end angle =290, radius = .4];
					\end{scope}
					\begin{scope}[shift={(9,-1.5)}]
						\draw [fill=white]  (0,0) circle (.4);
						\draw[white, line width=5pt] (60:.4) arc[start angle = 60, end angle =-90, radius = .4];
					\end{scope}
					
					\draw [fill=white, white]  (-6,3) circle (.3);
					\draw[very thick, white](-5.5,8/3) -- (-6,3);
					\draw[very thick, white](-6.5,8/3) -- (-6,3);
					\node at (-8.2,.7) {\tiny$\alpha'_3$};
					\draw[bend left, decoration={markings, mark=at position 0.45 with {\arrow{>}}}, postaction={decorate}] (-5.5, 8/3) to (-6.5, 8/3);
					\node at (-6,2.1) {\tiny $\alpha'_2$};
					\draw[very thick, white](-1.5,-2.62) -- (-1,-3);
					\draw[very thick, white](-.5,-3) -- (-1,-3);
					\draw[bend left, decoration={markings, mark=at position 0.45 with {\arrow{>}}}, postaction={decorate}] (-1.5,-2.62 ) to (-.5,-3);
					\node at (-.88,-2.3) {\tiny $\alpha_2$};
					\draw[very thick, white](.5,-3) -- (1,-3);
					\draw[very thick, white](1.5,-2.625) -- (1,-3);
					\draw[bend left, decoration={markings, mark=at position 0.45 with {\arrow{>}}}, postaction={decorate}] (.5,-3) to (1.5,-2.625);
					\node at (.75,-2.35) {\tiny $\alpha_3$};
					\draw[very thick, white](5.5,3) -- (5,3);
					\draw[very thick, white](4.5,2.5) -- (5,3);
					\draw[bend left, decoration={markings, mark=at position 0.45 with {\arrow{>}}}, postaction={decorate}] (5.5, 3) to (4.5, 2.5);
					\node at (5.3,2.3) {\tiny$\alpha''_3$};
					\draw[very thick, white](6.5,3) -- (7,3);
					\draw[very thick, white](7.5,2.5) -- (7,3);
					\draw[bend left, decoration={markings, mark=at position 0.45 with {\arrow{>}}}, postaction={decorate}] (7.5, 2.5) to (6.5, 3);
					\node at (6.5,2.4) {\tiny$\alpha''_2$};
					\draw[very thick, white](9,.5) -- (9,1);
					\draw[very thick, white](8.5,1.5) -- (9,1);
					\draw[bend left, decoration={markings, mark=at position 0.45 with {\arrow{>}}}, postaction={decorate}] (9.05,.48) to (8.51, 1.55);
					\node at (8.2,.8) {\tiny $\alpha''_1$};
					\node at (-8.3,-.5) {\tiny$\ell_{i-1}$};
					\node at (-2.6,-.5) {\tiny$\ell_{i}$};
					\node at (3.7,-.5) {\tiny$\ell_{i+1}$};
					\node at (9.7,-.5) {\tiny$\ell_{i+2}$};
				\end{scope}
			\end{tikzpicture}
			\caption{A local picture of a  simple closed curve $\gamma$ crossing through polygons of an admissible dissection $\zD$. 
			}\label{figure:simpleclosedcurves}
		\end{figure}
		
		\begin{Rem}\label{rem:nonproper}
			
			Let $A=\Bbbk Q/I$ be a  graded gentle algebra and let $(S, M, \eta, \zD)$ be a graded surface model of $A$. We denote by $C(A)$ the set of equivalence classes (with respect to cyclic permutation) of  cyclic paths $\za_{1}\za_{2}\cdots\za_{m}$ in $Q$ such that  $\alpha_1,\dotsc, \alpha_m$  are distinct arrows ($t(\alpha_m) = s(\alpha_1)$) and $\za_{i}\za_{i+1}\not\in I$ for all $1\le i\le m$, where the indices are taken modulo $m$.  For any cyclic path $p=\za_{1}\cdots\za_{m}$ in $C(A)$, we denote  $|p|:=\sum_{i=1}^{m} |\za_{i}|$.

			Note that $C(A)$  is in bijection with the set of non-stopped boundary components of $S$ and therefore it is empty if and only if $A$ is proper.  Moreover, the winding number of the non-stopped boundary component  corresponding to $p\in C(A)$ is $-|p|$.
		\end{Rem}

		\begin{Prop} \label{Prop:models_same_algebra}
			Let $A=\Bbbk Q/I$ be a graded gentle algebra. Let $(S_{1}, M_{1}, \eta_{1},  \zD_{1})$ and $(S_{2}, M_{2}, \eta_{2}, \zD_{2})$ be two surface models of $A$. Then there exists an orientation preserving homeomorphism $\varphi: S_{1}\xrightarrow{\simeq} S_{2}$ such that $\varphi(M_{1})=M_{2}$, $\varphi(\zD_{1})=\zD_{2}$,  and $\varphi_{*}(\eta_{1})$ is homotopic to $\eta_{2}$.
		\end{Prop}
		
		\begin{proof}
			Since by assumption $A(\zD_1) \simeq A \simeq A(\zD_2)$ as graded algebras, it follows that the ribbon graphs $\Gamma_1$ and $\Gamma_2$ associated to $\Delta_1$ and $\Delta_2$ are isomorphic (see \cite{Sc}) and  thus there is an orientation preserving homeomorphism $\varphi: S_1 \to S_2$ such that $\varphi(M_1) = M_2$ and $\varphi(\Delta_1) = \Delta_2$. (Note that the vertices in $\Gamma_i$, for $i =1,2$, correspond to  the arcs in $\zD_i$.)
			
			It remains to show that $\varphi_*(\eta_1)$ is homotopic to $\eta_2$. For this, by \cite[Proposition 3.4 (2)]{APS19} (see also \cite{C72}), it suffices to show that for any simple closed curve $\gamma$ in $S_2$ we have  $w_{\eta_2}(\gamma) = w_{\varphi_*(\eta_1)} ( \gamma)$.  Note that $w_{\varphi_*(\eta_1)} (\gamma) = w_{\eta_1}(\varphi^{-1}(\gamma))$ as winding numbers are preserved under homeomorphisms. Therefore, we only need to show that for any simple closed curve $\gamma$ in $S_2$ we have 
			\begin{align}\label{align:toshow}
				w_{\eta_2}(\gamma) = w_{\eta_1}(\varphi^{-1}(\gamma)).
			\end{align}
			
			To show this we may assume that $\gamma$ (after choosing a representative) intersects with the arcs in $\Delta_{i}$,  for $i=1,2$,  in minimal positions. Assume that $\gamma$ crosses  the polygons $P_1, \dotsc, P_k$ and is cut into segments $\gamma_1, \dotsc, \gamma_k$. Then $w_{\eta_2}(\gamma) = \sum_{i=1}^k w_{\eta_2}(\gamma_i)$. Since $\varphi$ is a homeomorphism and sends $\Delta_1$ to $\Delta_2$, it follows that $\varphi^{-1}(P_1), \dotsc, \varphi^{-1}(P_k)$ are the polygons of the dissection $\zD_1$ which intersect $\varphi^{-1}(\gamma)$. Note that, for each $1\leq i \leq k$,  the corresponding oriented boundary segments in $P_i$ and $\varphi^{-1}(P_i)$ give rise to  the same arrows in the algebra $A$. Then $w_{\eta_2}(\gamma_i)=w_{\eta_1}(\varphi^{-1}(\gamma_i))$  since they are expressed by the same formula as in \eqref{align:windingnumbersegments} and the claim holds.  
		\end{proof}
		
		\begin{Rem} \label{Rem:uptohomotopy}
			By Proposition \ref{Prop:models_same_algebra}, any graded gentle algebra, up to algebra isomorphism, admits a unique surface model.   We will see in the  Theorem \ref{Thm:derivedeq} that this statement still holds if one replaces \lq up to algebra isomorphism' by \lq up to derived equivalence'. 
		\end{Rem}
		
		\section{Graded gentle algebras of standard form and geometric invariants}
		In this section, we introduce a family of graded gentle algebras which we call of \emph{standard form} and  we show that any graded gentle algebra is derived equivalent to an algebra of standard form. We also show how to read off the geometric invariants from the underlying quiver of this standard form. 
		
		\subsection{The graded gentle algebras of standard form}\label{subsection:standardform}

		Given a graded surface with stops 
		$(S,M,\eta)$, we define gentle algebras of \emph{standard form} which are given by very specific admissible dissections such as in Figure \ref{Fig:StandardSurfaceModel}. Note that  in Figure \ref{Fig:StandardSurfaceModel}, the surface $S$ is of genus $g$ with   $b=u+v+1$ boundary components, where $v$ represents the number of non-stopped boundary components, and $u+1$ denotes the number of
		stopped boundary components with $m_0, m_1, \dotsc, m_u$ stops respectively. 
		Now we give the explicit description of the associated gentle algebras of standard form by quivers and relations.
		
		Let $u,v$ and $g$ be non-negative integers. Let $b=u+v+1$. Given a sequence of positive integers $m_0,\dots,m_u$, we define $Q(g; m_0, \dotsc, m_u;v)$  to be the following quiver
		\begin{align*}
			\begin{tikzpicture}[scale=1.4]
				\begin{scope}[shift={(-.9,0)}]
					\node[circle, inner sep=1pt, minimum size=3pt] (a1) at (1, 0){$1$};
					\node[circle, inner sep=1pt, minimum size=3pt] (a2) at (1.75, 0){$2$};
					\node[circle, inner sep=1pt, minimum size=3pt]  (a5) at (2.3,0) {\tiny $\dotsb$};
					\node[circle, inner sep=1pt, minimum size=3pt] (a6) at (3,0) {$ \ $};
					\node[circle, inner sep=1pt, minimum size=3pt] (a7) at (3.75,0) {\small $2g$};
					\node[circle, inner sep=1pt, minimum size=3pt] (a8) at (4.25,0) {};
					\node[circle, inner sep=1pt, minimum size=3pt] (a9) at (4.7,0) {\tiny $\dotsb$};
					\draw[transform canvas={yshift=1.5em},->]  (a1) -- node[above]{$\alpha_1$}  (a2) ;
					\draw[transform canvas={yshift=0em},<-] (a1) --node[above]{$\beta_1$}(a2);
					\draw[transform canvas={yshift=-1.5em},->] (a1) --node[above]{$\gamma_1$} (a2);
					\draw[transform canvas={yshift=0em},->] (a2) -- node[below]{\small$\delta_1$}  (a5); 
					\draw[transform canvas={yshift=0em},->] (a5) -- node[below]{\small $\delta_{g-1}$}  (a6); 
					\draw[transform canvas={yshift=1.5em},->]  (a6) -- node[above]{$\alpha_{g}$}  (a7) ;
					\draw[transform canvas={yshift=0em},<-] (a6) --node[above]{$\beta_{g}$}(a7);
					\draw[transform canvas={yshift=-1.5em},->] (a6) --node[above]{$\gamma_{g}$} (a7);
					\draw[->] (a7) --node[below]{\tiny$\theta_0$} (a8);
					\draw[->] (a8) --node[below]{\tiny$\theta_1$} (a9);
				\end{scope}
				\begin{scope}[shift={(4.4,0)}]
					\node[circle, inner sep=1pt, minimum size=3pt] (x1) at (0,0) { $1'$};
					\node[circle, inner sep=1pt, minimum size=2pt] (x2) at (1,0) {$2'$};
					\node[circle, inner sep=1pt, minimum size=2pt] (b1) at (-.3, .65) {};
					\node[circle, inner sep=1pt, minimum size=2pt] (b2) at (.5, .65) {\tiny $\dotsb$};
					\node[circle, inner sep=1pt, minimum size=2pt] (b3) at (1.3, .65) {};
					\draw[->] (a9) --node[below]{\tiny$\theta_{m_0-1}$} (x1);
					\draw[->] (x1) --node[right]{\tiny $x_1^1$} (b1);
					\draw[->] (b1) --node[above]{\tiny $x_2^1$} (b2);
					\draw[->] (b2) --node[above]{\tiny $x_{m_1-1}^1$} (b3);
					\draw[->] (b3) --node[left]{\tiny $x_{m_1}^1$} (x2);
					\draw[->] (x1) --node[below]{ \tiny$\theta_{m_0}$} (x2);
				\end{scope}
				\begin{scope}[shift={(6.7,0)}]
					\node[circle, inner sep=1pt, minimum size=3pt] (x5) at (-.65,0) {\tiny $\dotsb$};
					\node[circle, inner sep=1pt, minimum size=2pt] at (-.65,.3) {$\dotsb$};
					\node[circle, inner sep=1pt, minimum size=3pt] (x6) at (0,0) {};
					\node[circle, inner sep=1pt, minimum size=2pt] (x7) at (1,0) { \small $(2u)'$};
					\node[circle, inner sep=1pt, minimum size=2pt] (b1) at (-.3, .65) {};
					\node[circle, inner sep=1pt, minimum size=2pt] (b2) at (.5, .65) {\tiny $\dotsb$};
					\node[circle, inner sep=1pt, minimum size=2pt] (b3) at (1.3, .65) {};
					\draw[->] (x2) --node[right]{ } (x5);
					\draw[->] (x6) --node[right]{\tiny $x_1^u$} (b1);
					\draw[->] (b1) --node[above]{\tiny$x_2^u$} (b2);
					\draw[->] (b2) --node[above]{\tiny$x_{m_u-1}^u$} (b3);
					\draw[->] (b3) --node[left]{\tiny$x_{m_u}^u$} (x7);
					\draw[->] (x5) --node[right]{ } (x6);
					\draw[->] (x6) --node[below]{\tiny $\theta_{m_0+2u-2}$ } (x7);
				\end{scope}
				\begin{scope}[shift={(6.9,0)}]
					\node[circle, inner sep=1pt, minimum size=2pt] (y4) at (1, 0) {};
					\node[circle, inner sep=1pt, minimum size=2pt] (z1) at (1.5, 0) { };
					\draw [->] (1.45,.1) ..controls (1,.65) and (2,.65) .. node[above]{$y_1$}(1.55,.1);
					\node[circle, inner sep=1pt, minimum size=2pt] (z2) at (2, 0) {\tiny $\dotsb$};
					\node[circle, inner sep=1pt, minimum size=2pt] at (2, .3) { $\dotsb$};
					\node[circle, inner sep=1pt, minimum size=2pt] (z3) at (2.5, 0) {};
					\draw [->] (2.45,.1) ..controls (2,.65) and (3,.65) .. node[above]{$y_v$}(2.55,.1);
					\draw[->] (y4)--node[below]{\tiny $\epsilon_1$}(z1);
					\draw[->] (z1)--node[below]{\tiny $\epsilon_2$}(z2);
					\draw[->] (z2)--node[below]{\tiny$\epsilon_v$}(z3);
				\end{scope}
			\end{tikzpicture}
		\end{align*}

		\begin{figure}[H]
			{ \tiny \begin{center}
					\begin{tikzpicture}
						\begin{scope}[scale=0.6]
							\shadedraw[line width=0.01mm, top color= blue!15] (-30:9.5) arc(-30:210:9.5);
							\draw [] plot [smooth] coordinates {(210:9.5) (-6.6,-6.2) (-4,-4.75) (0,-3) (4,-4.75) (6.6,-6.2)  (-30:9.5)};
							\draw [dashed] (210:9.5) ..controls (-6.5,-5).. (-4, -4.75);
							\draw [dashed] (210:9.5) ..controls (-6.5,-4.5).. (-4, -4.75);
							\draw [dashed] (-30:9.5) ..controls (6.5,-5).. (4, -4.75);
							\draw [dashed] (-30:9.5) ..controls (6.5,-4.5).. (4, -4.75);
							\begin{scope}[rotate=110]
								\draw [thick,green] (0,0) ..controls (-1.8, 8.5) .. (0,8.5);
								\draw [thick,green] (0,0) ..controls (1.8, 8.5) .. (0,8.5);
								\draw [decoration={markings, mark=at position 0.45 with {\arrow{<}}}, postaction={decorate}, thick,gray] (0,6.5) ..controls (-.5,6.4) and (-.5, 9.6) .. (0,9.5);
								\node at (94:9)  {$s_1$};
								\draw [thick,dashed,gray] (0,6.5) ..controls (.5,6.4) and (.5, 9.6).. (0,9.5);
								\draw[fill=white] (-.6,6.3) ..controls (0,6) .. (.6,6.3);
								\draw [fill=white] (-.6,6.3) ..controls (0,6.6) .. (.6,6.3);
								\draw[fill=white] (-.6,6.3) -- (-.7,6.25);
								\draw[fill=white] (.6,6.3) -- (.7,6.25);
								\draw [decoration={markings, mark=at position 0.45 with {\arrow{<}}}, postaction={decorate},thick,gray]  (0, 6.25) circle (1);
								\node at (83:5)  {$t_1$};
								\draw [thick,green]  (0,0) ..controls (-.3, 6) .. (0,6.07);
								\draw [thick,green]  (0,0) ..controls (68:5) and (68:9) .. (72:9.5);
								\draw [dashed,green,thick]  (72:9.5) ..controls (75:9) and (2,5) .. (.35,5.9);
								\draw [dashed,green,thick]  (.35, 5.9) ..controls (.25, 6) .. (0, 6.07);
							\end{scope}
							
							\foreach \n in {0, 2, 4}
							{
								\node at (175-1.8*\n:7) [circle, fill, inner sep=.7pt]{};
							}
							\begin{scope}[rotate=63]
								\draw [thick,green] (0,0) ..controls (-1.8, 8.5) .. (0,8.5);
								\draw [thick,green] (0,0) ..controls (1.8, 8.5) .. (0,8.5);
								\draw [decoration={markings, mark=at position 0.45 with {\arrow{<}}}, postaction={decorate},thick,gray] (0,6.5) ..controls (-.5,6.4) and (-.5, 9.6) .. (0,9.5);
								\node at (94:9)  {$s_g$};
								\draw [thick,dashed,gray] (0,6.5) ..controls (.5,6.4) and (.5, 9.6).. (0,9.5);
								\draw[fill=white] (-.6,6.3) ..controls (0,6) .. (.6,6.3);
								\draw [fill=white] (-.6,6.3) ..controls (0,6.6) .. (.6,6.3);
								\draw[fill=white] (-.6,6.3) -- (-.7,6.25);
								\draw[fill=white] (.6,6.3) -- (.7,6.25);
								\draw [decoration={markings, mark=at position 0.45 with {\arrow{<}}}, postaction={decorate},thick,gray]  (0, 6.25) circle (1);
								\node at (83:5)  {$t_g$};
								\draw [thick,green]  (0,0) ..controls (-.3, 6) .. (0,6.07);
								\draw [thick,green]  (0,0) ..controls (68:5) and (68:9) .. (72:9.5);
								\draw [dashed,green,thick]  (72:9.5) ..controls (75:9) and (2,5) .. (.35,5.9);
								\draw [dashed,green,thick]  (.35, 5.9) ..controls (.25, 6) .. (0, 6.07);
							\end{scope}
							\begin{scope}[rotate=-22]
								\begin{scope}[shift={(90:6.5)}]
									\draw [ fill=gray!10] (0,0) circle (.8);
									\coordinate (a0) at (130:.8);
									\coordinate (a1) at (-90:.8);
									\coordinate (a2) at (20:.8);
									\coordinate (a3) at (80:.8);
									\node[fill,circle, red, inner sep=0pt, minimum size=2pt] at (-160:.8){};
									\node[fill,circle, red, inner sep=0pt, minimum size=2pt] at (-35:.8){};
									\node[fill,circle, red, inner sep=0pt, minimum size=2pt] at (50:.8){};
									\draw [decoration={markings, mark=at position 0.45 with {\arrow{<}}}, postaction={decorate}]  (102:.8) to (104:.801);
									\node at (100:1.2)  {\tiny $\theta_{m_0}$};
								\end{scope}
								\draw [green,thick]  (a0)..controls   (105:7) and (100:3.5).. (0,0);
								\draw [green,thick]  (a1) to (0,0);
								\draw [green,thick]  (a2)..controls   (76:7) and (80:3.5).. (0,0);
								\draw [thick,green] plot [smooth] coordinates {(a3) (85:7.8) (76:7) (70:4) (0,0)};
							\end{scope}
							\foreach \n in {0, 2, 4}
							{
								\node at (45-1.8*\n:6) [circle, fill, inner sep=.7pt]{};
							}
							\begin{scope}[rotate=-75]
								\begin{scope}[shift={(90:6.5)}]
									\draw [ fill=gray!10] (0,0) circle (.8);
									\coordinate (a0) at (130:.8);
									\coordinate (a1) at (-90:.8);
									\coordinate (a2) at (20:.8);
									\coordinate (a3) at (80:.8);
									\node[fill,circle, red, inner sep=0pt, minimum size=2pt] at (-160:.8){};
									\node[fill,circle, red, inner sep=0pt, minimum size=2pt] at (-35:.8){};
									\node[fill,circle, red, inner sep=0pt, minimum size=2pt] at (50:.8){};
									\draw [decoration={markings, mark=at position 0.45 with {\arrow{<}}}, postaction={decorate}]  (102:.8) to (104:.801);
									\node at (97:1.7)  {\tiny $\theta_{m_0+2u-2}$};
								\end{scope}
								\draw [green,thick]  (a0)..controls   (105:7) and (100:3.5).. (0,0);
								\draw [green,thick]  (a1) to (0,0);
								\draw [green,thick]  (a2)..controls   (76:7) and (80:3.5).. (0,0);
								\draw [thick,green] plot [smooth] coordinates {(a3) (85:7.8) (76:7) (70:4) (0,0)};
								\foreach \n in {0, 2, 4}
								{
									\node at (87-1*\n:4) [circle, fill, inner sep=.1pt]{};
								}
							\end{scope}
							
							\begin{scope}[rotate=-112]
								\begin{scope}[shift={(90:4)}]
									\draw [fill=gray!15] (0,0) circle (.25);
									\draw [decoration={markings, mark=at position 0.45 with {\arrow{<}}}, postaction={decorate}]  (90:.25) to (92:.251);
									\node at (90:.7)  {\tiny $y_1$};
								\end{scope}
								\draw[green, thick](0,0) to (90:3.75);
							\end{scope}
							\foreach \n in {0, 1, 2}
							{
								\node at (-31-2*\n:3) [circle, fill, inner sep=.2pt]{};
							}
							\begin{scope}[rotate=-130]
								\begin{scope}[shift={(90:4)}]
									\draw [fill=gray!15] (0,0) circle (.25);
									\draw [decoration={markings, mark=at position 0.45 with {\arrow{<}}}, postaction={decorate}]  (90:.25) to (92:.251);
									\node at (90:.7)  {\tiny $y_v$};
								\end{scope}
								\draw[green, thick](0,0) to (90:3.75);
							\end{scope}
							\begin{scope}[rotate=-12]
								\draw[green, thick](133:2) .. controls (133:3.5) and (125:3.5) .. (125:2);
								\draw[green, thick](116:2)..controls (114:3.5) and (108:3.5) .. (106:2);
								\foreach \n in {0, 2, 4}
								{
									\node at (75-1*\n:4) [circle, fill, inner sep=.1pt]{};
								}
								\foreach \n in {0, 2, 4}
								{
									\node at (123-1.4*\n:2.8) [circle, fill, inner sep=.3pt]{};
								}
							\end{scope}
							\begin{scope}[shift={(0,0)}]
								\draw [ fill=gray!10] (0,0) circle (2.2);
								\node[fill,circle, red, inner sep=0pt, minimum size=2pt] at (117:2.2){};
								\node[fill,circle, red, inner sep=0pt, minimum size=2pt] at (99:2.2){};
								\node[fill,circle, red, inner sep=0pt, minimum size=2pt] at (-90:2.2){};
								\draw [decoration={markings, mark=at position 0.45 with {\arrow{<}}}, postaction={decorate}]  (208:2.2) to (210:2.201);
								\node at (207:2.7)  {\tiny $\gamma_1$};
								\draw [decoration={markings, mark=at position 0.45 with {\arrow{<}}}, postaction={decorate}]  (198:2.2) to (200:2.201);
								\node at (195:2.7)  {\tiny $\beta_1$};
								\draw [decoration={markings, mark=at position 0.45 with {\arrow{<}}}, postaction={decorate}]  (183:2.2) to (185:2.201);
								\node at (183:2.7)  {\tiny $\alpha_1$};
								\draw [decoration={markings, mark=at position 0.45 with {\arrow{<}}}, postaction={decorate}]  (172:2.2) to (174:2.201);
								\draw [decoration={markings, mark=at position 0.45 with {\arrow{<}}}, postaction={decorate}]  (160:2.2) to (162:2.201);
								\node at (160:2.7)  {\tiny $\gamma_g$};
								\draw [decoration={markings, mark=at position 0.45 with {\arrow{<}}}, postaction={decorate}]  (149:2.2) to (151:2.201);
								\node at (147:2.7)  {\tiny $\beta_g$};
								\draw [decoration={markings, mark=at position 0.45 with {\arrow{<}}}, postaction={decorate}]  (137:2.2) to (139:2.201);
								\node at (136:2.7)  {\tiny $\alpha_g$};
								\draw [decoration={markings, mark=at position 0.45 with {\arrow{<}}}, postaction={decorate}]  (128:2.2) to (130:2.201);
								\node at (126:2.7)  {\tiny $\theta_0$};
								\draw [decoration={markings, mark=at position 0.45 with {\arrow{<}}}, postaction={decorate}]  (110:2.2) to (112:2.201);
								\draw [decoration={markings, mark=at position 0.45 with {\arrow{<}}}, postaction={decorate}]  (90:2.2) to (92:2.201);
								\node at (75:2.7)  {\tiny $x_1^1$};
								\draw [decoration={markings, mark=at position 0.45 with {\arrow{<}}}, postaction={decorate}]  (74:2.2) to (76:2.201);
								\draw [decoration={markings, mark=at position 0.45 with {\arrow{<}}}, postaction={decorate}]  (64:2.2) to (66:2.201);
								\node at (48:2.7)  {\tiny $x_{m_1}^1$};
								\draw [decoration={markings, mark=at position 0.45 with {\arrow{<}}}, postaction={decorate}]  (53:2.2) to (54:2.201);
								\node at (22:2.7)  {\tiny $x_1^u$};
								\node at (-1:2.8)  {\tiny $x_{m_u}^u$};
								\draw [decoration={markings, mark=at position 0.45 with {\arrow{<}}}, postaction={decorate}]  (38:2.2) to (40:2.201);
								\draw [decoration={markings, mark=at position 0.45 with {\arrow{<}}}, postaction={decorate}]  (21:2.2) to (22:2.201);
								\draw [decoration={markings, mark=at position 0.45 with {\arrow{<}}}, postaction={decorate}]  (11:2.2) to (12:2.201);
								\node at (-15:2.5)  {\tiny $\epsilon_1$};
								\draw [decoration={markings, mark=at position 0.45 with {\arrow{<}}}, postaction={decorate}]  (1:2.2) to (2:2.201);
								\draw [decoration={markings, mark=at position 0.45 with {\arrow{<}}}, postaction={decorate}]  (-11:2.2) to (-10:2.201);
								\draw [decoration={markings, mark=at position 0.45 with {\arrow{<}}}, postaction={decorate}]  (-32:2.2) to (-31:2.201);
							\end{scope}
						\end{scope}
					\end{tikzpicture}
			\end{center}}
			\caption{The standard surface model of an algebra $A$ of the form $(g; m_0, \dotsc, m_u; v)$, where the admissible dissection $\Delta$ is given by the arcs in green and  where the line field $\eta$ is determined by the grading of $A$. The boundary component in the middle contains $m_0$ stops. }
			\label{Fig:StandardSurfaceModel}
		\end{figure}

		Let $I(g; m_0, \dotsc, m_u; v)$ be the ideal generated by the union of the following sets
		\begin{itemize}
			\item $\{\za_{i}\zb_{i},\;  \zb_{i}\zg_{i},\;  \zg_{i}\zd_{i}, \;  \zd_{j}\za_{j+1} \mid 1\le i\le g,\ 1\le j\le g-1\}$ (denote $\delta_g=\theta_0$),
			\item $\{  \theta_i\theta_{i+1} \mid 0\leq i \leq N\}$, where $N=m_0+2u-2$ (denote $\theta_{N+1}=\epsilon_1$),
			\item $\{ \epsilon_iy_i,\; y_j \epsilon_{j+1} \mid 0 \leq i\leq v, \  1\leq j \leq v-1\}$.
		\end{itemize}
		
		Note that the relations are such that the corresponding algebra is a gentle algebra and the corresponding surface model is given by Figure \ref{Fig:StandardSurfaceModel}.
		
		\begin{Def}\label{Def:standardform}
			We say that a graded gentle algebra $A = \Bbbk Q/I$ is of \emph{standard form} $(g; m_0, \dotsc, m_u; v)$ if
			\[Q = Q(g; m_0, \dotsc, m_u;v)\quad \text{and} \quad I = I(g; m_0, \dotsc, m_u;v).\] 
			We call the corresponding  surface model as in Figure \ref{Fig:StandardSurfaceModel}  the \emph{standard surface model} of $A$.  
		\end{Def}
		
		Note that the quiver $Q(g;m_0, \dotsc, m_u; v)$ consists of three components: 
		\begin{itemize}
			\item 
			The leftmost component, associated with vertices  $1, 2, \dotsc, 2g$, originates from the \lq holes' of the surface;
			\item The middle component, involving vertices $1', 2', \dotsc, (2u)'$, represents  the stopped boundary components, excluding the special boundary in the center as illustrated in Figure \ref{Fig:StandardSurfaceModel}; 
			\item The rightmost component corresponds to the non-stopped boundary components. 
		\end{itemize}
		If $g=0$ or $u=0$ or $v=0$, the respective components of the quiver absent.
		
		Two graded gentle algebras of the same standard form are distinguished only by their gradings. Although their associated surfaces with stops are identical, their line fields may vary.

		\begin{Rem}\label{Rem:mb=1}
			If $b=1$ and $m_0\geq 1$, it follows that $u=v=0$. In this case, the surface $S$ has only one boundary component with $m_0$ stops and the quiver  $Q(g;m_0;0)$ is described as follows.
			
			\begin{align*}
				\begin{tikzpicture}[scale=1.4]
					\begin{scope}
						\node[circle, inner sep=1pt, minimum size=3pt] (a1) at (1, 0){$1$};
						\node[circle, inner sep=1pt, minimum size=3pt] (a2) at (2, 0){$2$};
						\node[circle, inner sep=1pt, minimum size=3pt] (a3) at (2.75, 0){$3$};
						\node[circle, inner sep=1pt, minimum size=3pt] (a4) at (3.75, 0){$4$};
						\node[circle, inner sep=1pt, minimum size=3pt]  (a5) at (4.5,0) {\tiny $\dotsb$};
						\node[circle, inner sep=1pt, minimum size=3pt] (a6) at (5.25,0) {$ \ $};
						\node[circle, inner sep=1pt, minimum size=3pt] (a7) at (6.25,0) {\small $2g$};
						\node[circle, inner sep=1pt, minimum size=3pt] (a8) at (7.25,0) {};
						\node[circle, inner sep=1pt, minimum size=3pt] (a9) at (8.25,0) {$\dotsb$};
						\node[circle, inner sep=1pt, minimum size=3pt] (a10) at (9.25,0) {};
						\draw[transform canvas={yshift=1.5em},->]  (a1) -- node[above]{$\alpha_1$}  (a2) ;
						\draw[transform canvas={yshift=0em},<-] (a1) --node[above]{$\beta_1$}(a2);
						\draw[transform canvas={yshift=-1.5em},->] (a1) --node[above]{$\gamma_1$} (a2);
						\draw[transform canvas={yshift=0em},->] (a2) -- node[below]{\small$\delta_1$}  (a3); 
						\draw[transform canvas={yshift=1.5em},->]  (a3) -- node[above]{$\alpha_2$}  (a4) ;
						\draw[transform canvas={yshift=0em},<-] (a3) --node[above]{$\beta_2$}(a4);
						\draw[transform canvas={yshift=-1.5em},->] (a3) --node[above]{$\gamma_2$} (a4);
						\draw[transform canvas={yshift=0em},->] (a4) -- node[below]{\small $\delta_2$}  (a5); 
						\draw[transform canvas={yshift=0em},->] (a5) -- node[below]{\small $\delta_{g-1}$}  (a6); 
						\draw[transform canvas={yshift=1.5em},->]  (a6) -- node[above]{$\alpha_{g}$}  (a7) ;
						\draw[transform canvas={yshift=0em},<-] (a6) --node[above]{$\beta_{g}$}(a7);
						\draw[transform canvas={yshift=-1.5em},->] (a6) --node[above]{$\gamma_{g}$} (a7);
						\draw[transform canvas={yshift=0em},->] (a7) -- node[below]{\small $\theta_0$}  (a8); 
						\draw[transform canvas={yshift=0em},->] (a8) -- node[below]{\small $\theta_1$}  (a9); 
						\draw[transform canvas={yshift=0em},->] (a9) -- node[below]{\small $\theta_{m_0-2}$}  (a10); 
					\end{scope}
				\end{tikzpicture}
			\end{align*}
			In particular, if  $m_0=1$ then the arrows $\theta_0,\dotsc, \theta_{m_0-2}$ disappear in $Q(g;1;0)$, as shown in \eqref{equ:An} below. 
		\end{Rem}

		The following result shows that, up to derived equivalence, all graded gentle algebras are of standard form. 
		\begin{Prop}\label{prop:standardsation}
			Let $A$ be a graded gentle algebra. Let $(S, M, \eta, \Delta)$ be a surface model of $A$. Then $A$ is derived equivalent to a graded gentle algebra of the form $(g; m_0, \dotsc, m_u; v)$. 
		\end{Prop}
		\begin{proof}
			In the graded surface with stops $(S, M, \eta)$, we consider the admissible dissection given by the green arcs  in Figure \ref{Fig:StandardSurfaceModel} instead of $\Delta$. Note that the corresponding graded gentle algebra is of the form $(g; m_0, \dotsc, m_u; v)$, which by Theorem \ref{thm:hkkequivalence} is derived equivalent to $A$.  
		\end{proof}
		
		\begin{Rem}
			Let $B$ be a graded gentle algebra, which is not necessarily homologically smooth and proper. Then its surface model $(S, M, \eta)$ might have fully-stopped  and non-stopped boundary components.  In this case, there is  a standard  admissible dissection $\Delta$ on $(S, M, \eta)$ similar to Figure \ref{Fig:StandardSurfaceModel}, whose underlying quiver is as follows,
			\begin{align*}
				\begin{tikzpicture}[scale=1.4]
					\begin{scope}[shift={(-.9,0)}]
						\node[circle, inner sep=1pt, minimum size=3pt] (a1) at (1, 0){$1$};
						\node[circle, inner sep=1pt, minimum size=3pt] (a2) at (1.75, 0){$2$};
						\node[circle, inner sep=1pt, minimum size=3pt]  (a5) at (2.3,0) {\tiny $\dotsb$};
						\node[circle, inner sep=1pt, minimum size=3pt] (a6) at (3,0) {$ \ $};
						\node[circle, inner sep=1pt, minimum size=3pt] (a7) at (3.75,0) {\small $2g$};
						\node[circle, inner sep=1pt, minimum size=3pt] (a8) at (4.25,0) {};
						\node[circle, inner sep=1pt, minimum size=3pt] (a9) at (4.7,0) {\tiny $\dotsb$};
						\draw[transform canvas={yshift=1.5em},->]  (a1) -- node[above]{$\alpha_1$}  (a2) ;
						\draw[transform canvas={yshift=0em},<-] (a1) --node[above]{$\beta_1$}(a2);
						\draw[transform canvas={yshift=-1.5em},->] (a1) --node[above]{$\gamma_1$} (a2);
						\draw[transform canvas={yshift=0em},->] (a2) -- node[below]{\small$\delta_1$}  (a5); 
						\draw[transform canvas={yshift=0em},->] (a5) -- node[below]{\small $\delta_{g-1}$}  (a6); 
						\draw[transform canvas={yshift=1.5em},->]  (a6) -- node[above]{$\alpha_{g}$}  (a7) ;
						\draw[transform canvas={yshift=0em},<-] (a6) --node[above]{$\beta_{g}$}(a7);
						\draw[transform canvas={yshift=-1.5em},->] (a6) --node[above]{$\gamma_{g}$} (a7);
						\draw[->] (a7) --node[below]{\tiny$\theta_0$} (a8);
						\draw[->] (a8) --node[below]{\tiny$\theta_1$} (a9);
					\end{scope}
					\begin{scope}[shift={(4.4,0)}]
						\node[circle, inner sep=1pt, minimum size=3pt] (x1) at (0,0) { $1'$};
						\node[circle, inner sep=1pt, minimum size=2pt] (x2) at (1,0) { $2'$};
						\node[circle, inner sep=1pt, minimum size=2pt] (b1) at (-.3, .65) {};
						\node[circle, inner sep=1pt, minimum size=2pt] (b2) at (.5, .65) {\tiny $\dotsb$};
						\node[circle, inner sep=1pt, minimum size=2pt] (b3) at (1.3, .65) {};
						\draw[->] (a9) --node[below]{\tiny$\theta_{m_0-1}$} (x1);
						\draw[->] (x1) --node[right]{\tiny $x_1^1$} (b1);
						\draw[->] (b1) --node[above]{\tiny $x_2^1$} (b2);
						\draw[->] (b2) --node[above]{\tiny $x_{m_1-1}^1$} (b3);
						\draw[->] (b3) --node[left]{\tiny $x_{m_1}^1$} (x2);
						\draw[->] (x1) --node[below]{ \tiny$\theta_{m_0}$} (x2);
					\end{scope}
					\begin{scope}[shift={(6.7,0)}]
						\node[circle, inner sep=1pt, minimum size=3pt] (x5) at (-.65,0) {\tiny $\dotsb$};
						\node[circle, inner sep=1pt, minimum size=2pt] at (-.65,.3) {$\dotsb$};
						\node[circle, inner sep=1pt, minimum size=3pt] (x6) at (0,0) {};
						\node[circle, inner sep=1pt, minimum size=2pt] (x7) at (1,0) { \small $(2u)'$};
						\node[circle, inner sep=1pt, minimum size=2pt] (b1) at (-.3, .65) {};
						\node[circle, inner sep=1pt, minimum size=2pt] (b2) at (.5, .65) {\tiny $\dotsb$};
						\node[circle, inner sep=1pt, minimum size=2pt] (b3) at (1.3, .65) {};
						\draw[->] (x2) --node[right]{ } (x5);
						\draw[->] (x6) --node[right]{\tiny $x_1^u$} (b1);
						\draw[->] (b1) --node[above]{\tiny$x_2^u$} (b2);
						\draw[->] (b2) --node[above]{\tiny$x_{m_u-1}^u$} (b3);
						\draw[->] (b3) --node[left]{\tiny$x_{m_u}^u$} (x7);
						\draw[->] (x5) --node[right]{ } (x6);
						\draw[->] (x6) --node[below]{\tiny $\theta_{m_0+2u-2}$ } (x7);
					\end{scope}
					\begin{scope}[shift={(8.5,0)}]
						\node[circle, inner sep=1pt, minimum size=2pt] (y1) at (-.7, 0) {};
						\node[circle, inner sep=1pt, minimum size=2pt] (y2) at (0,0) {};
						\draw [->] (-.05,.1) ..controls (-.5,.65) and (.5,.65) ..node[above]{$y_1$}  (.05,.1);
						\node[circle, inner sep=1pt, minimum size=2pt] (y3) at (.5, 0) {\tiny $\dotsb$};
						\node[circle, inner sep=1pt, minimum size=2pt] at (.5, .3) { $\dotsb$};
						\node[circle, inner sep=1pt, minimum size=2pt] (y4) at (1, 0) {};
						\draw [->] (.95,.1) ..controls (.5,.65) and (1.5,.65) ..node[above]{$y_v$} (1.05,.1);
						\node[circle, inner sep=1pt, minimum size=2pt] (z1) at (1.5, 0) { };
						\draw [->] (1.45,.1) ..controls (1,.65) and (2,.65) .. node[above]{$z_1$}(1.55,.1);
						\node[circle, inner sep=1pt, minimum size=2pt] (z2) at (2, 0) {\tiny $\dotsb$};
						\node[circle, inner sep=1pt, minimum size=2pt] at (2, .3) { $\dotsb$};
						\node[circle, inner sep=1pt, minimum size=2pt] (z3) at (2.5, 0) {};
						\draw [->] (2.45,.1) ..controls (2,.65) and (3,.65) .. node[above]{$z_w$}(2.55,.1);
						\draw[->] (y1)--node[below]{}(y2);
						\draw[->] (y2)--node[below]{}(y3);
						\draw[->] (y3)--node[below]{\tiny$\theta_N$}(y4);
						\draw[->] (y4)--node[below]{\tiny $\epsilon_1$}(z1);
						\draw[->] (z1)--node[below]{\tiny $\epsilon_2$}(z2);
						\draw[->] (z2)--node[below]{\tiny$\epsilon_w$}(z3);
					\end{scope}
				\end{tikzpicture}
			\end{align*}
			where the loops $z_1, \dotsc, z_w$, with relations $z_i^2=0$, for $1\leq i \leq w$, correspond to the fully-stopped boundary components, respectively. 
		\end{Rem}

		\subsection{The geometric invariants of algebras of  standard form}
		\label{section:geometricinvariantsofstandardform}
		In this subsection we calculate the geometric invariants considered in Theorem \ref{Prop:homotopic_linefield} for any graded gentle algebra of standard form $(g; m_0, \dotsc, m_u; v)$. 
		
		Let $A$ be a graded gentle algebra of the form $(g; m_0, \dotsc, m_u; v)$. Let $(S, M, \eta, \Delta)$ be the associated standard surface model of $A$ illustrated in Figure \ref{Fig:StandardSurfaceModel}, where (the homotopy class of) the line field $\eta$ is uniquely determined by the grading of $A$. Let $s_i, t_i$ be the natural non-separating simple closed curves around each genus hole for $1 \le i \le g$, which induce a symplectic basis of $\mathrm{H}_1(\bar S)$ as in Subsection \ref{subsection:combinatorialinvariants}.  
		
		Note that the dissection $\Delta$ cuts the surface $S$ into polygons. Then by \eqref{align:windingnumbersegments} (see also  \cite[Remark 3.15]{LP20}) the winding numbers of the curves $s_i$ and $t_i$ are given by  
		\begin{equation} \label{eq:windings_i}
			w_\eta(s_i) = |\za_i| + |\zb_i| -1\\
			\mbox{ \, and \,
			} \\
			w_\eta(t_i) = |\zb_i| + |\zg_i| -1. 
		\end{equation}
		(See Example \ref{Ex:windingnumberg=2} below for an explicit calculation.)
		Similarly, the winding number of the $i$-th boundary component with $m_i$ stops is given by 
		\begin{align}\label{eq:winding-boundaryi}
			w_\eta(\partial_j S) = -|\theta_{m_0+2j-2}|+\sum_{k=1}^{m_j} |x_k^j| \quad \text{for $1\leq j \leq u$}.
		\end{align}
		The winding numbers of the non-stopped boundary components are given by 
		\begin{align}\label{equ:C}
			w_\eta(\partial_{u+j}S) = -|y_j| \quad \text{for $1\leq j \leq v$}.
		\end{align}
		Then by the Poincar\'e--Hopf index theorem \eqref{algin:PHindex}, we have 
		\begin{equation}
			\label{equ:HP}
			w_{\eta}(\partial_0S) = 4-4g-2b-\sum_{j=1}^{u+v}w_\eta(\partial_j S),
		\end{equation}
		where $\partial_0S$ is the boundary component in the middle of Figure \ref{Fig:StandardSurfaceModel}. In particular, if $S$ has only one boundary component (i.e.\ $b=1$ and thus $u=v=0$) then we have $w_{\eta}(\partial S) = 2-4g.$

		Based on the above computation,  we introduce the following integers from the grading of $A$:
		\begin{align}\label{align:numbersofthesumgradingsAn}
			\begin{aligned}
				a_{i}:=&|\za_{i}|+|\zb_{i}|-1  \qquad \qquad\qquad \qquad\qquad\text{for $1\leq i \leq g$}\\ b_{i}:=&|\zb_{i}|+|\zg_{i}|-1  \qquad \text{for $1\leq i \leq g$} \\
				c_j := & \begin{cases} -|\theta_{m_0+2j-2}|+\sum_{k=1}^{m_j} |x_k^j|   & \quad \qquad\qquad \quad \text{for $ 1\leq j \leq u$}\\
					-|y_j| & \qquad\qquad \quad \quad \text{for $u< j \leq u+v$}
				\end{cases}
			\end{aligned}
		\end{align}
		and $c_0: = 4-4g-2b-\sum_{j=1}^{u+v} c_j.$
		By Formulas \eqref{eq:windings_i}, \eqref{eq:winding-boundaryi} and  \eqref{equ:HP}, the integers $a_i, b_i$ correspond to  the winding numbers of the simple closed curves $s_i$ and $t_i$, respectively, while $c_j$ represents the winding number of  the $j$-th boundary component. That is, the sequence \eqref{align:allthewindingnumbers} can be written as 
		\begin{align}\label{align:windingnumberse}
			W_\eta =\{c_0, c_1, \dotsc, c_{b-1}, a_1, \dotsc, a_g, b_1, \dotsc, b_g\}.
		\end{align}
		
		\begin{Ex}\label{Ex:windingnumberg=2}
			Let $A$ be a graded gentle algebra of the form $(g=2;1;0)$, and let $(S,M,\eta, \zD)$ denote the corresponding standard surface model. The surface $S$ is of genus $2$ with a single boundary and exactly one stop on it, and  $\zD$ cuts $S$ into a single polygon, as illustrated in Figure \ref{Fig:polygon}. 
			
			By applying \eqref{align:windingnumbersegments} to the left polygon of Figure \ref{Fig:polygon}, we deduce $$w_{\eta}(s_i) =|\alpha_i|+|\beta_i|-1 \quad \text{and} \quad w_{\eta}(t_i) = |\beta_i|+|\gamma_i|-1$$ for $i=1,2$.  For the right polygon, the simple closed curve $u$, which is homotopic to the boundary, is subdivided  by the dissection into segments $u_1, u_2, \dotsc, u_8$ in sequential order. By \eqref{align:windingnumbersegments} we have $$w_\eta(u_8) = -6+|\alpha_1|+|\beta_1|+|\gamma_1|+|\delta_1|+|\alpha_2|+|\beta_2|+|\gamma_2|.$$ Note that $u_1$ is an oriented angle from $\ell_1$ to $\ell_2$ and thus $w_\eta(u_1) = -|\gamma_1|$. Similarly, we have 
			$$w_\eta(u_2) = -|\beta_1|,\; w_\eta(u_3)= -|\alpha_1|, \dotsc,  \; w_\eta(u_7)=-|\alpha_2|.$$ 
			So $w_\eta(\partial S)=w_\eta(u) = \sum_{i=1}^8 w_\eta(u_i) = -6$, which satisfies the Poincar\'e--Hopf index theorem. 
		\end{Ex}

		\begin{figure}[H]
			\begin{tikzpicture}[scale=0.4, transform shape]
				\begin{scope}[decoration={markings, mark=at position 0.5 with {\arrow{>}}}] 
					\foreach \x in {0,1,...,7}
					{
						\draw[thick, green] (\x*40-70:5) to (\x*40-30:5);
					}
					\foreach \x in {1,..., 7}
					{\begin{scope}[shift={(\x*40-70:5)}, rotate=\x*40+20]
							\draw[white, line width=5pt] (0,0) to (20:.5);
							\draw[white, line width=5pt] (0,0) to (160:.5);
							\draw[bend left, postaction={decorate}] (160:.5) to (20:.5);
						\end{scope}
					}
					\foreach \x in {1,..., 8}
					{\begin{scope}[shift={(\x*40-70:5)}, rotate=\x*40+20]
							\coordinate (a\x) at (160:1.5);
						\end{scope}
					}
					\draw[] (-70:5) to [curve through = {.. (-90:4)..}] (-110:5);
					\node[fill,circle, red, inner sep=0pt, minimum size=6pt] at (-90:4){};
					\draw[bend right, postaction={decorate}] (a3) to (a1);
					\draw[bend right, postaction={decorate}] (a4) to (a2);
					\draw[bend right, postaction={decorate}] (a7) to (a5);
					\draw[bend right, postaction={decorate}] (a8) to (a6);
					\node at (2.3,-1)  {\large $s_1$};
					\node at (1.6,2)  {\large $t_1$};
					\node at (-2.2,-1)  {\large $t_2$};
					\node at (-1.7,2)  {\large $s_2$};
					\node at (-30:4.3)  {\large $\alpha_1$};
					\node at (10:4.3)  {\large $\beta_1$};
					\node at (50:4.3)  {\large $\gamma_1$};
					\node at (90:4.3)  {\large $\delta_1$};
					\node at (130:4.3)  {\large $\alpha_2$};
					\node at (170:4.3)  {\large $\beta_2$};
					\node at (210:4.3)  {\large $\gamma_2$};
					\node at (-50:5)  {\large $\ell_1$};
					\node at (-10:5)  {\large $\ell_2$};
					\node at (30:5)  {\large $\ell_1$};
					\node at (70:5)  {\large $\ell_2$};
					\node at (110:5)  {\large $\ell_3$};
					\node at (150:5)  {\large $\ell_4$};
					\node at (190:5)  {\large $\ell_3$};
					\node at (230:5)  {\large $\ell_4$};
				\end{scope}
				\begin{scope}[shift={(15,0)}, decoration={markings, mark=at position 0.5 with {\arrow{>}}}] 
					\foreach \x in {0,1,...,7}
					{
						\draw[thick, green] (\x*40-70:5) to (\x*40-30:5);
					}
					
					\foreach \x in {1,..., 7}
					{\begin{scope}[shift={(\x*40-70:5)}, rotate=\x*40+20]
							\draw[white, line width=5pt] (0,0) to (20:.4);
							\draw[white, line width=5pt] (0,0) to (160:.4);
							\draw[bend left, postaction={decorate}] (160:.4) to (20:.4);
							\draw[postaction={decorate}] (160:1.5) ..controls(90:2).. (20:1.5);
							\coordinate (c\x) at (90:2.1);
						\end{scope}
					}
					
					\foreach \x in {0, 8}
					{\begin{scope}[shift={(\x*40-70:5)}, rotate=\x*40+20]
							\coordinate (a\x) at (160:1.5);
							\coordinate (b\x) at (20:1.5);
						\end{scope}
					}
					\draw[postaction={decorate}] (a8) ..controls(90:-2).. (b0);
					\node at (c1) {$u_3$};
					\node at (c2) {$u_2$};
					\node at (c3) {$u_1$};
					\node at (c4) {$u_4$};
					\node at (c5) {$u_7$};
					\node at (c6) {$u_6$};
					\node at (c7) {$u_5$};
					\node at (90:-2) {$u_8$};
					
					\draw[] (-70:5) to [curve through = {.. (-90:4)..}] (-110:5);
					\node[fill,circle, red, inner sep=0pt, minimum size=6pt] at (-90:4){};
					\node at (-30:4.3)  {\large $\alpha_1$};
					\node at (10:4.3)  {\large $\beta_1$};
					\node at (50:4.3)  {\large $\gamma_1$};
					\node at (90:4.3)  {\large $\delta_1$};
					\node at (130:4.3)  {\large $\alpha_2$};
					\node at (170:4.3)  {\large $\beta_2$};
					\node at (210:4.3)  {\large $\gamma_2$};
					\node at (-50:5)  {\large $\ell_1$};
					\node at (-10:5)  {\large $\ell_2$};
					\node at (30:5)  {\large $\ell_1$};
					\node at (70:5)  {\large $\ell_2$};
					\node at (110:5)  {\large $\ell_3$};
					\node at (150:5)  {\large $\ell_4$};
					\node at (190:5)  {\large $\ell_3$};
					\node at (230:5)  {\large $\ell_4$};
				\end{scope}
			\end{tikzpicture}
			\caption{The polygon cut out by the standard admissible dissection $\Delta$ for $g=2, b=1=m_0$. The corresponding gentle algebra $A(\Delta)$ is described in \eqref{equ:An}.  In the left polygon, there are four closed curves $s_i, t_i$, for $i=1,2$, whereas the right polygon contains a single closed curve $u$, subdivided  into $8$ segments $u_i$ by the green arcs in sequential order. Note that $u$ is homotopic to the (unique) boundary component.}
			\label{Fig:polygon}
		\end{figure}

		Next, we proceed to derive the explicit expressions for the invariants of interest using the stated winding number formula. The invariant $\w{\mathcal A}(\eta)$ in the case of $g =1$ can be calculated using the integers $a_i, b_i, c_j$ introduced in \eqref{align:numbersofthesumgradingsAn}. More precisely, we have 
		\begin{align}\label{align:wAA1}
			\begin{aligned}
				\w{\mathcal A}(\eta)&=\gcd(w_{\eta}(s_1), w_{\eta}(t_1), w_{\eta}(\partial_{0}S)+2, \dots, w_{\eta}(\partial_{b-1}S)+2)\\
				&=\gcd(a_1, b_1,c_0+2, c_1+2, \dotsc, c_{b-1}+2)\\
				&=\gcd(a_1, b_1,c_1+2,\dotsc, c_{b-1}+2),
			\end{aligned}
		\end{align}
		where the third equality follows since $c_0 +2=-\sum_{j=1}^{b-1}(c_j +2)$ by \eqref{algin:PHindex}. In particular, if $g=1 =b$ then 
		\begin{align}\label{align:widetildeAinvariant}
			\w{\mathcal A}(\eta)=\gcd(|\alpha_1|+|\beta_1|-1, |\beta_1|+|\gamma_1|-1) = \gcd(a_1,b_1).
		\end{align}

		In the case of $g>1$, we have the following descriptions for the invariant $\sigma$ and the Arf invariant $\mathcal A(\eta)$.
		\begin{enumerate}[\rm(a)]
			\item If there exists $1\leq i \leq g$ or $0\leq j< b$ such that at least one of $a_{i}, b_{i}, c_j$  is odd, then $\sigma(\eta) = 1$;
			\item  If  $a_i, b_i$ and $c_j$ are even for all $1\leq i \leq g$ and $0\leq j <b$, then  $\sigma(\eta) = 0$;
			\item Under the assumption in (b), if we further assume that  $c_j= 2\ \mod 4$ for all $0\leq j < b$, then the Arf invariant $\mathcal A(\eta)$ is defined by 
			\[
			\mathcal A(\eta) := 
			\sum_{i=1}^g (\frac{1}{2}a_i +1) (\frac{1}{2}b_i+1) \ \mod 2.
			\] 
			
		\end{enumerate}

		The following result follows directly from Theorem \ref{Prop:homotopic_linefield}. 
		\begin{Prop}\label{prop:derivedinvariantstandard}
			Let $A$ and $A'$ be two graded gentle algebras of the form $(g; m_0, \dotsc, m_u;v)$. Let $a_i,b_i,c_j$ and $a_i',b_i',c_j'$ be the integers defined as in \eqref{align:numbersofthesumgradingsAn} associated to $A$ and $A'$ respectively. Let $(S, M, \eta, \zD)$ and $(S,M, \eta', \zD)$ be the standard surface models of $A$ and $A'$. Then there exists an orientation preserving homeomorphism $\varphi: S\ra S$ such that $\varphi(M)=M$ and $\varphi_{*}(\eta)$ is homotopic to $\eta'$  if and only if up to permutation of indices  we have
			$c_j=c_j'$ for each $0\leq j\le b-1$ and in addition 
			\begin{enumerate}[\rm (a)]
				\item if $g=1$ then $\w{\mathcal A}(\eta) = \w{\mathcal A}(\eta')$.
				\item if $g>1$ then one of the following conditions holds 
				\begin{itemize}
					\item 
					$\sigma(\eta)=1=\sigma(\eta')$;
					\item 
					$\sigma(\eta)=0=\sigma(\eta')$ and there exists $0\leq j \le b-1$ such that $c_j = 0 \ \mod 4$;
					\item 
					$\sigma(\eta)=0=\sigma(\eta')$  and $c_j = 2\ \mod 4$ for all $0\leq j \le b-1$,  and 
					\begin{equation}\label{Arf} \sum_{i=1}^g  (\frac{1}{2}a_i 
						+1) (\frac{1}{2}b_i +1) = \sum_{i=1}^g (\frac{1}{2}a'_i 
						+1) (\frac{1}{2}b'_i +1)\  \mod 2.\end{equation}
				\end{itemize}
			\end{enumerate}
			Under the above conditions, $A$ is derived equivalent to $A'$.
		\end{Prop}

		\begin{Rem}
			In section \ref{Section:derivedinvariant}, we will show  that the conditions in Proposition \ref{prop:derivedinvariantstandard} are also necessary for $A$ and $A'$ to be derived equivalent.
		\end{Rem}

		\begin{Rem}\label{Rem:furtherstandard}
			Let $A$ be a graded gentle algebra of the form $(g; m_0, \dotsc, m_u; v)$ with integers $a_i,b_i, c_j$ as above. From Proposition \ref{prop:derivedinvariantstandard}, we have the following observations. 
			
			\begin{enumerate}[\rm (1)]
				\item Note that, in general, there are different gradings on $A$ such that the invariants considered in Proposition \ref{prop:derivedinvariantstandard} remain unchanged (so that the correspoding algebras are derived  equivalent). In the following, we list some possible changes of gradings, which will be used in the proof of Theorem \ref{Cor:siltingexistence}.
				
				\begin{enumerate}[(a)]
					\item
					Since $c_j$ is fixed for $\eta$, the only constraint for the gradings of the arrows $x_k^j$ in $Q(g;m_0,\dotsc, m_u; v)$ is given by
					\[
					c_j = -|\theta_{m_0+2j-2}|+\sum_{k=1}^{m_j} |x_k^j|.
					\]
					It follows that we can choose $|x_k^j|\ll 0$ for each $1\leq k \leq m_j$ so that $|\theta_{m_0+2j-2}|<0$ and the above equality holds. 
					
					\item Recall that the winding number of the non-stopped boundary component corresponding to $y_i$ is given by $-|y_i|$ by \eqref{align:numbersofthesumgradingsAn}, thus the grading of $y_i$ cannot be changed.
					\item 
					For the case $g=1$, if $\widetilde{A}(\eta) \neq 0$, then we can adjust the gradings of $\za_1$, $\zb_1$ and $\zg_1$ to be negative so that $\w{A}(\eta)$ is unchanged. This follows from the fact that by \eqref{align:wAA1},  $\widetilde A(\eta) = \gcd (|\alpha_1|+|\beta_1|-1, |\beta_1|+|\gamma_1|-1, c_1+2,\dotsc, c_{b-1}+2).$
					
					\item
					For the higher genus case, since the invariants involving $a_i, b_i$ are considered  $\mod 2$, we can  choose $|\alpha_i|, |\beta_i| < 0$ for $1\leq i \leq g$ such that $A(\eta)$ unchanged.
				\end{enumerate}
				The new gradings  give rise to a line field $\eta'$, which is in the same orbit of $\eta$ under the mapping class group by Proposition \ref{prop:derivedinvariantstandard}. Denote by $A'$ the algebra with the new gradings. As a result, if all $|y_i|\le 0$ and $\widetilde{A}(\eta) \neq 0$ for $g =1$, the algebra $A'$, which is derived equivalent to $A$, is non-positive.
				
				\item Similar to (1), under certain conditions, we can consider a different change of gradings,  which we will use in the proof of Theorem \ref{Thm:partialsiltingmainresult}.
				\begin{itemize}
					\item 
					Assume $g=1$, $b>1$ and   $\w{A}(\eta) = \gcd(c_1+2, \dots, c_{b-1}+2)$.  Then $A$ is derived equivalent to a graded gentle algebra $A'$ of the same form, whose grading yields numbers $a_1',b_1',c_j$ such that $a_1'=0=b_1'$. 
					The reason is as follows. Denote  by $\eta'$ the line field induced by the grading of $A'$. Then
					\begin{align*}\w{A}(\eta') &= \gcd(a_1', b_1', c_1+2, \dots, c_{b-1}+2)\\
						&= \gcd(c_1+2, \dots, c_{b-1}+2)\\
						&=\w{A}(\eta).\end{align*}
					It then follows from Proposition \ref{prop:derivedinvariantstandard} that $A$ and $A'$ are derived equivalent. 
					\item Assume $g>1$. Then by Proposition~\ref{prop:derivedinvariantstandard} (b),   $A$ is derived equivalent to a graded gentle algebra $A'$ of the same form, with numbers $a_i',b_i',c_j$ such that  $a_1'=0=b_1'$ and $a_i',b_i'<0$ for each $2\leq i\leq g$. 
				\end{itemize}
			\end{enumerate}
		\end{Rem}

		\begin{Rem}\label{Rem:nonzero}
			If $S$ is of genus $1$ and has a non-stopped boundary component whose winding number is non-negative then by \eqref{align:wAA1}  we have
			\[\widetilde A(\eta)=\gcd(a_1, b_1,c_1+2,\dotsc, c_{b-1}+2)\not=0.\]
		\end{Rem}

		Consider the following assumptions.
		\begin{Assu} \label{asumption:silting}
			Let $(S, M, \eta)$ be a graded surface with stops.
			\begin{enumerate}[\rm (1)]
				\item Assume that the winding number of each non-stopped boundary component (if it exists) of $S$ is non-negative.
				\item Assume that $\widetilde{A}(\eta) \neq 0$ for the case where   $S$ is a torus with exactly one boundary component and one stop.
			\end{enumerate}
		\end{Assu}

		Thanks to the study of above numerical  invariants, we have the following result. 
		\begin{Thm}\label{Cor:siltingexistence}
			Let $(S, M, \eta)$ be a graded surface with stops.
			Keep Assumption \ref{asumption:silting}. 
			Then there is an admissible dissection $\Delta$ on $(S, M, \eta)$ such that $A(\Delta)$ is non-positive. 
		\end{Thm}
		\begin{proof}
			It follows from Proposition \ref{prop:derivedinvariantstandard} that two line fields are in the same orbit under the action of the mapping class group of $S$ if and only if the associated numerical invariants coincide.  Consider the line field $\eta'$ constructed in Remark \ref{Rem:furtherstandard}(1).  We show the result by considering the following cases.

			\emph{Case 1} Assume $g\not=1$. We consider the standard dissection $\Delta'$ as illustrated in Figure \ref{Fig:StandardSurfaceModel} and  the associated gentle algebra $A(\Delta')$ is of standard form as in Definition \ref{Def:standardform}. 
			Note that Assumption \ref{asumption:silting}(1) is equivalent to that the grading of each loop $y_i$ is non-positive by \eqref{equ:C}. Then as explained in Remark \ref{Rem:furtherstandard}(1),  $A(\Delta')$ is non-positive in terms of the line field $\eta'$.

			\emph{Case 2} Assume $g =1$ and $\widetilde A(\eta) \neq 0$. We also  consider the standard dissection $\Delta'$ and the associated algebra $A(\zD')$. Since $\widetilde A(\eta)=\gcd(a_1, b_1,c_1+2,\dotsc, c_{b-1}+2) \neq 0,$  in the construction of $\eta'$ in Remark \ref{Rem:furtherstandard}(1)(c),  we can choose $a_1, b_1\ll 0$ so that  $|\alpha_1|, |\beta_1|$ and $|\gamma_1|$ are negative. That is, in this case, the gradings of all the arrows of $A(\zD')$ are non-positive. 
			
			\emph{Case 3}  Assume $g=1$ and $\widetilde A(\eta) = 0$. By Assumption \ref{asumption:silting}(2), this can only happen when either $\# M =1$ and there exist non-stopped boundary components, or $\# M > 1$. For the former subcase, since the winding number of each non-stopped boundary component is non-negative by Assumption \ref{asumption:silting}(1) it follows from Remark \ref{Rem:nonzero} that $\widetilde A(\eta) \neq 0$. Contradiction. Therefore, we have  $\# M > 1$. In this case, by \cite[Proposition 3.5]{CS22}, there exists an exceptional dissection $\Delta$ (i.e.\ the associated quiver is acyclic), so that We can shift the gradings of arcs in $\Delta$ so that $A(\Delta)$ is non-positive.

			Since $\eta$ and $\eta'$ are in the same orbit, it follows that there is a homoemorphism $\varphi\colon S \to S$ such that $\varphi_*(\eta') = \eta$. Thus  the dissection  $\Delta:=\varphi(\Delta')$ on $(S, M, \eta)$ gives rise to a non-positive graded gentle algebra $A(\Delta)$, which is derived equivalent to $ A(\Delta')$.   \end{proof}
		
		\begin{Rem}
			Theorem \ref{Cor:siltingexistence} shows that under Assumption \ref{asumption:silting},  $\mathcal W(S, M, \eta)$ admits a silting object (cf.\ Theorem \ref{Thm:existence_of_silting}); This generalizes  \cite[Theorem 5.7]{CJS}, which gives a partial result on the existence of silting objects. 
		\end{Rem}

		\subsection{The special cases of one boundary component and one stop.}
		\label{section:formA^{g}}
		In this subsection, we study the graded gentle algebras of the special form $(g; 1;0)$ such that $g\geq 1$ and there is only one boundary component and it has only one stop.
		This algebra was introduced in \cite{CJS,CS22}  and plays an important role in the study of the (non-)existence of silting objects.
		
		Note that in this case the quiver $Q(g;1;0)$ is given by 
		\begin{equation}
			\label{equ:An}
			\xymatrix{1 \ar@<1pc>[r]^-{\za_{1}} \ar@<-1pc>[r]^-{\zg_{1}}& 2 \ar[l]_-{\zb_{1}} \ar[r]^-{\zd_{1}}& 3 \ar@<1pc>[r]^-{\za_{2}} \ar@<-1pc>[r]^-{\zg_{2}} & 4 \ar[l]_-{\zb_{2}} \ar[r]^-{\zd_{2}} & \cdots \ar[r]^-{\zd_{g-1}} & 2g-1 \ar@<1pc>[r]^-{\za_{g}}  \ar@<-1pc>[r]^-{\zg_{g}}
				& 2g \ar[l]_-{\zb_{g}}}
		\end{equation}
		and the ideal $I(g;1;0)$ is generated by   $\{\za_{i}\zb_{i}, \zb_{i}\zg_{i}, \zg_{j}\zd_{j}, \zd_{j}\za_{j+1} \mid 1\le i\le g, 1\le j\le g-1\}.$ See Figure \ref{Fig:CanonicalSurfaceModel} for the corresponding surface model.

		\begin{figure}
			{ \tiny \begin{center}
					\begin{tikzpicture}
						\begin{scope}[scale=0.6]
							
							\shadedraw[top color= blue!15] (-30:9.5) arc(-30:210:9.5);
							\node[fill,circle, green, inner sep=0pt, minimum size=5pt] at (0,0){};
							\draw [thick] (210:9.5) ..controls (-6.6,-7).. (-4, -4.75);
							\draw [dashed] (210:9.5) ..controls (-6.5,-5).. (-4, -4.75);
							\draw [dashed] (210:9.5) ..controls (-6.5,-4.5).. (-4, -4.75);
							\draw [thick] (-30:9.5) ..controls (6.6,-7).. (4, -4.75);
							\draw [dashed] (-30:9.5) ..controls (6.5,-5).. (4, -4.75);
							\draw [dashed] (-30:9.5) ..controls (6.5,-4.5).. (4, -4.75);
							\draw [thick] (-4, -4.75) ..controls (0,-1.8).. (4, -4.75);
							
							\begin{scope}[rotate=80]
								\draw [thick,green] (0,0) ..controls (-2.5, 8.5) .. (0,8.5);
								\draw [thick,green] (0,0) ..controls (2.5, 8.5) .. (0,8.5);
								\draw [decoration={markings, mark=at position 0.45 with {\arrow{<}}}, postaction={decorate},thick,gray]  (0,6.5) ..controls (-.5,6.4) and (-.5, 9.6) .. (0,9.5);
								\node at (94:9)  {$s_1$};
								\draw [thick,dashed,gray] (0,6.5) ..controls (.5,6.4) and (.5, 9.6).. (0,9.5);
								\draw[fill=white] (-.8,6.3) ..controls (0,6) .. (.8,6.3);
								\draw [fill=white] (-.8,6.3) ..controls (0,6.6) .. (.8,6.3);
								\draw[fill=white] (-.8,6.3) -- (-.9,6.25);
								\draw[fill=white] (.8,6.3) -- (.9,6.25);
								\draw [decoration={markings, mark=at position 0.45 with {\arrow{<}}}, postaction={decorate},thick,gray]   (0, 6.25) circle (1.2);
								\node at (83:4.8)  {$t_1$};
								\draw [thick,green]  (0,0) ..controls (-.3, 6) .. (0,6.07);
								\draw [thick,green]  (0,0) ..controls (60:5) and (60:9) .. (65:9.5);
								\draw [dashed,green,thick]  (65:9.5) ..controls (70:9) and (3,5.5) .. (.4,6);
								\draw [dashed,green,thick]  (.4, 6) ..controls (.25, 6) .. (0, 6.07);

								\node at (100:2)  {$\gamma_1$};
								\node at (83:2)  {$\beta_1$};
								\node at (67:2)  {$\alpha_1$};
								\node at (53:2)  {$\delta_1$};
							\end{scope}
							\begin{scope}[rotate=20]
								\draw [thick,green] (0,0) ..controls (-2.5, 8.5) .. (0,8.5);
								\draw [thick,green] (0,0) ..controls (2.5, 8.5) .. (0,8.5);
								\draw [decoration={markings, mark=at position 0.45 with {\arrow{<}}}, postaction={decorate},thick,gray] (0,6.5) ..controls (-.5,6.4) and (-.5, 9.6) .. (0,9.5);
								\node at (94:9)  {$s_2$};
								\draw [thick,dashed,gray] (0,6.5) ..controls (.5,6.4) and (.5, 9.6).. (0,9.5);
								\draw[fill=white] (-.8,6.3) ..controls (0,6) .. (.8,6.3);
								\draw [fill=white] (-.8,6.3) ..controls (0,6.6) .. (.8,6.3);
								\draw[fill=white] (-.8,6.3) -- (-.9,6.25);
								\draw[fill=white] (.8,6.3) -- (.9,6.25);
								\draw [decoration={markings, mark=at position 0.45 with {\arrow{<}}}, postaction={decorate},thick,gray]   (0, 6.25) circle (1.2);
								\node at (83:4.8)  {$t_2$};
								\draw [thick,green]  (0,0) ..controls (-.3, 6) .. (0,6.07);
								\draw [thick,green]  (0,0) ..controls (60:5) and (60:9) .. (65:9.5);
								\draw [dashed,green,thick]  (65:9.5) ..controls (70:9) and (3,5.5) .. (.4,6);
								\draw [dashed,green,thick]  (.4, 6) ..controls (.25, 6) .. (0, 6.07);
								\node at (100:2)  {$\gamma_2$};
								\node at (83:2)  {$\beta_2$};
								\node at (67:2)  {$\alpha_2$};
							\end{scope}
							\begin{scope}[rotate=-50]
								\draw [thick,green] (0,0) ..controls (-2.5, 8.5) .. (0,8.5);
								\draw [thick,green] (0,0) ..controls (2.5, 8.5) .. (0,8.5);
								\draw [decoration={markings, mark=at position 0.45 with {\arrow{<}}}, postaction={decorate},thick,gray] (0,6.5) ..controls (-.5,6.4) and (-.5, 9.6) .. (0,9.5);
								\node at (94:9)  {$s_g$};
								\draw [thick,dashed,gray] (0,6.5) ..controls (.5,6.4) and (.5, 9.6).. (0,9.5);
								\draw[fill=white] (-.8,6.3) ..controls (0,6) .. (.8,6.3);
								\draw [fill=white] (-.8,6.3) ..controls (0,6.6) .. (.8,6.3);
								\draw[fill=white] (-.8,6.3) -- (-.9,6.25);
								\draw[fill=white] (.8,6.3) -- (.9,6.25);
								\draw [decoration={markings, mark=at position 0.45 with {\arrow{<}}}, postaction={decorate},thick,gray]  (0, 6.25) circle (1.2);
								\node at (83:4.9)  {$t_g$};
								\draw [thick,green]  (0,0) ..controls (-.3, 6) .. (0,6.07);
								\draw [thick,green]  (0,0) ..controls (60:5) and (60:9) .. (65:9.5);
								\draw [dashed,green,thick]  (65:9.5) ..controls (70:9) and (3,5.5) .. (.4,6);
								\draw [dashed,green,thick]  (.4, 6) ..controls (.25, 6) .. (0, 6.07);
								\node at (100:2)  {$\gamma_g$};
								\node at (83:2)  {$\beta_g$};
								\node at (67:2)  {$\alpha_g$};
							\end{scope}
							\foreach \n in {0, 2, 4}
							{\node at (74-2.5*\n:2) [circle, fill, inner sep=.5pt]{};
								\node at (74-2.5*\n:7) [circle, fill, inner sep=1pt]{};
							}
							\draw[fill=white] (0,0) circle (1.5);
							\draw[fill=gray!15] (0,0) circle (1.5);
							\node[fill,circle, red, inner sep=0pt, minimum size=4pt] at (-90:1.5){};
							\draw [decoration={markings, mark=at position 0.45 with {\arrow{>}}}, postaction={decorate}] (178:1.5)--(177:1.498);
							
							\draw [decoration={markings, mark=at position 0.45 with {\arrow{>}}}, postaction={decorate}] (161:1.5)--(160:1.498);
							\draw [decoration={markings, mark=at position 0.45 with {\arrow{>}}}, postaction={decorate}] (146:1.5)--(145:1.498);
							\draw [decoration={markings, mark=at position 0.45 with {\arrow{>}}}, postaction={decorate}] (131:1.5)--(130:1.498);
							\draw [decoration={markings, mark=at position 0.45 with {\arrow{>}}}, postaction={decorate}] (117:1.5)--(116:1.498);
							\draw [decoration={markings, mark=at position 0.45 with {\arrow{>}}}, postaction={decorate}] (99:1.5)--(98:1.498);
							\draw [decoration={markings, mark=at position 0.45 with {\arrow{>}}}, postaction={decorate}] (85:1.5)--(84:1.498);
							\draw [decoration={markings, mark=at position 0.45 with {\arrow{>}}}, postaction={decorate}] (65:1.5)--(64:1.498);
							\draw [decoration={markings, mark=at position 0.45 with {\arrow{>}}}, postaction={decorate}] (48:1.5)--(47:1.498);
							\draw [decoration={markings, mark=at position 0.45 with {\arrow{>}}}, postaction={decorate}] (32:1.5)--(31:1.498);
							\draw [decoration={markings, mark=at position 0.45 with {\arrow{>}}}, postaction={decorate}] (15:1.5)--(14:1.498);
							
						\end{scope}
					\end{tikzpicture}
			\end{center}}
			\caption{Standard surface model of an algebra of the form $(g;1;0)$, where the arcs in green form an admissible dissection.}
			\label{Fig:CanonicalSurfaceModel}
		\end{figure}

		\begin{Rem}\label{Rem:Agderived} Let $A$ and $A'$ be two graded gentle algebras of the form $(g;1;0)$ with the numbers $a_i,b_i$  and $a_i', b_i'$ as defined in \eqref{align:numbersofthesumgradingsAn}. Note that $c_0=c_0'=-2.$  
			By Proposition \ref{prop:derivedinvariantstandard} if $a_i=a_i'$ and $b_i=b_i'$ then $A$ and $A'$ are derived equivalent. 
		\end{Rem}

		The following lemma shows a special feature (in terms of the surface model) of the algebra of the form $(1;1;0)$, which will be used in the proof of Proposition \ref{Prop:no_silting} below.

		\begin{Lem}\label{Lem:A^{(1)}}
			Let $(S, M, \eta)$ be a graded surface with stops such that $S$ is a torus with exactly one boundary component and with one stop on the boundary (i.e.\ $b=1=\#M$). Then for any  admissible dissection $\zD$ on $S$, the corresponding graded gentle algebra is of the form $(1;1;0)$
			\begin{equation*}
				\xymatrix{1 \ar@<1pc>[r]^-{\za_{1}} \ar@<-1pc>[r]^-{\zg_{1}}& 2 \ar[l]_-{\zb_{1}} }
			\end{equation*}
			and moreover $\w{\mathcal A}(\eta) = \gcd(|\alpha_1|+|\beta_1|-1,|\beta_1|+|\gamma_1|-1) $.
		\end{Lem}
		\begin{proof}
			Let $A=\Bbbk Q/I$ be the graded gentle algebra given by an admissible dissection $\zD$ on $S$. Recall from \cite[Proposition 1.11]{APS19} that $\#\Delta =\#M+b+2g-2$, it follows that $\#Q_0 =\#\zD =2$.  By \cite[Proposition 3.6]{CSS}, we have that $\#M = 2 \#Q_0  - \#Q_1$, which yields $\#Q_1 =3$. Therefore, since $A$ is homologically smooth and proper, it must be of the standard form $(1;1;0)$, as there is only one gentle algebra with two vertices and three arrows (without loops) between the two vertices. 
			
			Denote $\zD=\{\ell_1, \ell_2\}$. Note that $\Delta$ might not be the standard dissection as in Figure \ref{Fig:CanonicalSurfaceModel}.  Nevertheless, cutting the surface $S$ along $\zD$ we obtain a $5$-gon as illustrated in Figure \ref{Fig:5gon}. 
			\begin{figure}[ht]
				\begin{tikzpicture}[scale=0.3, transform shape]
					\begin{scope}
						\foreach \x in {0,1,2,3}
						{
							\draw[thick, green] (\x*72-54:5) to (\x*72+18:5);
						}
						\draw[] (-126:5) to [curve through = {.. (-90:3.5)..}] (-54:5);
						\node[fill,circle, red, inner sep=0pt, minimum size=9pt] at (-90:3.5){};
						\foreach \x in {0, 1,2,3}
						{\begin{scope}[ shift={(\x*72-54:5)},  rotate=\x*72]
								\coordinate (a\x) at (72:1);
							\end{scope}
						}
						\foreach \x in {1,2,3, 4}
						{\begin{scope}[ shift={(\x*72-54:5)},  rotate=\x*72-72]
								\coordinate (b\x) at (-108:1);
								\coordinate (c\x) at (-108:{5*sin(36)});
							\end{scope}
						}
						\draw[white, line width=5pt] (162:5) -- (b3);
						\draw[white, line width=5pt] (162:5) -- (a3);
						\draw[white, line width=5pt] (90:5) -- (b2);
						\draw[white, line width=5pt] (90:5) -- (a2);
						\draw[white, line width=5pt] (18:5) -- (b1);
						\draw[white, line width=5pt] (18:5) -- (a1);
						\foreach \x in {1,2,3}
						{
							\draw[bend left, decoration={markings, mark=at position 0.5 with {\arrow{>}}}, postaction={decorate}] (b\x) to (a\x);
						}
						\node at (18:3.5)  {\huge $\alpha_1$};
						\node at (90:3.5)  {\huge $\beta_1$};
						\node at (162:3.5)  {\huge $\gamma_1$};
						\node[right] at (c1)  {\huge $\ell_1$};
						\node[above] at (c2)  {\huge $\ell_2$};
						\node[above] at (c3)  {\huge $\ell_1$};
						\node[left] at (c4)  {\huge $\ell_2$};
						\draw[decoration={markings, mark=at position 0.35 with {\arrow{>}}}, postaction={decorate}, bend right] (c4) to (c2);
						\draw[decoration={markings, mark=at position 0.35 with {\arrow{>}}}, postaction={decorate}, bend right] (c3) to (c1);
						\node at (1.85,1.6)  {\huge $t'$};
						\node at (-1.85, 1.6)  {\huge $s'$};
					\end{scope}
				\end{tikzpicture}
				\caption{Cutting the torus in Lemma \ref{Lem:A^{(1)}} along an arbitrary admissible dissection $\Delta=\{\ell_1,\ell_2\}$ into a $5$-gon, where the simple closed curves $s', t'$ are homotopic to (the smoothings of) $\ell_2, \ell_1$ respectively.}
				\label{Fig:5gon}
			\end{figure}
			
			Then using the formula in \eqref{align:windingnumbersegments} we  compute the winding numbers of the simple closed curves $s', t'$ in Figure \ref{Fig:5gon}
			\[
			w_\eta(s') = |\alpha_1|+|\beta_1|-1 \quad \text{and} \quad w_\eta(t') =|\beta_1|+|\gamma_1|-1.
			\]
			Therefore, we obtain that $ \w{\mathcal A}(\eta)=\gcd(|\alpha_1|+|\beta_1|-1,|\beta_1|+|\gamma_1|-1) $ since  $s',t'$ form a symplectic basis of $\mathrm{H}_1(\bar S)$, compare with Remark \ref{Remark:linefieldinvariant} (2).
		\end{proof}

		\begin{Rem}
			\label{Rem:notclosedunderderived}
			Note that Lemma \ref{Lem:A^{(1)}} does not hold for $(g;1;0)$ with $g>1$. For instance, on the surface model of the graded gentle algebra of the form $(2;1;0)$ there is an admissible dissection $\Delta$ such that the associated gentle algebra $A(\Delta)$ is given by the following quiver \[
			\xymatrix{1 \ar@<.3pc>[r]^-{\za_{1}}  \ar@<-.3pc>[r]_-{\zb_{1}}  & 2\ar@<.3pc>[r]^-{\za_{2}}\ar@<-.3pc>[r]_-{\zb_{2}}  & 3 \ar@<.3pc>[r]^-{\za_{3}} \ar@<-.3pc>[r]_-{\zb_{3}}
				& 4 \ar@<-.3pc>@/_1.5pc/[lll]_-{\delta}
			}
			\]
			with relations $\{\alpha_1\beta_{2}, \beta_1\alpha_2, \alpha_2\beta_3, \beta_2\alpha_3, \delta \alpha_1, \beta_3\delta\}$.
		\end{Rem}
		
		\section{The behavior of silting objects in surface models}
		Recall that an object $P\in \per(A)$ is called \emph{pre-silting} if $$\Hom_{\per(A)}(P, P[i])=0,\quad \text{for each $i> 0$},$$
		and it is called \emph{silting} if, moreover, $\per(A)=\thick(P)$.  We say that a (pre-)silting object $P$ is {\it basic} if its indecomposable direct summands are  pairwise non-isomorphic. In the sequel, whenever we mention a (pre-)silting object, we always assume it to be  basic.

		\subsection{Silting reduction of graded gentle algebras and surface cuts} In this subsection, we recall from \cite{CJS} some results on the silting reduction of graded gentle algebras, which play an important role in the sequel. Geometrically, performing silting reduction in the perfect derived categories of graded gentle algebras corresponds to cutting the surfaces  along arcs.

		\subsubsection{Silting reduction}
		
		Let $A=\Bbbk Q/I$ be a graded gentle algebra (not necessarily homologically smooth or proper).  Let $e$ be the sum of $n$ vertex idempotents (i.e.\ idempotents corresponding to some vertices in $Q$). Without loss of generality, we may re-label the vertices so that $e=e_{1}+\dots+e_{n}$.   We first recall the following definition.
		
		\begin{Def}[{\cite[Definition 2.8]{CJS}}]\label{definition:left algebra}
			Given an idempotent $e$ as above, we define the graded gentle algebra $A_{e}:=\Bbbk Q_{e}/ I_{e}$  as follows
			\begin{enumerate}[\rm(1)]
				\item The set of vertices of $Q_{e}$ is $Q_{0}\backslash \{1,2,\dotsc, n\}$.
				\item  The arrows of $Q_{e}$ are of the form $[\za_{1}\cdots\za_{k}]:s(\za_{1})\ra t(\za_{k})$, where $\za_{1}\cdots\za_{k}$ is a path of $Q$ with $\za_i \za_{i+1} \in I$ and $t(\za_i) \in \{1,2,\dotsc,n \}$, for $1\leq i \leq k-1$,  and $s(\za_{1}), t(\za_{k})\notin \{1,2,\dotsc,n \}$. In particular, there is an arrow $[\alpha]$ in $Q_{e}$ for each arrow $\alpha$ in $Q$ such that $s(\alpha), t(\alpha) \notin \{1, 2, \dotsc,n\}$. 
				\item  The grading of the arrow $[\za_{1}\cdots\za_{k}] $ is $|[\za_{1}\cdots\za_{k}]|:=\sum_{i=1}^{k}|\za_{i}|-k+1$.
				\item The ideal $I_{e}$ is generated by the set $\{[\za_{1}\cdots\za_{k}][\zb_{1}\cdots\zb_{l}]\mid \za_{k}\zb_{1}\in I  \}$.
			\end{enumerate} 
		\end{Def}
		If $e=0$, we set $A_{e}=A$. The following result is due to \cite{CJS}, which will be used frequently. We refer to \cite{BBD} for the definition of recollements. 
		\begin{Thm}[{\cite[Theorems 2.10, 3.1 and 5.1 (1)]{CJS}}]\label{Thm:siltingreduction} Let $A$ be any graded gentle algebra (not necessarily homologically smooth or proper). Let $e=e_1+\dotsb+e_n$ be the sum of some vertex idempotents of $A$ as above. Then the following hold.
			\begin{enumerate}[\rm(a)]
				\item There is a recollement 
				\begin{equation} \label{mainrecoll}
					\xymatrixcolsep{4pc}\xymatrix{
						\D(A_{e}) \ar[r]^{i_{*}=i_{!}}
						&\D(A) \ar@/_1.5pc/[l]_-{i^{*}} \ar@/^1.5pc/[l]_-{i^{!}}
						\ar[r]^{j^{*}=j^!}
						&\D(eAe) \ar@/^-1.5pc/[l]_-{j_{!}} \ar@/^1.5pc/[l]_-{j_{*}}
				} \end{equation}
				where  $j_{!}=?\ot^{\bf L}_{eAe}eA, \ j^{*}=?\ot_{A}^{\bf L}Ae$ and $j_{*}=\RshHom_{eAe}(Ae, ?)$. 
				\item Assume $A$ is such that $A^{\gg 0}=0$  and $A^0$ is finite dimensional.  Let $P$ be a silting object in  $\thick(eA)$ of $\per(A)$. Then the silting reduction $\per(A)/\thick(P)$ is triangle equivalent to $\per(A_{e})$ and there is a bijection
				\[ \{\text{(pre-)silting objects in $\per(A_{e})$}\} \overset{1:1}{\longleftrightarrow} \{ \mbox{(pre-)silting objects in $\per(A)$  containing $P$ as a direct summand}\}.\]
				\item  Let $P=eA$. Then the graded gentle algebra $A_{e}$ admits the surface model $(S_{P}, M_{P}, \eta_{P}, \zD_{P}\backslash \{ \ell_{i}\}_{i=1}^{n})$ which is obtained by cutting  and contracting $S$ along the arcs $\{ \ell_{i}\}_{i=1}^{n}$. Here, $\ell_i$ corresponds  to the vertex $e_i$.
			\end{enumerate}
		\end{Thm}
		\begin{Rem}\label{Rem:thm5.2b}
			Since $\RshHom_{A}(eA, eA) \simeq eAe$ as dg algebras, it follows from \cite[Subsection 4.2]{Keller94} that there is a triangle equivalence 
			\[
			\thick(eA) \simeq \per(eAe), \quad X \mapsto \RshHom_{A}(eA, X).
			\] 
			In particular, this induces a bijection between (pre-)silting objects in $\thick(eA)$ and in $\per(eAe)$.  
		\end{Rem}

		Let $(S,M,\eta, \zD)$ be a graded surface with stops. Assume that  $\partial_{0}S, \dotsc, \partial_{u}S$ are stopped boundary components and $\partial_{u+1}S, \dotsc, \partial_{u+v}S$ are non-stopped boundary components. Let $\zg$ be an arc connecting a stopped boundary component  (say, $\partial_0 S$) to  a non-stopped boundary component (say, $\partial_{u+v}S$). Let $(S_{\zg}, M_{\zg}, \eta_{\zg})$ be the graded surface with stops obtained by cutting along $\zg$. It is clear that the genus of $S$ and $S_\zg$ are the same (say, $g$). 
		
		Recall that we denote by $a_i, b_i$ (resp. $\w{a_i},\w{b_i}$) the winding numbers of the simple closed curves $s_i$ and $t_i$  (see Figure \ref{Fig:StandardSurfaceModel}) of $S$ (resp. $S_{\zg}$), respectively. Assume that the arc $\gamma$ does not intersect $s_i, t_i$. As an application of silting reductions, we have the following topological result. 
		
		\begin{Prop}\label{prop:cutarcnonstop}
			The cut surface $(S_{\zg}, M_{\zg}, \eta_{\zg})$  has $u+1$ stopped boundary components with the same number of stops as in $\{\partial_0S, \dotsc, \partial_uS\}$ and $v-1$ non-stopped boundary components $\{\partial_{u+1} S, \dotsc,  \partial_{u+v-1}S\}$. We have $w_{\eta_{\zg}}(\partial_{0}S)= w_{\eta}(\partial_{0}S)+ w_{\eta}(\partial_{u+v}S)+2$ and $w_{\eta_{\zg}}(\partial_{i}S)= w_{\eta}(\partial_{i}S)$  for all $1\le i \leq u+v-1$. Moreover, $a_i=\w{a_i}$ and $b_i=\w{b_i}$ for all $1\le i\le g$.
		\end{Prop}
		\begin{proof}
			We extend $\gamma$ into a standard dissection $\Delta$ as in Figure \ref{Fig:StandardSurfaceModel} whose corresponding algebra $A(\Delta)$ is given as in Subsection \ref{subsection:standardform} so that the rightmost vertex (denoted by $e_\gamma$) corresponds to $\gamma$. Then  it follows from Theorem \ref{Thm:siltingreduction} that the cut surface $S_\gamma$ along $\gamma$ also admits a standard dissection $\Delta_\gamma$ whose corresponding algebra $A(\Delta_\gamma) \simeq A(\Delta)_{e_\gamma}$  is just obtained from $A(\Delta)$ by removing the vertex $e_\gamma$ together with the incident arrows (The degree of each arrow does not change). By Formulas \eqref{eq:winding-boundaryi} and \eqref{equ:C} we have $w_{\eta_{\zg}}(\partial_{i}S)= w_{\eta}(\partial_{i}S)$   for all $1\le i \leq u+v-1$ and by Formula \eqref{equ:HP} we have $w_{\eta_{\zg}}(\partial_{0}S)= w_{\eta}(\partial_{0}S)+ w_{\eta}(\partial_{u+v}S)+2.$ Moreover, Formula \eqref{align:numbersofthesumgradingsAn} implies that $a_i=\w{a_i}$ and $b_i=\w{b_i}$ for all $1\le i\le g$.
		\end{proof}
		
		\subsection{Silting objects in the surface model}
		In this subsection, we show a remarkable property of silting objects in the perfect derived category of a graded gentle algebra in terms of the surface model. 
		Namely, in Proposition \ref{Prop:behaviorsiltingobjects} we show:
		\begin{center}{\it Silting objects give rise to admissible dissections.} \end{center} 
		This is one of our motivations to study silting objects for graded gentle algebras and it will play an essential role in the proof of the complete derived invariant in Theorem \ref{Thm:derivedeq}. 
		In the case of an ungraded finite dimensional gentle algebra, Proposition \ref{Prop:behaviorsiltingobjects} was shown in \cite[Theorem 5.2]{APS19}.
		
		We first give a useful observation.
		\begin{Lem}\label{Lem:siltingwinding}
			Let $A$ be a homologically smooth graded gentle algebra with surface model $(S, M, \eta, \zD)$. Assume $\per(A)$ admits a silting object $P$, then any non-stopped boundary component (if it exists) has non-negative winding number. 
		\end{Lem}
		\begin{proof}
			We may assume $A$ is of the standard form $(g;m_0,\dots,m_u;v)$ up to derived equivalence. Denote the non-stopped boundary components by $\partial_{u+1}S,\dots, \partial_{u+v}S$ which correspond to loops $y_1,\dots, y_v$ in $A$, respectively. By \eqref{equ:C}, we have
			\[
			w_\eta(\partial_{u+j}S) = -|y_j| \quad \text{for $1\leq j \leq v$}.
			\]
			If $|y_j|>0$ for some $j=1,\dots, v$. Write $k=s(y_j)=t(y_j)$. Then we have
			\begin{align}\label{align:nonnegativenonstop}
				\Hom_{\per(A)}(e_kA,e_kA[n])\simeq \h^{n}(e_kAe_k) \simeq \h^n(\Bbbk [y_j]) \neq 0 \quad \text{for $n \gg 0$}.
			\end{align}
			But on the other hand, $\per(A)$ has a silting object $P$,  it follows that  $\Hom_{\per(A)}(P,P[i])=0$ for each $i >0$ and $\per(A)=\thick(P)$. Then for any $X\in \per(A)$ we have $$\Hom_{\per(A)}(X,X[n])=0 \quad \text{for $n\gg 0$}, $$ since any $X \in \per(A)$ can be built out of $A$ by taking direct summands, extensions and shifts. In particular, we have $\Hom_{\per(A)}(e_kA, e_kA[\gg 0])=0$, which is a contradiction to  \eqref{align:nonnegativenonstop}. So $|y_j|\le 0$ for each $j=1,\dots, v$ and the assertion holds. 
		\end{proof}
		
		\begin{Prop}\label{Prop:behaviorsiltingobjects}
			Let $A$ be a homologically smooth graded gentle algebra with surface model $(S, M, \eta, \zD)$. Let $P$ be a basic silting object in  $\per(A)$. Then the graded arcs corresponding to $P$ form an admissible dissection in $(S, M, \eta)$.
		\end{Prop}

		\begin{proof} 
			Let  $P=\bigoplus_{i=1}^n P_i$ be a basic silting object in $\per(A)$, where each $P_i$ is indecomposable.  By the classification of indecomposable objects in $\per(A)$, see \cite[Theorem 4.3]{HKK17}, each $P_i$ corresponds to a curve $\ell_i$ in $S$ such that $\ell_i$ is either an  arc (with endpoints in $M$) or a primitive closed curve in $S\backslash \partial S$. We divide the proof into the following two steps.
			
			\emph{Step 1:} We show that $\{\ell_i\}_{i=1}^n$ forms an admissible collection.
			
			First we claim that $\ell_i$ cannot be a primitive closed curve for $1\leq i \leq n$. Indeed, if $\ell_i$ is a primitive closed curve then $\Hom_{\per(A)}(P_i, P_i[1]) \neq 0$, see \cite[Theorem 3.3]{OPS18} and Proposition \ref{Prop:simpleclosedcurves}, which contracts to the assumption that $P_i$ is a pre-silting object.
			
			Next we show that $\ell_i$ has no self-intersections or does not intersect $\ell_j$ in the interior of $S$.  For contradiction suppose that there is a (self-)intersection.  Then by \cite[Theorem 4.11]{QZZ22} and also \cite[Theorem 3.3]{OPS18}, we obtain that 
			$$
			\Hom_{\per(A)}(P_i, P_j[k_{i,j}]) \neq 0 \neq \Hom_{\per(A)}(P_j, P_i[1-k_{i,j}]) \quad \text{for some $k_{i,j} \in \mathbb Z$},
			$$ 
			where $k_{i,j}$ is the grading induced by the corresponding (self)-intersection. 
			We infer that $P_i \oplus P_j$  or $P_i$ if $i=j$ is not a pre-silting object. As a result, we obtain that the $\{\ell_i\}_{i=1}^n$ are pairwise non-intersecting (in the interior of $S$) arcs on $S$.

			By Remark \ref{Rem:admissiblecollection} it remains to show that these arcs do not enclose a subsurface  without stops. 
			We show this by using a similar argument to the one in the proof of \cite[Lemma 3.6]{APS19} for the ungraded case.  More precisely, suppose for contradiction that there are arcs among $\{\ell_i\}_{i=1}^n$ which do enclose a subsurface $S'$ without stops. Without loss of generality, we may assume that these arcs are $\ell_{1},\dotsc, \ell_{s}$.  Assume that $S'$ is of genus $g'$ with $b'=b_1'+b_2'$ boundary components, where $b_1'$ is the number of boundaries formed by these arcs and $b_2'$  represents the number of non-stopped boundaries of $S'$ that originally belong to $S$.
			Let us re-index these arcs in such a way that the $j$-th ($1\leq j \leq b_1'$) boundary component of $S'$ consists of the ordered arcs $\ell_{j,1},\cdots, \ell_{j, s_{j}}$   and let $\zd_{j}$ be the (smoothing) concatenation of the $\ell_{j,i}$, for $1\le i\le s_{j}$.  
			
			Let $a_{j, i}$ be the integer such that $\Hom_{\per(A)}(\ell_{j,i},\ell_{j,i+1}[a_{j,i}])\not=0$,  for $1\le j\le b'_1$  and $1\le i\le s_{j}$, where we set $\ell_{j, s_j+1}=\ell_{j, 1}$. We claim that at least one of the $a_{j,i}$'s must be strictly positive which contradicts the assumption that $P$ is a pre-silting object. Hence, the arcs in $\{l_i\}_{i=1}^n$ do not enclose a subsurface  without stops and they form an admissible collection.
			
			To prove the claim, by 
			\cite[Lemma 3.19]{OZ}, for $1\leq j \leq b_1'$ we have 
			\[ w_\eta(\delta_{j})=s_{j}-\sum_{i=1}^{s_{j}}a_{j,i}.\] 
			On the other hand,  by the Poincar\'e-Hopf index formula we have 
			\[\sum_{j=1}^{b'}w_\eta(\delta_{j})=4-4g'-2b'. \]
			Combining the above two equalities we obtain 
			\begin{align}\label{align:equalityaij}
				\sum_{j=1}^{b_1'}\sum_{i=1}^{s_{j}}a_{j, i}=\sum_{j=1}^{b_1'}s_{j}+4g'+2b'-4+\sum_{l=1}^{b_2'} w_\eta(\delta_{l+b_1'}).  
			\end{align}
			where the $w_\eta(\delta_{l+b_1'})$ are the winding numbers of the boundary components on $S'$ which come from non-stopped boundary components of $S$. By our assumption that there exists a silting object, it follows by Lemma \ref{Lem:siltingwinding} that $w_\eta(\delta_{l+b_1'})\geq 0$ for $1\leq l \leq b_2'$. Note that $b_1'> 0$ and  $s_j > 0$ for $1\leq j \leq b_1'$. 
			
			If $g'\ge 1$ or if $g'=0$ and $b'\ge 2$, then from \eqref{align:equalityaij} we have 
			$$\sum_{j=1}^{b_1'}\sum_{i=1}^{s_{j}}a_{j, i}>0$$ and thus the claim holds.  If  $g'=0$ and $b'=b_1'+b_2'=1$, then $b_1' = 1$ and $b_2'=0$. It follows that the equality \eqref{align:equalityaij} becomes  $$\sum_{i=1}^{s_{1}}a_{1,i}=s_{1} -2.$$ 
			Note that in this case $s_{1}$ cannot be $1$ or $2$, since otherwise  $S'$ would be a contractible  monogon or a digon given by two homotopic arcs, respectively. This yeilds $s_1>2$. Therefore,  $\sum_{i=1}^{s_{1}}a_{1,i}>0$ and thus the claim  holds in this case.  This shows $\{\ell_i\}_{i=1}^n$ forms an admissible collection.

			\emph{Step 2:} We show  $\{\ell_i\}_{i=1}^n$ forms an admissible dissection.
			
			If it is not the case, one can complete it to an admissible dissection since by Step 1, $\{\ell_i\}_i$ forms an admissible collection. Denote by $A'$ the associated graded gentle algebra  whose grading is induced by $\eta$. Then by \cite{HKK17}, there is a triangle equivalence $F:\per(A)\xra{\simeq} \per(A')$, and moreover, $F(P)\cong e'A'$ where $e'$ is the sum of the idempotents of $A'$ corresponding to $\{\ell_i\}_{i=1}^n$. Since $\thick(P)=\per(A)$ it follows that $\thick(e'A') = \per(A')$ and thus $e'=1_{A'}$. The latter means that $\{\ell_i\}_{i=1}^n$ is already an admissible dissection on $S$. 
		\end{proof}

		\begin{Rem}\label{Rem:siltingdissection}
			\begin{enumerate}[\rm(1)]
				\item Assume that $A$ is a homologically smooth and proper gentle algebra. Then any pre-silting object in $\per(A)$ corresponds to an admissible collection by the first step of the proof of Proposition  \ref{Prop:behaviorsiltingobjects}. Note that in this case, there is no non-stopped boundary components and thus the last summand in \eqref{align:equalityaij} absents (so Lemma \ref{Lem:siltingwinding} is not needed). But in general, if there is a non-stopped boundary component whose winding number is negative, then pre-silting objects may not give admissible collections.
				\item Proposition \ref{Prop:behaviorsiltingobjects} shows that any silting object $P$ in $\per(A)$ gives a formal generator in the form of a full formal arc system in  the partially wrapped Fukaya category associated to the graded surface with stops $(S, M, \eta)$ as defined in \cite{HKK17}. Namely, the dg endomorphism algebra of $P$ is quasi-isomorphic to the non-positively graded gentle algebra corresponding to the admissible dissection given by $P$. 
			\end{enumerate}
		\end{Rem}

		\section{Existence of silting objects in $\per(A)$}\label{Sec:silting}
		The aim of this section is to provide a complete answer to the existence of silting objects in $\per(A)$ for any homologically smooth graded gentle algebra $A$.

		In general, it is not easy to check whether $\per(A)$ has a silting object or not. But in case $A$ is a homologically smooth graded gentle algebra, we give a complete answer in this section by using the surface model.
		Recall from Section \ref{section:formA^{g}} the graded gentle algebras of the form $(g;1;0)$ with winding numbers $a_i, b_i$.
		
		The main  result in this section is the following.

		\begin{Thm}\label{Thm:existence_of_silting}
			Let $A=\Bbbk Q/I$ be a homologically smooth graded gentle algebra.  Then $\per(A)$ admits silting objects if and only if the following two conditions hold.
			\begin{enumerate}[\rm(1)]
				\item If the set $C(A)$ of cycles without relations  is nonempty, then  for any  $p\in C(A)$, we have $|p|\le 0$. \label{conditiona'}
				
				\item As a graded algebra, $A$ is not isomorphic to a graded gentle algebra given by the quiver
				\begin{align*}
					\xymatrix{ & 1 \ar@<1pc>[r]^-{\za} \ar@<-1pc>[r]^-{\zg}& 2 \ar[l]_-{\zb}} 
				\end{align*}
				with relations $\{\za\zb, \zb\zg  \}$ and satisfying  $|\za|+|\zb|=1$ and $|\zb|+|\zg|=1$.  \label{conditionb'} 
			\end{enumerate}
		\end{Thm}

		Theorem \ref{Thm:existence_of_silting} follows from Theorem \ref{Cor:siltingexistence}
		and Propositions \ref{Prop:no_silting} and \ref{Prop:existence_of_siltingfornonproper} below. (The contrapositive of) Theorem \ref{Thm:existence_of_silting} can be re-stated in terms of partially wrapped Fukaya categories as follows.
		\begin{Cor} \label{Cor:siltinginFukaya}
			Let ${\cal W} (S,M, \eta)$ be the partially wrapped Fukaya category of the graded surface with stops $(S,M, \eta)$.  Then  ${\cal W} (S,M, \eta)$  does not admit silting objects if and only if one of the following holds. 
			\begin{enumerate}[\rm(a)]
				\item There exists a non-stopped boundary component $\partial_iS$ (i.e.\ $\partial_i S \cap M = \emptyset$)  whose winding number is negative. 
				\item $S$ is the torus with one boundary component and one stop (i.e.\ $\# M =1$) such that the winding number of each non-separating simple closed curve is zero. 
			\end{enumerate}
		\end{Cor}
		
		Recall that $C(A)=\emptyset$ if and only if $A$ is proper (see Remark \ref{rem:nonproper}), and in this case, $\per(A)$ has silting objects if and only if Condition \eqref{conditionb'} in Theorem \ref{Thm:existence_of_silting} holds.
		As a direct consequence, we have the following result for homologically smooth and proper graded gentle algebras.
		
		\begin{Cor}\label{Cor:existence_of_silting}
			Let $A=\Bbbk Q/I$ be a homologically smooth and proper graded gentle algebra and let $(S,M,\eta)$ be the corresponding graded surface model of $A$. Then the following are equivalent.
			\begin{enumerate}[\rm(i)]
				\item $\per(A)$ has no (non-zero pre-) silting object.
				\item $A$ is isomorphic, as a graded algebra,  to an algebra of the form $(1;1;0)$ with $a_{1}=0=b_{1}$.
				\item $S$ is the torus with one boundary and one stop on that boundary such that any non-separating simple closed curve on $S$ has winding number zero.
			\end{enumerate}
		\end{Cor}
		
		\begin{Ex}\label{Ex:torustwostops}
			Let $A$ be the graded gentle algebra given by the quiver 
			\[\xymatrix{
				1 \ar@<1pc>[r]^-{\za} \ar@<-1pc>[r]^-{\zg} & 2 \ar[l]_{\zb} \ar[r]^\theta & 3
			}\]
			with relations $I=\langle \za\zb,\zb\zg,\zg \theta\rangle$ and gradings $|\za|=|\zb|=|\theta|=0$, $|\gamma|=1$. Note that the surface model of $A$ is a torus with one boundary component such that $\#M =2$ (see Figure \ref{Fig:StandardSurfaceModel}).  It follows from Corollary \ref{Cor:existence_of_silting} that $\per(A)$ has a silting object, although $\w{A}(\eta)=0$ by \eqref{align:wAA1}.
		\end{Ex}

		\begin{Rem}
			We have seen in Proposition \ref{prop:derivedinvariantstandard} that all the algebras of the form $(1;1;0)$ with $a_1=0=b_1$ are derived equivalent. Since the existence of silting objects is invariant under derived equivalence, Corollary \ref{Cor:existence_of_silting} yields that among the derived equivalence classes of homologically smooth and proper graded gentle algebras, there is only one class, namely the one corresponding to algebras of the form $(1;1;0)$ with $a_1=0=b_1$, whose perfect derived category does not admit silting objects.  
		\end{Rem}

		\subsection{The proof of Theorem \ref{Thm:existence_of_silting}}\label{subsection:prooftheorm42}
		
		Let $A$ be a graded gentle algebra with graded surface model $(S, M,\eta, \zD)$. 
		Recall from Assumption  \ref{asumption:silting} that we have considered the following two assumptions on $(S, M, \eta)$, under which there exist silting dissections on $\mathcal W(S, M, \eta)$:
		\begin{enumerate}[\rm (1)]
			\item Assume that the winding number of each non-stopped boundary component (if it exists) of $S$ is non-negative.
			\item Assume that $\widetilde{A}(\eta) \neq 0$ if  $S$ is a torus with exactly one boundary component and one stop.
		\end{enumerate}
		Condition \eqref{conditiona'} (resp.\ \eqref{conditionb'})  in Theorem \ref{Thm:existence_of_silting} is equivalent to the first (resp.\ second) assumption. The former equivalence is due to  Remark \ref{rem:nonproper}, which states that the set of non-stopped boundaries bijectively corresponds to the set $C(A)$ and the winding number of the non-stopped boundary component  corresponding to $p\in C(A)$ is $-|p|$. The latter equivalence is explained as follows: if $S$ is a torus with exactly one boundary and one stop, by \eqref{align:widetildeAinvariant}, we have
		\[\widetilde{A}(\eta)=\rm{gcd}(|\za|+|\zb|-1,|\zb|+|\zg|-1)\not=0. \]

		Therefore,  the `if' part of Theorem \ref{Thm:existence_of_silting} follows directly from Theorem \ref{Cor:siltingexistence}, because in this case, there is a dissection $\zD'$ whose corresponding graded gentle algebra $A(\zD')$ is non-positive graded. The image of $A(\zD')$ under the equivalence $\per(A(\zD'))\simeq \per(A)$ gives a silting object in $\per(A)$. In the remainder of the section we prove the \lq only if' part of Theorem \ref{Thm:existence_of_silting}. We consider the two cases: proper and non-proper.
		
		\subsubsection{The  proper case}
		\label{subsubsection:smoothproper}
		The following result completes the proof of Theorem \ref{Thm:existence_of_silting} for the proper case.

		\begin{Prop}\label{Prop:no_silting}
			Let $A$ be the algebra of the form $(1;1;0)$ with $a_1=0=b_1$. Then $\per(A)$ has no nonzero pre-silting objects. In particular, $\per(A)$ has   no silting objects.
		\end{Prop}
		
		\begin{proof}
			Assume $\per(A)$ has a pre-silting object $P$, then by Remark \ref{Rem:siltingdissection}(1), $P$ gives an admissible collection in  the graded surface model $(S,M,\eta,\zD)$ of $A$, which can be completed into an admissible dissection $\zD'=\{1',2'\}$.  Assume that $1'$ gives rise to a direct summand $P'$ of $P$. By Lemma \ref{Lem:A^{(1)}}, the graded gentle algebra $A(\zD')$ is of the form
			\[ \xymatrix{ & 1' \ar@<1pc>[r]^-{\za'} \ar@<-1pc>[r]^-{\zg'}& 2' \ar[l]_-{\zb'}} \]
			with relations $\za'\zb'=0=\zb'\zg'$,  and moreover, 
			\[\widetilde{A}(\eta)=\rm{gcd}(|\za'|+|\zb'|-1,|\zb'|+|\zg'|-1). \ \]
			Since \[\h^n(e_{1'}A(\zD')e_{1'})=\Hom_{\per(A')}(e_{1'}A(\zD'),e_{1'}A(\zD')[n])=\Hom_{\per(A)}(P',P'[n])=0\]
			for $n>0$ and the path $\zg'\zb'$ is a non-zero element in $\h^{|\zg'|+|\zb'|}(e_{1'}A(\zD')e_{1'})$, it follows that 
			$|\zg'|+|\zb'|\le 0$ and  thus $\widetilde{A}(\eta)\not=0$, which contradicts to $\widetilde{A}(\eta) =0$. So $\per(A)$ has no pre-silting objects. 
		\end{proof}

		\begin{Rem}
			We note that a direct  calculation  as stated in \cite{CJS}, also  shows  Proposition~\ref{Prop:no_silting} by showing that there always is a positive self-extension for any graded homotopy string or band complex in $\per(A)$ for $A=(1;1;0)$ with $a_1=0=b_1$. A recent preprint  \cite{LZ23} also announces a  proof using  yet a different approach.  \end{Rem}

		\subsubsection{The non-proper case} In this subsection, we show the following result, which completes the proof of Theorem \ref{Thm:existence_of_silting} in the non-proper case. 
		\begin{Prop}\label{Prop:existence_of_siltingfornonproper}
			Let $A=\Bbbk Q/I$ be a homologically smooth graded gentle algebra which is  non-proper (that is, $C(A)\not=\emptyset$). 
			Assume  $\per(A)$ admits a silting object. Then $|p|\le 0$ for each $p\in C(A)$.
		\end{Prop}
		\begin{proof}
			By Remark \ref{rem:nonproper}, any $p\in C(A)$  corresponds to a non-stopped boundary component, whose winding number is exactly $-|p|$. Then the assertion follows directly from Lemma \ref{Lem:siltingwinding}.
		\end{proof}

		\begin{Rem}
			We may also show Proposition \ref{Prop:existence_of_siltingfornonproper} using Hochschild homology. If $\per(A)$ admits a silting object $T$ then the dg endomorphism algebra $B$ of $T$ is derived equivalent to $A$. It follows from \cite{Keller06} that $\mathrm{HH}_*(A, A) \simeq \mathrm{HH}_*(B, B)$. 
			
			On one hand, since the cohomology of $B$ is concentrated in non-positive degree we know that $\mathrm{HH}_{i}(B, B)  =0$ for $i<0,$ see e.g.\ \cite{Keller06}. On the other hand, from \cite[Chapter 4]{CSSS}  each cyclic path $p \in C(A)$ gives a nonzero cocycle $[p] \in \mathrm{HH}_{-|p|}(A, A)$. In particular, if there is $p \in C(A)$ with $|p|> 0$ then $\mathrm{HH}_{-|p|}(A, A) \neq 0$. This is a contradiction. 
		\end{Rem}
		
		In the following two sections, we will provide two applications (Theorems \ref{Thm:derivedeq} and \ref{Thm:partialsiltingmainresult}) of our main Theorem \ref{Thm:existence_of_silting}.

		\section{A complete derived invariant for graded gentle algebras} \label{Section:derivedinvariant}
		
		In this section we give a   first application of Theorem \ref{Thm:existence_of_silting}. Namely, we  give  a complete derived invariant for graded gentle algebras.
		This invariant is given in terms of their geometric surface models. By the equivalences in \cite{HKK17} and \cite{LP20} this  then also gives a complete invariant of triangle equivalences for partially wrapped Fukaya categories of graded surfaces with stops. 
		
		More precisely, we have the following result.

		\begin{Thm}\label{Thm:derivedeq}
			Let $A$ and $B$ be two homologically smooth graded gentle algebras with associated surface models  $(S_{A}, M_{A}, \eta_{A},\zD_{A})$ and
			$(S_{B}, M_{B}, \eta_{B}, \zD_{B})$,  respectively.

			Then $\D(A)$ is triangle equivalent to $\D(B)$ if and only if there exists an orientation preserving homeomorphism $\varphi: S_{A}\ra S_{B}$ such that $\varphi(M_{A})=M_{B}$  and $\varphi_{*}(\eta_{A})$ is homotopic to $\eta_{B}$.
		\end{Thm}
		
		The \lq if' part was proved in \cite{LP20}, while the `only if' implication was proposed as a conjecture in 
		\cite[Remark 3.19]{LP20}. 
		To establish the `only if' part, we examine two distinct cases based on the presence or absence of silting objects; refer to Theorem \ref{Thm:existence_of_silting}.

		\subsection{The case where silting objects exist}
		\begin{Prop}\label{prop:properlp0}
			The \lq only if' part of Theorem \ref{Thm:derivedeq} holds if $\per(A)$ admits silting objects. 
		\end{Prop}
		\begin{proof}
			Assume that $F:\D(A) \ra \D(B)$ is a triangle equivalence. Assume that $\per(A)$ admits a silting object $P$. We may assume that $P$ is basic. Then $F(P)$ is also a basic silting object of $\per(B)$. By Proposition~\ref{Prop:behaviorsiltingobjects},
			$P$ and $F(P)$ give rise to admissible dissections $\zD_{P}$ of  $(S_{A}, M_{A}, \eta_A)$ and  $\zD_{F(P)}$ of $(S_{B}, M_{B}, \eta_B)$, respectively. Let $A'$ (resp.\ $B'$) be the graded gentle algebra corresponding to the dissection $\zD_P$ (resp.\ $\zD_{F(P)}$).  Then  we have algebra isomorphisms 
			\[A' \cong \bop_{i\in\Z}\Hom_{\per(A)}(P, P[i]) \cong \bop_{i\in\Z}\Hom_{\per(B)}(F(P), F(P)[i])\cong B',\]
			where the first and third isomorphisms are due to Remark \ref{Rem:siltingdissection}(2), and the second one is induced by the triangle equivalence $F$.

			Note that $(S_A, M_A, \eta_A, \Delta_P)$ (resp.\ $(S_B, M_B, \eta_B, \Delta_{F(P)})$) is a surface model of $A'$ (resp.\ $B'$). Since $A'$ is isomorphic to $B'$, by Proposition \ref{Prop:models_same_algebra} there exists an orientation preserving homeomorphism $\varphi: S_{A}\ra S_{B}$ such that $\varphi(M_{A})=M_{B}$, $\varphi(\zD_{P})=\zD_{F(P)}$,  and $\varphi_{*}(\eta_{A})$ is homotopic to $\eta_{B}$.
		\end{proof}

		\subsection{The case where silting objects do not exist}
		In this subsection, 
		we show the \lq only if' part of Theorem \ref{Thm:derivedeq} for the case $\per (A)$ has no silting objects. Note that in this case, $\per(B)$ also has no silting objects. Assume that $S_A$ and $S_B$ have $v$ and $v'$ non-stopped boundary components, $u+1$ and $u'+1$ stopped boundary components,  respectively. For the convenience of the readers, we present our argument in the following two propositions.   Recall that $A$ is proper if and only if $v=0$.
		
		\begin{Prop}\label{prop:properlp}
			The \lq only if' part of Theorem \ref{Thm:derivedeq} holds if $A$ is proper and $\per(A)$ has no silting objects.
		\end{Prop} 
		
		\begin{proof}
			Since $A$ and $B$ are proper, by Theorem \ref{Thm:existence_of_silting}, $A$ and $B$ must be of the form $(1;1;0)$ with $a_A=b_A = a_B=b_B=0$. Let $(S, M, \eta, \zD)$ and $(S', M', \eta',\zD')$ be the standard surface models of $A$ and $B$, respectively. Then $S'=S, M'=M, \zD' = \zD$. 
			
			Since $\w{\mathcal A}(\eta)=0=\w{\mathcal A}(\eta')$, 
			the winding number of any non-separating simple closed curve is zero (see Remark \ref{Remark:linefieldinvariant}(1)). Note that there is only one boundary component on $S$ (and $S'$) and its winding number is $-2$. 
			Then  $\eta$ is homotopic to $\eta'$ (see for example  \cite[Proposition 1.4]{APS19} or  \cite{C72}). It follows that the identity map $\rm{id}\colon S\to S'$ gives a homeomorphism between these two surface models. By Proposition \ref{Prop:models_same_algebra},  there exists an orientation preserving homeomorphism $\varphi^{A}: S_{A}\ra S$ (resp.\
			$\varphi^{B}: S'\ra S_{B}$) such that $\varphi^{A}(M_{A})=M$ (resp.\ $\varphi^{B}(M')=M_{B}$) and $\varphi^{A}_{*}(\eta_{A})$ is homotopic to $\eta$ (resp.\ $\varphi^{B}_{*}(\eta')$ is homotopic to $\eta_{B}$).
			The  homeomorphism $\varphi=\varphi^{B}\circ \rm{id} \circ \varphi^{A}$ satisfies the condition we want.
		\end{proof}
		
		\begin{Prop}\label{prop:properlp2}
			The \lq only if' part of Theorem \ref{Thm:derivedeq} holds if $A$ is non-proper and $\per(A)$ has no silting objects.
		\end{Prop} 
		\begin{proof}
			We apply induction on the number $v$ of non-stopped boundary components of $S_A$.
			If $v=0$, then the assertion holds by Propositions \ref{prop:properlp0} and  \ref{prop:properlp}.
			
			Now assume that the assertion holds for $v=k-1\ge 0$. We show that it is also true for $v=k$ by comparing the geometric invariants  considered in \cite{LP20}.
			Up to derived equivalence, we may assume that $A=kQ/I$ is of standard form $(g; m_0, \dotsc, m_u; v=k)$ (see Definition \ref{Def:standardform}), and $\zD_A$ is the standard dissection. 
			Denote by $\zg$ the arc connecting the boundary component $\partial_0 S_A$ and the non-stopped boundary component $\partial_{u+v}S_A$ as illustrated in Figure \ref{Fig:StandardSurfaceModel}, and denote by $e\in A$ the corresponding idempotent.  
			Let  $\zg'$ be the arc in $S_B$ which is the image of $\zg$ under the equivalences 
			\[\mathcal W(S_A, M_A,\eta_A)\simeq \per(A)\simeq \per(B) \simeq \mathcal W(S_B,M_B,\eta_B).\]
			Then $\End(\zg')=\End(\zg)=\Bbbk[x]$, and by Corollary \ref{Cor:kxcurve}, $\zg'$ connects a stopped boundary component and a non-stopped boundary component. We may extend $\zg'$ to an admissible dissection $\zD'$ of $S_B$, and without loss of generality, we may assume that 
			$\zD'$ is of standard form as in Figure \ref{Fig:StandardSurfaceModel} and $\zg'$ is exactly the arc connecting $\partial_0 S_B$ and $\partial_{u'+v'} S_B$.  Note that by Proposition \ref{Prop:kxcurve}, 
			\begin{equation}\label{equ:u+v=u'+v'} 
				w_{\eta_A}(\partial_{u+v}S_A)=w_{\eta_B}(\partial_{u'+v'}S_B)=-|x|.
			\end{equation}

			We denote by $B'=kQ'/I'$ the standard graded gentle algebra given by $\zD'$, and denote by $e'\in B$ the corresponding idempotent of $\zg'$. Then it follows from \cite{HKK17} that $\per(B)$ is triangle equivalent to $\per(B')$ and thus by silting reduction there are triangle equivalences  
			\[\per(A_e)\simeq \per(A)/\thick(eA) \simeq \per(B')/\thick(e'B')\simeq \per(B'_{e'}).\]
			On the other hand, note that the surface model $(S_{A_e}, M_{A_e}, \eta_{A_e})$ (resp. $(S_{B'_{e'}}, M_{B'_{e'}}, \eta_{B'_{e'}})$) of $A_e$ (resp. $B'_{e'}$) has $v-1$ (resp $v'-1$) non-stopped boundary components, and $u+1$ (resp. $u'+1$) stopped boundary components. By the induction hypothesis, the two surface models $(S_{A_e}, M_{A_e}, \eta_{A_e})$ and $(S_{B'_{e'}}, M_{B'_{e'}},\eta_{B'_{e'}})$ are isomorphic. So we have $v=v', u=u'$. 
			
			Denote the sequences \eqref{align:allthewindingnumbers} of winding numbers associated to $S_A, S_{A_e}, S_B, S_{B'_{e'}}$ by 
			\begin{align*}
				W_{A} &= \{ c^A_0, \dotsc, c^A_{b-1}, a^A_1, \dotsc, a^A_g, b^A_1, \dotsc, b^A_g\}\\ 
				W_{A_e} & = \{ c_0^{A_e}, \dotsc, c_{b-2}^{A_e}, a_1^{A_e}, \dotsc, a_g^{A_e}, b_1^{A_e}, \dotsc, b_g^{A_e}\}\\ 
				W_{B'} & = \{ c^{B'}_0, \dotsc, c^{B'}_{b-1}, a^{B'}_1, \dotsc, a^{B'}_g, b^{B'}_1, \dotsc, b^{B'}_g\}\\ 
				W_{B'_{e'}} & = \{ c^{B'_{e'}}_0, \dotsc, c^{B'_{e'}}_{b-2}, a^{B'_{e'}}_1, \dotsc, a^{B'_{e'}}_g, b^{B'_{e'}}_1, \dotsc, b^{B'_{e'}}_g\}.
			\end{align*}
			Since the two surface models $S_{A_e}$ and $S_{B'_{e'}}$ are isomorphic, it follows that for $0\leq i \leq b-2$
			\begin{align}\label{align:cutting74}
				c_i^{A_e} = c_i^{B'_{e'}}.
			\end{align}
			By Proposition \ref{prop:cutarcnonstop} we have 
			\begin{equation}\label{equ:A_eA2}
				c_{0}^{A_e}=c_{0}^{A}+c_{b-1}^{A}+2, \quad \quad c_{0}^{B'_{e'}} = c_{0}^{B'}+c_{b-1}^{B'} + 2,
			\end{equation}
			and for $1\leq i \leq b-2$
			\begin{equation}\label{equ:A_eA1}
				c_i^{A_e}= c_i^{A}, \quad \quad c_i^{B'_{e'}}= c_i^{B'}.
			\end{equation}
			Moreover, it follows from Proposition \ref{prop:cutarcnonstop} that
			\begin{equation}\label{equ:a_i=a_i}
				a_i^A=a_i^{A_e},  \quad \quad b_i^A=b_i^{A_e}, \quad \quad a_i^{B'} = a_i^{B'_{e'}}, \quad \quad b_i^{B'} = b_i^{B'_{e'}}
			\end{equation}
			for $1\le i\le g$.

			Note that by \eqref{equ:u+v=u'+v'}, we have $c_{b-1}^A=w_{\eta_A}(\partial_{u+v}S_A)=w_{\eta_B}(\partial_{u'+v'}S_B)=c_{b-1}^{B'}$. Combining \eqref{align:cutting74},  \eqref{equ:A_eA2} and \eqref{equ:A_eA1}, we have
			\begin{equation}\label{equation:ciAB} c_i^A= c_i^{B'} 
			\end{equation} for all $0\leq i \leq b-1$
			and
			\[\#(M_{A}\cap\partial_{j} S_{A})= \#(M_{B}\cap\partial_{j} S_{B}),\]
			for and $0\leq j \leq b-1$. To compare further invariants, we need consider the following two cases. 
			
			\noindent
			\emph{Case 1:} $g=1$. We have
			\begin{align*} 
				\w{\cal A}(\eta_A)&=\gcd(a^A_1,b^A_1,c^A_0+2,c^A_1+2, \dots, c^A_{b-2}+2, c^A_{b-1}+2) \\
				&= \gcd(\gcd(a^A_1,b^A_1,c^A_0+c_{b-1}^A+4,c^A_1+2, \dots, c^A_{b-2}+2), c^A_{b-1}+2) \\
				&= \gcd(\gcd(a^{A_e}_1,b^{A_e}_1,c_0^{A_e}+2,c^{A_e}_1+2, \dots, c^{A_e}_{b-2}+2), c^A_{b-1}+2) \\
				&= \gcd(\w{\mathcal A}(\eta_{A_e}),c^A_{b-1}+2),
			\end{align*} 
			where the second equality follows from the gcd property and the third uses the equalities \eqref{equ:A_eA2},  \eqref{equ:A_eA1} and \eqref{equ:a_i=a_i} involving $A$ and $A_e$.
			Similarly, we have $\w{\mathcal A}(\eta_B)=\gcd(\w{\mathcal A}(\eta_{B'_{e'}}),c_{b-1}^{B'}+2)$. Since $\w{\mathcal A}(\eta_{A_e})=\w{\mathcal A}(\eta_{B'_{e'}})$ by the induction assumption, we have $\w{\mathcal A}(\eta_A)=\w{\mathcal A}(\eta_B)$.
			
			\noindent
			\emph{Case 2:} $g>1.$ We show the following invariants (see Section \ref{subsection:combinatorialinvariants}) coincide.
			\begin{enumerate}[\rm(a)]
				\item[$\bullet$] 
				We claim
				$\sigma(\eta_A)=\sigma(\eta_{B'})$ (see \eqref{align:sigmainvariant} for the definition of the function $\sigma$). 
				
				Indeed, assume $\sigma(\eta_A)=0$. Then all $a_i^A, b_i^A$ and $c_j^A (=c_j^{B'}$ by \eqref{equation:ciAB}) are even, and thus  $\sigma(\eta_{A_e})=0$ by  \eqref{equ:A_eA2},  \eqref{equ:A_eA1} and \eqref{equ:a_i=a_i}. It follows that  
				$\sigma(\eta_{B'_{e'}})=\sigma(\eta_{A_e}) =0$ where the first equality holds since the surface models of $B'_{e'}$ and $A_e$ are isomorphic by the induction hypothesis. So $a_i^{B'}$ and $b_i^{B'}$ are all even for all $1\le i\le g$ and we have $\sigma(\eta_{B'})=0$.
				
				Assume $\sigma(\eta_A)=1$. We consider two cases. 
				If one of $c_j^A$ is odd,  so is $c_j^{B'}$. Then, $\sigma(\eta_{B'})=1$. If all $c_j^A$ (=$c_j^{B'}$) are even, and one of $a_i^A, b_i^A$ is odd, then by $a_i^A = a_i^{A_e}, b_i^A= b_i^{A_e}$ we obtain that $\sigma(\eta_{B'_{e'}})=\sigma(\eta_{A_e})=1$. Therefore, one of $a_i^{B'_{e'}}, b_i^{B'_{e'}}$ must be odd, which implies that $\sigma(\eta_{B'})=1$ (since $a_i^{B'_{e'}}=a_i^{B'}, b_i^{B'_{e'}}=b_i^{B'}$). This proves the claim.
				
				\item[$\bullet$] Assume $\sigma(\eta_A)=\sigma(\eta_{B'})=0$. If there exists $0\le j\le b-1$ such that $c_j^A= 0 \mod 4$, then $c_j^{B'}= 0 \mod 4$ by \eqref{equation:ciAB}. So all the invariants coincide in this case. 
				
				\item[$\bullet$] Now assume $\sigma(\eta_A)=\sigma(\eta_{B'})=0$ and $c_j^A= 2=c_j^{B'} \mod 4$ for all $0\le j\le b-1$. We claim that  $\cal A(\eta_A)=\cal A(\eta_{B'})$.
				Indeed, note that by \eqref{equ:A_eA2} and \eqref{equ:A_eA1}, the Arf invariant $\cal A(\eta_{A_e})$ and $\cal A(\eta_{B'_{e'}})$ are also well-defined, and moreover, they are equal by the induction hypothesis. 
				Since $\mathcal A(\eta) := 
				\sum_{i=1}^g (\frac{1}{2}a_i +1) (\frac{1}{2}b_i+1) \ \mod 2$, which only depends on $a_i$ and $b_i$, then by \eqref{equ:a_i=a_i}
				\[\cal A(\eta_A)=\cal A(\eta_{A_e})= \cal A(\eta_{B'_{e'}})=\cal A(\eta_{B'}).\]
			\end{enumerate}
			The above shows that the invariants of $S_A$ and $S_B$ coincide. By \cite[Corollary 1.10]{LP20}, the assertion holds.
		\end{proof}

		\begin{proof}[The proof of Theorem \ref{Thm:derivedeq}]
			The \lq if' part follows from \cite[Corollary 7.4]{LP20}. The \lq only if' part follows from Propositions \ref{prop:properlp0},  \ref{prop:properlp} and \ref{prop:properlp2} above.
		\end{proof}

		As an application of Theorem~\ref{Thm:derivedeq} we give a complete derived invariant for homologically smooth and proper graded gentle algebras. Recall  the invariants $\sigma(\eta), \w{\mathcal  A}(\eta)$ and $\mathcal A(\eta)$ from Subsection \ref{subsection:combinatorialinvariants}. Then we have the following result.  
		
		\begin{Cor}\label{Cor:complete}
			Let $A$ and $B$ be two graded gentle algebras with associated surface models $(S_{A}, M_{A},\eta_{A},\zD_{A})$ and $(S_{B}, M_{B}, \eta_{B},\zD_{B})$,  respectively. Then $A$ and $B$ are derived equivalent if and only if $S_A$ and $S_B$ have the same number $b$ of boundary components and there exists a numbering of these boundary components such that for each boundary component one has
			\[\#(M_{A}\cap\partial_{j} S_{A})= \#(M_{B}\cap\partial_{j} S_{B}),\]
			\[w_{\eta_{A}}(\partial_{j} S_{A})=w_{\eta_{B}}(\partial_{j} S_{B}),  \]
			and in addition,
			\begin{itemize} 
				\item if $g(S_{A})=g(S_{B})=1$ then $\w{\mathcal A}(\eta_{A})=\w{\mathcal A}(\eta_{B})$; 
				\item if $g(S_{A})=g(S_{B}) >1$ then $\zs(\eta_{A})=\zs(\eta_{B})$ and $\mathcal A(\eta_{A})=\mathcal A (\eta_{B})$ whenever the latter two invariants are defined.
			\end{itemize}
		\end{Cor}
		\begin{proof}
			The proof follows from Theorem \ref{Thm:derivedeq} and \cite[Corollary 1.10]{LP20} (see also Theorem \ref{Prop:homotopic_linefield} in this paper).
		\end{proof}

		As an application of the above corollary, we  give a complete classification of derived equivalence classes  for the graded gentle algebras of the form $(g; m_0, \dotsc, m_u; v)$. Recall that a sufficient condition is given in Proposition \ref{prop:derivedinvariantstandard}.

		\begin{Cor}\label{Prop:classificationA1}Let $A$ and $A'$ be two graded gentle algebras of the form $(g; m_0, \dotsc, m_u;v)$ with numbers $a_i,b_i,c_j$ and $a_i',b_i',c_j'$ defined in the paragraph following Example \ref{Ex:windingnumberg=2}. Then $A$ and $A'$ are derived equivalent if and only if we have
			$c_j=c_j'$ for each $0\leq j\leq b-1$ and in addition 
			\begin{enumerate}[\rm (a)]
				\item if $g=1$ then $\w{\mathcal A}(\eta) = \w{\mathcal A}(\eta')$.
				\item if $g>1$ then one of the following conditions holds 
				\begin{itemize}
					\item 
					$\sigma(\eta)=1=\sigma(\eta')$;
					\item 
					$\sigma(\eta)=0=\sigma(\eta')$ and there exists $0\leq j < b$ such that $c_j= 0 \ \mod 4$;
					\item 
					$\sigma(\eta)=0=\sigma(\eta')$  and $c_j = 2\ \mod 4$ for all $0\leq j < b$,  and 
					\begin{equation*}\sum_{i=1}^g  (\frac{1}{2}a_i 
						+1) (\frac{1}{2}b_i +1) = \sum_{i=1}^g (\frac{1}{2}a'_i 
						+1) (\frac{1}{2}b'_i +1)\  \mod 2.\end{equation*}
				\end{itemize}
			\end{enumerate}
		\end{Cor}
		
		\begin{proof}
			This directly follows from Corollary~\ref{Cor:complete} and Proposition \ref{prop:derivedinvariantstandard}.
		\end{proof}
		
		\begin{Rem}
			Corollary \ref{Prop:classificationA1} (a) implies that a graded gentle algebra of the standard form $(1;1;0)$ with winding numbers $a, b$ is derived equivalent to another graded gentle algebra of the same form with $a', b'$ if and only if $\gcd(a_1,b_1)=\gcd(a'_1,b'_1)$.  In this case,  if $a_1=0=b_1$, then $a'_1=0=b'_1$.
		\end{Rem}

		\section{Pre-silting and partial silting objects in $\per(A)$}
		In this section, for any homologically smooth and proper graded gentle algebra $A$ we give a necessary and sufficient condition under which  all pre-silting objects in $\per(A)$ are partial silting. As a result, we obtain a whole family of examples that have pre-silting objects which are not partial silting in the setting of finite dimensional (ungraded) algebras. 
		
		\subsection{Partial silting objects}
		Recall that a pre-silting object $P$ in $\per(A)$ is \emph{partial silting} if there exists $Q$ such that $P\oplus Q$ is a silting object in $\per(A)$. We now state the main result of this section.
		
		\begin{Thm}\label{Thm:partialsiltingmainresult}
			Let $A$ be a homologically smooth and proper graded gentle algebra. Let $(S, M, \eta, \Delta)$ be a surface model of $A$. Then the following are equivalent.
			\begin{enumerate}[\rm(i)]
				\item Every pre-silting object in $\per(A)$ is partial silting.
				\item Either $S$ is of genus $0$,  or $S$ is of genus $1$ with  
				\begin{equation}\label{equ:Anot=gcdofboudary}
					\w{\A}(\eta)\not=\gcd\{w_{\eta}(\partial_{0}S)+2, \dots, w_{\eta}(\partial_{b-1}S)+2\}. 
				\end{equation}
			\end{enumerate}
		\end{Thm}
		Recall that the invariant $\w{\A}(\eta)$ is defined by 
		\[
		\w{\mathcal A}(\eta)=\gcd\{ w_{\eta}(s), w_{\eta}(t), w_{\eta}(\partial_{0}S)+2, \dots, w_{\eta}(\partial_{b-1}S)+2\}. \]
		We postpone the proof of Theorem \ref{Thm:partialsiltingmainresult} to the end of the section.
		\begin{Rem}\label{Rem:thm5.6explanation}
			\begin{enumerate}[\rm (1)]
				\item
				Since  the zero object is by definition a  pre-silting object, the first statement implies that $\per(A)$ must have a silting object. 
				\item
				The inequality  \eqref{equ:Anot=gcdofboudary} implies that $\w{\A} (\eta) \neq 0$ and $ \gcd\{w_\eta(s), w_\eta(t)\}\neq 0$ for any non-separating simple closed curves $s, t$ which induce a symplectic basis of $\mathrm H_{1}(\bar{S})$ where  $\bar{S}$ is the closed surface obtained from $S$ by filling boundaries. 
			\end{enumerate}
		\end{Rem}

		The following corollary follows from the proof of Theorem \ref{Thm:partialsiltingmainresult}. It answers a question posed to us by M. Kalck and may be used to construct explicit pre-silting objects which are not partial silting.
		\begin{Cor}\label{Cor:prenotpartial}
			Let $A=\Bbbk Q/I$ be a homologically smooth and proper graded gentle algebra. The following are equivalent.
			\begin{enumerate}[\rm(a)]
				\item There is a pre-silting object in $\per(A)$ which is not partial silting. 
				
				\item There exists a  graded gentle algebra $B$ such that there is a triangle equivalence $F:\per(B)\xra{\simeq} \per(A)$ and there is an idempotent $e\in B$ such that a connected component of $B_{e}$ is the standard form $(1;1;0)$ with  $a_1=0=b_1$.  
			\end{enumerate}
			Moreover, if the conditions above are satisfied, the algebra $B$ in $\rm (b)$ can be chosen in such a way that $\per(eBe)$ admits a silting object $P$ and in this case 
			$F(P)$ is a pre-silting object in $\per(A)$ which is not partial silting.
		\end{Cor}
		
		Recall that an algebra is of standard form $(1;1;0)$ with  $a_1=0=b_1$ if and only if it is given by the following quiver  
		\begin{align*}
			\xymatrix{ & 1 \ar@<1pc>[r]^-{\za} \ar@<-1pc>[r]^-{\zg}& 2 \ar[l]_-{\zb}} 
		\end{align*}
		with relations $\{\za\zb, \zb\zg  \}$ and satisfying  $|\za|+|\zb|=1$ and $|\zb|+|\zg|=1$.

		As a corollary of Theorem \ref{Thm:partialsiltingmainresult} we immediately obtain infinitely many finite dimensional algebras having pre-silting objects which are not partial silting.

		\begin{Cor}\label{cor:finitedimensionalgentle}
			Let $A$ be a finite dimensional (ungraded)  gentle algebra, which is homologically smooth. If the genus of the  surface of $A$ is greater than $1$ then there exists a pre-silting object in $\per(A)$ which is not partial silting.
		\end{Cor}

		Note that by Theorem \ref{Thm:partialsiltingmainresult}, there are also finite dimensional ungraded gentle algebras $A$ of genus $1$ such that $\per(A)$ has pre-silting objects which are not partial silting.
		We give two examples  of finite dimensional algebras that have pre-silting objects which are not partial silting. The first example provides an infinite family of such algebras including 
		the one considered in \cite{LZ23}, which corresponds to  the case  $p=q=u=v=1$ in the example below. 
		
		\begin{Ex}\label{Ex:notpartial1}
			Let $A$ be the ungraded  (zero-graded) gentle algebra given by the following quiver $Q$ \[
			\begin{tikzpicture}[scale=1.4]
				\begin{scope}
					\node[circle, inner sep=1pt, minimum size=3pt] (a1) at (0, 0){$1$};
					\node[circle, inner sep=1pt, minimum size=3pt] (a2) at (4, 0){$2$};
					\node[circle, inner sep=1pt, minimum size=3pt] (a3) at (8, 0){$3$};
					\node[circle, inner sep=1pt, minimum size=3pt] (b1) at (1,.4){};
					\node[circle, inner sep=1pt, minimum size=3pt] (b2) at (2,.4){$\dotsb$};
					\node[circle, inner sep=1pt, minimum size=3pt] (b3) at (3,.4){};
					\node[circle, inner sep=1pt, minimum size=3pt] (c1) at (1,-.4){};
					\node[circle, inner sep=1pt, minimum size=3pt] (c2) at (2,-.4){$\dotsb$};
					\node[circle, inner sep=1pt, minimum size=3pt] (c3) at (3,-.4){};
					\node[circle, inner sep=1pt, minimum size=3pt] (d1) at (5,.4){};
					\node[circle, inner sep=1pt, minimum size=3pt] (d2) at (6,.4){$\dotsb$};
					\node[circle, inner sep=1pt, minimum size=3pt] (d3) at (7,.4){};
					\node[circle, inner sep=1pt, minimum size=3pt] (e1) at (5,-.4){};
					\node[circle, inner sep=1pt, minimum size=3pt] (e2) at (6,-.4){$\dotsb$};
					\node[circle, inner sep=1pt, minimum size=3pt] (e3) at (7,-.4){};
					\draw[->]  (a1) -- node[above]{$\alpha_1$}  (b1) ;
					\draw[->]  (b1) -- node[above]{$\alpha_2$}  (b2) ;
					\draw[->]  (b2) -- node[above]{$\alpha_{p-1}$}  (b3) ;
					\draw[->]  (b3) -- node[above]{$\alpha_p$}  (a2) ;
					\draw[->]  (a1) -- node[below]{$\beta_1$}  (c1) ;
					\draw[->]  (c1) -- node[below]{$\beta_2$}  (c2) ;
					\draw[->]  (c2) -- node[below]{$\beta_{q-1}$}  (c3) ;
					\draw[->]  (c3) -- node[below]{$\beta_q$}  (a2) ;
					\draw[->]  (a2) -- node[above]{$\gamma_1$}  (d1) ;
					\draw[->]  (d1) -- node[above]{$\gamma_2$}  (d2) ;
					\draw[->]  (d2) -- node[above]{$\gamma_{u-1}$}  (d3) ;
					\draw[->]  (d3) -- node[above]{$\gamma_u$}  (a3) ;
					\draw[->]  (a2) -- node[below]{$\delta_1$}  (e1) ;
					\draw[->]  (e1) -- node[below]{$\delta_2$}  (e2) ;
					\draw[->]  (e2) -- node[below]{$\delta_{v-1}$}  (e3) ;
					\draw[->]  (e3) -- node[below]{$\delta_v$}  (a3) ;
				\end{scope}
			\end{tikzpicture}
			\]
			with relations $\{\alpha_p\gamma_1,\beta_q\delta_1\}$. Here, $p, q, u, v\geq 1$. By calculating the relevant winding numbers and using  Theorem \ref{Thm:partialsiltingmainresult}, we can easily see that  $\per(A)$ has a pre-silting object which is not partial silting.
			
			Namely,  let $(S, M, \eta, \Delta)$ be the surface model of $A$. Note that $S$ is of genus $1$ with one boundary component and  $(p+q+u+v-2)$  stops, as illustrated in Figure \ref{figure:2polygons}.

			\begin{figure}[ht]
				\begin{tikzpicture}[scale=.8,transform shape]
					\begin{scope}
						\coordinate (d0) at (0:1.5);
						\coordinate (d1) at (15:1.5);
						\coordinate (d2) at (30:1.5);
						\coordinate (d3) at (45:1.5);
						\coordinate (d4) at (60:1.5);
						\coordinate (d5) at (75:1.5);
						\draw[] (1.5,0) arc[start angle = 0, end angle = 90, radius = 1.5];
						\draw[ thick,green] (0,1.5) -- (0,4.5);
						\draw[ thick,green] (1.5,6) -- (4.5,6);
						\draw[thick,green] (6,4.5) -- (6,1.5);
						\draw[thick,green] (4.5,0) -- (1.5,0);
						\begin{scope}[shift={(6, 6)}]
							\coordinate (a0) at (-105:1.5);
							\coordinate (a1) at (-120:1.5);
							\coordinate (a2) at (-135:1.5);
							\coordinate (a3) at (-150:1.5);
							\coordinate (a4) at (-165:1.5);
							\coordinate (a5) at (-180:1.5);
							\draw[] (0,-1.5) arc[start angle = -90, end angle = -180, radius = 1.5];
						\end{scope}
						\begin{scope}[shift={(0, 6)}]
							\draw[]  (1.5,0) arc[start angle = 0, end angle = -90, radius = 1.5];
							\coordinate (b0) at (-20:1.5);
							\coordinate (b1) at (-40:1.5);
							\coordinate (b2) at (-50:1.5);
							\coordinate (b3) at (-70:1.5);
							\node[fill,circle, red, inner sep=0pt, minimum size=3pt] at (-10:1.5){};
							\node[fill,circle, red, inner sep=0pt, minimum size=3pt] at (-30:1.5){};
							\node[fill,circle, red, inner sep=0pt, minimum size=3pt] at (-45:1.5){};
							\node[fill,circle, red, inner sep=0pt, minimum size=3pt] at (-60:1.5){};
							\node[fill,circle, red, inner sep=0pt, minimum size=3pt] at (-80:1.5){};\end{scope}
						\draw[ bend left = 10, thick,green] (a4) to (b0);
						\draw[ bend left=20, thick,green] (a3) to (b1);
						\draw[bend right = 20, thick,green] (d4) to (b2);
						\draw[bend right=10, thick,green] (d5) to (b3);
						\begin{scope}[shift={(6,0)}]
							\draw[]  (-1.5,0) arc[start angle = 180, end angle = 90, radius = 1.5];
							\coordinate (c0) at (110:1.5);
							\coordinate (c1) at (130:1.5);
							\coordinate (c2) at (140:1.5);
							\coordinate (c3) at (160:1.5);
							\node[fill,circle, red, inner sep=0pt, minimum size=3pt] at (100:1.5){};
							\node[fill,circle, red, inner sep=0pt, minimum size=3pt] at (120:1.5){};
							\node[fill,circle, red, inner sep=0pt, minimum size=3pt] at (135:1.5){};
							\node[fill,circle, red, inner sep=0pt, minimum size=3pt] at (150:1.5){};
							\node[fill,circle, red, inner sep=0pt, minimum size=3pt] at (170:1.5){};
						\end{scope}
						\draw[bend right=10, thick,green] (a0) to (c0);
						\draw[bend right = 20,  thick,green] (a1) to (c1);
						\draw[bend left = 20, thick,green] (d2) to (c2);
						\draw[bend left = 10, thick,green] (d1) to (c3);
						\draw[ thick,green] (a2) to (d3);
						
						\node at (2.5,-.3)  {\small $\ell_3$};
						\node at (2.5,6.3)  {\small $\ell_3$};
						\node at (-.3,3.5)  {\small $\ell_1$};
						\node at (6.3,3.5)  {\small $\ell_1$};
						\node at (4, 3.5)  {\small $\ell_2$};
					\end{scope}
					\draw[decoration={markings, mark=at position 0.3 with {\arrow{>}}}, postaction={decorate}, gray] (3,0) -- (3,6);
					\node at (2.8, 4)  {\small $s'$};
					\draw[decoration={markings, mark=at position 0.3 with {\arrow{>}}}, postaction={decorate}, gray] (0,3) -- (6,3);
					\node at (4,2.8)  {\small $t'$};
					\draw[decoration={markings, mark=at position 0.7 with {\arrow{>}}}, postaction={decorate}](82:1.5) to (80:1.498);
					\node at (81:1.8)  {\small $\alpha_1$};
					\foreach \n in {0, 2, 4}
					{
						\node at (70-1.5*\n:1.7) [circle, fill, inner sep=.1pt]{};
					}
					
					\draw[decoration={markings, mark=at position 0.7 with {\arrow{>}}}, postaction={decorate}](53:1.5) to (51:1.498);
					\node at (52:1.8)  {\footnotesize $\alpha_p$};
					\draw[decoration={markings, mark=at position 0.7 with {\arrow{>}}}, postaction={decorate}](39:1.5) to (37:1.498);
					\node at (37:1.9)  {\footnotesize $\delta_1 $};
					\foreach \n in {0, 2, 4}
					{
						\node at (26-1.5*\n:1.7) [circle, fill, inner sep=.1pt]{};
					}
					\draw[decoration={markings, mark=at position 0.7 with {\arrow{>}}}, postaction={decorate}](9:1.5) to (7:1.498);
					\node at (8:1.9)  {\footnotesize $\delta_v$};
					\begin{scope}[shift={(6,6)}, rotate = 180, transform shape]
						\draw[ decoration={markings, mark=at position 0.7 with {\arrow{>}}}, postaction={decorate}](82:1.5) to (80:1.498);
						\foreach \n in {0, 2, 4}
						{
							\node at (70-1.5*\n:1.7) [circle, fill, inner sep=.1pt]{};
						}
						\draw[ decoration={markings, mark=at position 0.8 with {\arrow{>}}}, postaction={decorate}](49:1.5) to (47:1.498);
						\draw[decoration={markings, mark=at position 0.7 with {\arrow{>}}}, postaction={decorate}](35:1.5) to (33:1.498);
						\foreach \n in {0, 2, 4}
						{
							\node at (26-1.5*\n:1.7) [circle, fill, inner sep=.1pt]{};
						}
						\draw[decoration={markings, mark=at position 0.7 with {\arrow{>}}}, postaction={decorate}](8:1.5) to (6:1.498);
					\end{scope}
					
					\begin{scope}[shift={(6,6)}]
						\node at (262:1.9)  {\small $\beta_1$};
						\node at (232:1.9)  {\footnotesize $\beta_q$};
						\node at (217:1.9)  {\footnotesize $\gamma_1 $};
						\node at (188:1.9)  {\footnotesize $\gamma_u $};
					\end{scope}
				\end{tikzpicture}
				\caption{A presentation of the surface of genus $1$ with one boundary component which contains   $p+q+u+v-2$ stops. Here the two arcs $\ell_1$ and $\ell_3$ are identified along the same direction so that the quotient space is $S$. }
				\label{figure:2polygons}
			\end{figure}
			
			Then using  formula \eqref{align:windingnumbersegments} we compute the winding numbers of the two non-separating simple closed curves $s', t'$ in Figure \ref{figure:2polygons}:
			\begin{align*}
				w_\eta(s') &=|\delta_1|+\dotsb+|\delta_v|-|\gamma_1|-\dotsb-|\gamma_u|  =0 \\
				w_\eta(t') &= -|\alpha_1|-\dotsb-|\alpha_p|+|\beta_1|+\dotsb+|\beta_q|=0.
			\end{align*}
			Since $S$ has only one boundary component we have $w_\eta(\partial S) = -2$ by the Poincar\'e--Hopf index theorem \eqref{algin:PHindex}. It follows that  $$\w{\mathcal A}(\eta) = \gcd\{w_\eta(s'), w_\eta(t'),  w_\eta(\partial S)+2\} =0 =\gcd\{ w_\eta(\partial S)+2\}.$$ Then by Theorem \ref{Thm:partialsiltingmainresult} there is a pre-silting object in $\per(A)$ which is not partial silting.

			In fact, we may construct such a pre-silting object explicitly. By Proposition \ref{prop:standardsation} the algebra $A$ is derived equivalent to a graded gentle algebra of the form $(g=1; m_{1}=p+q+u+v-2)$ (see Definition \ref{Def:standardform} or Remark \ref{Rem:mb=1}):
			\begin{align*}
				\begin{tikzpicture}[scale=1.2]
					\begin{scope}
						\node[circle, inner sep=1pt, minimum size=3pt] (a6) at (5.25,0) {\small $ 1'$};
						\node[circle, inner sep=1pt, minimum size=3pt] (a7) at (6.25,0) {\small $2'$};
						\node[circle, inner sep=1pt, minimum size=3pt] (a8) at (7.25,0) {\small $3'$};
						\node[circle, inner sep=1pt, minimum size=3pt] (a9) at (8.25,0) {$\dotsb$};
						\node[circle, inner sep=1pt, minimum size=3pt] (a10) at (10,0) {};
						\draw[transform canvas={yshift=1.5em},->]  (a6) -- node[above]{$\alpha_1'$}  (a7) ;
						\draw[transform canvas={yshift=0em},<-] (a6) --node[above]{$\beta_1'$}(a7);
						\draw[transform canvas={yshift=-1.5em},->] (a6) --node[above]{$\gamma_1'$} (a7);
						\draw[transform canvas={yshift=0em},->] (a7) -- node[above]{\small $\delta_1'$}  (a8); 
						\draw[transform canvas={yshift=0em},->] (a8) -- node[above]{\small $\delta_{2}'$}  (a9); 
						\draw[transform canvas={yshift=0em},->] (a9) -- node[above]{\small $\delta_{p+q+u+v-3}'$}  (a10); 
					\end{scope}
				\end{tikzpicture}
			\end{align*}
			with $|\za_{1}'|+|\zb_{1}'|=1=|\zb_{1}'|+|\zg_{1}'|$.   Take $P = e_3'A'$.  By choosing a standard admissible dissection as in Figure \ref{Fig:StandardSurfaceModel}, we obtain a triangle equivalence  $F:\per(A')\xra{\simeq} \per(A)$ so that $F(P)$ is given by  
			\[
			e_{t(\gamma_1)} A\xrightarrow{\gamma_1} e_2 A \xrightarrow{\alpha_p} e_{s(\alpha_p)} A,
			\]
			where $\gamma_1, \alpha_p$ are arrows in the quiver $Q$. 
			By Corollary \ref{Cor:prenotpartial}, $F(P)$ is a  pre-silting object in $\per(A)$ which is not partial silting.
		\end{Ex}

		\begin{Ex}\label{Ex:notpartial2}
			Let $A$ be the ungraded gentle algebra given by the following quiver 
			\[
			\xymatrix@R=5pc{
				1\ar@/^1.5pc/[rr]^-{\theta} \ar@<-.3pc>[r]_-{\beta} & 2 \ar@<-.3pc>[l]_-{\alpha} \ar[r]^-{\gamma} & 3\ar@<.3pc>[r]^-{\zeta} \ar@<-.3pc>[r]_-{\delta} & 4 & 
			}
			\]
			with relations $\{\beta\alpha, \alpha \theta, \theta \delta, \gamma \zeta\}.$
			
			Note that the surface model of a $A$ is such that $(S, M, \eta, \Delta)$ where $S$ is of genus $1$ with two boundary components, each of which has only one stop. It is illustrated in Figure \ref{Fig:8gon} where the arcs with the same labelings (i.e.\ $\ell_1, \ell_2, \ell_4$) are identified so that the quotient space is the surface $S$.  
			
			Then using  formula \eqref{align:windingnumbersegments} we compute the winding numbers of the two non-separating simple closed curves $s', t'$ illustrated in Figure \ref{Fig:8gon}
			\begin{align*}
				w_\eta(s') &=(1-|\alpha|-|\theta|)+|\gamma| =1 \\
				w_\eta(t') &= 1-|\alpha|-|\beta|=1.
			\end{align*}
			Similarly, the winding numbers of the two boundary components are  
			\begin{align*}
				w_\eta(\partial_0 S) &= (-3+|\beta|+|\alpha|+|\theta|+|\delta|)-|\theta|-|\eta| = -3\\
				w_\eta(\partial_1 S) &=(-1+|\gamma|+|\eta|) - |\alpha|-|\beta|-|\gamma|-|\delta|=-1.
			\end{align*}
			(One may also obtain that $w_\eta(\partial_1 S)= -1$ from $w_\eta(\partial_0 S)=-3$ and \eqref{algin:PHindex}.)  It follows that  $$\w{\mathcal A}(\eta) = \gcd(w_\eta(s'), w_\eta(t'),  w_\eta(\partial_0 S)+2, w_\eta(\partial_1 S)+2) =1 =\gcd\{ w_\eta(\partial_0 S)+2, w_\eta(\partial_1 S)+2\}.$$ Then by Theorem \ref{Thm:partialsiltingmainresult} there is a pre-silting object in $\per(A)$ which is not partial silting.

			\begin{figure}[ht]
				\begin{tikzpicture}[scale=0.3, transform shape]
					\begin{scope}
						\foreach \x in {-3, -2, -1, 0}
						{
							\draw[thick, green] (\x*45-45:5) to (\x*45:5);
						}
						\draw[thick, green] (45:5) to (90:5);
						\draw[thick, green] (90:5) to (135:5);
						\draw[thick, green] (90:5) to (-45:5);
						\draw[] (135:5) to [curve through = {.. (157.5:4)..}] (180:5);
						\node[fill,circle, red, inner sep=0pt, minimum size=6pt] at (157.5:4){};
						\draw[] (0:5) to [curve through = {.. (22.5:4)..}] (45:5);
						\node[fill,circle, red, inner sep=0pt, minimum size=6pt] at (22.5:4){};
						\foreach \x in {1,2,3}
						{\begin{scope}[shift={(-\x*45:5)}]
								\coordinate (a\x) at (112.5-\x*45:1);
								\coordinate (b\x) at (247.5-\x*45:1);
								\coordinate (f\x) at (247.5-\x*45:2);
							\end{scope}
						}

						\begin{scope}[shift={(-45:5)}]
							\coordinate (c1) at (112.5:1);
							\draw[white, line width=5pt] (0,0) -- (a1);
							\draw[white, line width=5pt] (0,0) -- (b1);
							\draw[white, line width=5pt] (0,0) -- (112.5:1);
						\end{scope}
						
						\draw[white, line width=5pt] (-90:5) -- (b2);
						\draw[white, line width=5pt] (-90:5) -- (a2);
						\draw[white, line width=5pt] (-135:5) -- (b3);
						\draw[white, line width=5pt] (-135:5) -- (a3);

						\begin{scope}[shift={(90:5)}]
							\coordinate (d1) at (-22.5:1);
							\coordinate (g4) at (-22.5:2);
							\coordinate (d2) at (-67.5:1);
							\coordinate (d3) at (-157.5:1);
							\coordinate (g3) at (-157.5:2);
							\draw[white, line width=5pt] (0,0) -- (d1);
							\draw[white, line width=5pt] (0,0) -- (d3);
							\draw[white, line width=5pt] (0,0) -- (d2);
						\end{scope}
						
						\draw[bend left, decoration={markings, mark=at position 0.5 with {\arrow{>}}}, postaction={decorate}] (b1) to node[xshift=-8, yshift=3, above]{\huge $\theta$}   (c1); 
						
						\draw[bend left, decoration={markings, mark=at position 0.5 with {\arrow{>}}}, postaction={decorate}] (c1) to node[xshift=3, yshift=5, above]{\huge $\eta$}   (a1); 
						\draw[bend left, decoration={markings, mark=at position 0.5 with {\arrow{>}}}, postaction={decorate}] (b2) to node[xshift=5, yshift=5, above]{\huge $\alpha$}   (a2);
						\draw[bend left, decoration={markings, mark=at position 0.5 with {\arrow{>}}}, postaction={decorate}] (b3) to node[xshift=7, yshift=4, above]{\huge $\beta$}   (a3);
						\draw[bend left, decoration={markings, mark=at position 0.4 with {\arrow{>}}}, postaction={decorate}] (d1) to node[xshift= 7, yshift=-4, below]{\huge $\gamma$}   (d2); 
						\draw[bend left, decoration={markings, mark=at position 0.5 with {\arrow{>}}}, postaction={decorate}] (d2) to node[xshift=-3, yshift=-7, below]{\huge $\delta$}   (d3); 
						\node at (-157.5:5.2)  {\huge $\ell_1$};
						\node at (-112.5:5.2)  {\huge $\ell_2$};
						\node at (-67.5:5.2)  {\huge $\ell_1$};
						\node at (-22.5:5.2)  {\huge $\ell_4$};
						\node at (67.5:5.2)  {\huge $\ell_2$};
						\node at (112.5:5.2)  {\huge $\ell_4$};
						\node at (22.5:2.5)  {\huge $\ell_3$};
						\draw[bend right=15, decoration={markings, mark=at position 0.5 with {\arrow{>}}}, postaction={decorate}] (f2) to (g4);
						\draw[bend left, decoration={markings, mark=at position 0.35 with {\arrow{>}}}, postaction={decorate}] (f3) to (f1);
						\node at (.7, -1.3)  {\huge $s'$};
						\node at (-2.4, -1.2)  {\huge $t'$};
					\end{scope}
				\end{tikzpicture}
				\caption{A presentation of the surface of genus $1$ with two boundary components, each of which has only one stop. Note that the dissection $\Delta$ (given by arcs in green) cuts the surface into two polygons.}
				\label{Fig:8gon}
			\end{figure}
		\end{Ex}

		\subsection{The proof of Theorem \ref{Thm:partialsiltingmainresult}}%
		This subsection is devoted to the proof of Theorem \ref{Thm:partialsiltingmainresult}.
		
		We need the following preparation. Let $A=\Bbbk Q/I$ be any
		graded gentle algebra (not necessarily homologically smooth or proper).  
		Recall from Remark \ref{rem:nonproper}  the set $C(A)$ of cyclic paths in $A$.  Consider the following assumption on $A$:
		\begin{equation}\label{conditionA}
			\mbox{The algebra $A$ is homologically smooth and $|p|<0$ for any $p\in C(A)$. }
		\end{equation}
		Recall that  $A$ is proper if and only if $C(A)=\emptyset$. Also note that
		if $A$ satisfies \eqref{conditionA}, then the precondition for $A$ in Theorem  \ref{Thm:siltingreduction} (b) holds.
		We have the  following  useful observation.
		\begin{Lem} \label{Lem:existsofsiltingforA_{e}}
			Let $A=\Bbbk Q/I$ be a graded gentle algebra satisfying \eqref{conditionA}.  Then the following hold.
			\begin{enumerate}[\rm (a)]
				\item Assume the surface model of $A$ is of  genus zero, then $\per(A)$ has silting objects. 
				\item Let $eA$ be a  pre-silting object in $\per(A)$ where $e=\sum_{k\in \mathcal I} e_k$  is a finite sum of some vertex idempotents. Then $A_e$ also satisfies  condition  $\eqref{conditionA}$.  As a result, if  the surface model of $A$ is of  genus zero then $\per(A_e)$ has silting objects.
			\end{enumerate}
		\end{Lem}
		\begin{proof}
			(a) 
			This directly follows from Theorem \ref{Thm:existence_of_silting}.

			(b) By \cite[Lemma 4.7]{CJS},  $A_{e}$ is  homologically smooth. Let $q$ be any element in $C(A_{e})$ of length $m$. To check the condition \eqref{conditionA}, we only need to show $|q|<0$. By Definition \ref{definition:left algebra} we may write $q$ as a path in $Q_e$ 
			\[  i_{1}\xra{[\za^{1}_{1}\cdots\za^{1}_{l(1)}]}  i_{2} \xra{[\za^{2}_{1}\cdots\za^{2}_{l(2)}]}  i_3\to \cdots\to
			i_{m} \xra{[\za^{m}_{1}\cdots\za^{m}_{l(m)}]} i_{1},\]
			where $t(\za^{j}_{i})\in \{e_i \mid i \in \mathcal I\}$ and $\za^{j}_{i}\za^{j}_{i+1}\in I, \za^{j}_{l(j)}\za^{j+1}_{1}\not\in I$ for any $1\le j\le m$ and $1\le i<l(j)$. Again the indices are taken modulo $m$. 
			
			Denote by $N$  the cardinality of the set $\{ j \mid l(j)>1\}$.
			If $N=0$, i.e.\ $l(j)=1$ for each $1\leq j \leq m$ then $q$ is an element in  $C(A)$ and thus by assumption we have $|q|<0$. 
			
			If $N\geq 1$ then there exist $1\leq j_1< j_2< \dotsb < j_N \leq m$ such that $l(j_1) , \dotsc, l(j_N) >1$. For each $1\leq k\leq N$ consider the following subpath of $q$  
			\[ i_{j_k}\xra{[\za^{j_k}_{1}\alpha_{2}^{j_k}\cdots\za^{j_k}_{l(j_k)}]}  i_{j_k+1} \xra{[\za^{j_k+1}_{1}]}  i_{j_k+2}\to  \cdots \to i_{j_{k+1}-1} \xra{[\za^{j_{k+1}-1}_{1}]}  i_{j_{k+1}} \xra{[\za^{j_{k+1}}_{1}\za^{j_{k+1}}_{2}\cdots\za^{j_{k+1}}_{l(j_{k+1})}]}  i_{j_{k+1}+1}.
			\]   
			Note that $\alpha_i^{j_k}\in eAe$ for each $1< i < l(j_k)$ and $\alpha_{l(j_k)}^{j_k} \alpha_1^{j_k+1}\dotsc \alpha_1^{j_{k+1}-1} \alpha_1^{j_{k+1}}\in eAe.$ Since $eA$ is pre-silting, $eAe$ is non-positive. It follows that $|\alpha_i^{j_k}|\leq 0$ for each $1< i < l(j_k)$, and $|\alpha_{l(j_k)}^{j_k}  \alpha_1^{j_k+1}\dotsc \alpha_1^{j_{k+1}-1} \alpha_1^{j_{k+1}}|\leq 0$. This yields $|\alpha_{2}^{j_k}\alpha_3^{j_k} \dotsc \alpha_{l(j_k)}^{j_k} \alpha_1^{j_{k}+1}\dotsc \alpha_1^{j_{k+1}-1} \alpha_1^{j_{k+1}}| \leq 0$ for $1\leq k\leq N$.  So by Definition \ref{definition:left algebra}, we have
			\[|q|=\sum_{j=1}^{m}\sum_{i=1}^{l(j)}\left(|\za^{j}_{i}|+1-l(j)\right)< \sum_{j=1}^{m}\sum_{i=1}^{l(j)}|\za^{j}_{i}| =\sum_{k=1}^N |\alpha_{2}^{j_k}\alpha_3^{j_k} \dotsc \alpha_{l(j_k)}^{j_k} \alpha_1^{j_{k}+1}\dotsc \alpha_1^{j_{k+1}-1} \alpha_1^{j_{k+1}}|  \leq 0.\]
			This shows that $A_e$ satisfies the condition  \eqref{conditionA}. The last statement then follows from (a) since the surface model of $A_e$ is also of genus zero by Theorem \ref{Thm:siltingreduction} (c). 
		\end{proof}

		The following geometric observation will be used later. 
		\begin{Lem}\label{lemma:toruscutting}
			Let $(S, M, \eta)$ be a graded surface with stops of genus $1$. Let $\{\ell_i\}_{i=1}^n$ be an admissible collection on $S$ such that the cut surface along the  $\{\ell_i\}_{i=1}^n$ is still of genus $1$. Then $\{\ell_i\}_{i=1}^n$ can be extended to an admissible dissection on $S$ such that the associated gentle algebra is given by the  quiver \begin{align}\label{align:quiverextension}
				\xymatrix@1{ &Q= 1 \ar@<1pc>[r]^-{\za_{1}} \ar@<-1pc>[r]^-{\zg_{1}}& 2 \ar[l]_-{\zb_{1}}\ar[r]^-{\delta_1} & *+++[F.]{Q'} \Big).
			} \end{align}
			and the vertices corresponding to $\{\ell_i\}_{i=1}^n$ are contained in $Q_0'$. Here, $\alpha_1\beta_1 = \beta_1 \gamma_1=\zg_{1}\zd_{1}=0$. 
		\end{Lem}
		\begin{proof}
			By assumption, there exist non-separating simple closed curves $s', t'$, which induce a symplectic basis of $\mathrm H_{1}(\bar{S})$,  such that $s'$ and $t'$ do not intersect $\{\ell_i\}_{i=1}^n$ (otherwise, the cut surface would be of genus zero). We may cut the surface $S$ along $s', t'$ into a $4$-gon, see Figure \ref{figure:torus4gon}, so that $\{\ell_i\}_{i=1}^n$ and the boundary components of $S$ lie in the interior of the $4$-gon. 
			
			\begin{figure}[ht]
				\begin{tikzpicture}[scale=.7,transform shape]
					\begin{scope}
						\draw[decoration={markings, mark=at position 0.55 with {\arrow{>}}}, postaction={decorate}, thick] (0,0) -- (0,5);
						\draw[decoration={markings, mark=at position 0.55 with {\arrow{>}}}, postaction={decorate}, thick] (0,5) -- (6,5);
						\draw[decoration={markings, mark=at position 0.55 with {\arrow{<}}}, postaction={decorate}, thick] (6,5) -- (6,0);
						\draw[decoration={markings, mark=at position 0.5 with {\arrow{<}}}, postaction={decorate}, thick] (6,0) -- (0,0);
						\node at (3,-.4)  {\small $s'$};
						\node at (3,5.4)  {\small $s'$};
						\node at (-.4,2.5)  {\small $t'$};
						\node at (6.4,2.5)  {\small $t'$};
						\begin{scope}[shift={(1.8,3)}]
							\draw[fill=gray!15] (0,0) circle (.7);
							\foreach \x in {0,1,...,11}
							{\coordinate (d\x) at (30*\x:.7);
							}
							\draw[thick, green] (d7) ..controls (-120:1).. (d9);
							\draw[thick, green, dotted] (d4) ..controls (150:1.2).. (d6);
							\node at (-135:1)  {\tiny $\ell_2$};
							\draw[thick, green, dotted] (d1) ..controls (20:1.2) and (-20:1.5).. (d11);
							\node[fill,circle, red, inner sep=0pt, minimum size=3pt] at (0:.7){};
							\node[fill,circle, red, inner sep=0pt, minimum size=3pt] at (150:.7){};
							\node[fill,circle, red, inner sep=0pt, minimum size=3pt] at (-120:.7){};
							\draw[decoration={markings, mark=at position 0.55 with {\arrow{<}}}, postaction={decorate}] (-41:.7) -- (-40:.701);
							\draw[decoration={markings, mark=at position 0.55 with {\arrow{<}}}, postaction={decorate}] (-71:.7) -- (-70:.701);
							\draw[decoration={markings, mark=at position 0.55 with {\arrow{<}}}, postaction={decorate}] (-161:.7) -- (-160:.701);
							\draw[decoration={markings, mark=at position 0.55 with {\arrow{<}}}, postaction={decorate}] (-251:.7) -- (-250:.701);
							\draw[decoration={markings, mark=at position 0.55 with {\arrow{<}}}, postaction={decorate}] (-299:.7) -- (-298:.701);
						\end{scope}
						\begin{scope}[shift={(4,3)}]
							\draw[fill=gray!15] (0,0) circle (.7);
							\foreach \x in {0,1,...,11}
							{\coordinate (c\x) at (30*\x:.7);
							}
							\node[fill,circle, red, inner sep=0pt, minimum size=3pt] at (60:.7){};
							\node[fill,circle, red, inner sep=0pt, minimum size=3pt] at (0:-.7){};
							\node[fill,circle, red, inner sep=0pt, minimum size=3pt] at (0:.7){};
							\draw[thick, green] (c1) ..controls (60:1).. (c3);
							\node at (45:1)  {\tiny $\ell_1$};
						\end{scope}
						\draw[thick, green] (d3) ..controls (3,4).. (c4);
						\draw[thick, green] (d10) ..controls (3,2).. (c7);
						\draw[thick, dotted, green] (c9) to (5,0);
						\draw[thick, dotted, green] (c10) to (6, 1);
						\draw[thick, dotted, green] (c11) ..controls (5.5, 2.5).. (5.5,5);
						\draw[thick, dotted, green] (c8) ..controls (3, 1).. (0,1);
						\node[xshift=.2cm, red, thick] at (c0) {$v$};
						\node at (2,.7)  {\tiny$\ell_2'$};
						\node at (5.5,1)  {\tiny$\ell_2'$};
						\node at (4.5,.6)  {\tiny$\ell_1'$};
						\node at (5.25,4.3)  {\tiny$\ell_1'$};
						\node at (3, 4.1)  {\tiny $\ell_3$};
						\node at (2.7,2)  {\tiny$\ell_4$};
						\begin{scope}[shift={(4,3)}]
							\node at (-40:1)  {\tiny$\gamma_1$};
							\node at (-65:1)  {\tiny $\beta_1$};
							\node at (-100:1)  {\tiny $\alpha_1$};
							\node at (-130:1)  {\tiny $\delta_1$};
							\draw[decoration={markings, mark=at position 0.55 with {\arrow{<}}}, postaction={decorate}] (-41:.7) -- (-40:.701);
							\draw[decoration={markings, mark=at position 0.55 with {\arrow{<}}}, postaction={decorate}] (-71:.7) -- (-70:.701);
							\draw[decoration={markings, mark=at position 0.55 with {\arrow{<}}}, postaction={decorate}] (-98:.7) -- (-97:.701);
							\draw[decoration={markings, mark=at position 0.55 with {\arrow{<}}}, postaction={decorate}] (-128:.7) -- (-127:.701);
							\draw[decoration={markings, mark=at position 0.55 with {\arrow{<}}}, postaction={decorate}] (-248:.7) -- (-247:.701);

						\end{scope}
					\end{scope}
				\end{tikzpicture}
				\caption{An example on how an admissible collection $\{\ell_i\}_{i=1}^n$ may be extended to an admissible dissection whose quiver is given by \eqref{align:quiverextension}.}
				\label{figure:torus4gon}
			\end{figure}
			
			Since $\{\ell_i\}_{i=1}^n$ forms an admissible collection, there exists a stop which connects to the edges of the $4$-gon without crossing the arcs $\{\ell_i\}_{i=1}^n$. Let $v$ be such a stop  as illustrated in Figure \ref{figure:torus4gon}.  Then we may extend $\{\ell_i\}_{i=1}^n$ to an admissible dissection such that next to $v$ there are exactly two new arcs $\ell_1', \ell_2'$, which are homotopic to $s'$ and $t'$ (after smoothing) respectively, and all the other new  (dotted) arcs are just around stops. This admissible dissection gives rise to the quiver $Q$. In particular, the arc $l_1'$ (resp.\ $l_2'$)  corresponds to the vertex $1$ (resp.\ $2$) in $Q$.
		\end{proof}

		The following lemma shows that $(i) \Rightarrow (ii)$ in Theorem \ref{Thm:partialsiltingmainresult}. We recall Definition \ref{Def:standardform} for the gentle algebras of standard form $(g; m_0,\dotsc, m_u; v)$. 
		\begin{Lem}\label{Lem:(a)to(b)}
			Let $A$ be a homologically smooth and proper  graded gentle algebra with  surface model $(S, M, \eta, \Delta)$.  If either $S$ is of genus $1$ with 
			\begin{equation} \label{equ:A=gcdofboundary}
				\w{\A}(\eta)=\gcd\{w_{\eta}(\partial_{0}S)+2, \dots, w_{\eta}(\partial_{b-1}S)+2\}, 
			\end{equation}
			or  $S$ is of genus strictly greater than $1$, 
			then there exists a pre-silting object in  $\per(A)$ which is not partial silting. 
		\end{Lem}
		\begin{proof}
			By Proposition \ref{prop:standardsation}, $A$ is derived equivalent to an algebra $A'$ of standard form $(g; m_0, \dotsc, m_u;0)$. The line field $\eta$ induces a grading of $A'$. Let $a_i,b_i,c_j$ be the numbers defined in \eqref{align:numbersofthesumgradingsAn}. Then by Remark \ref{Rem:furtherstandard}(2), we may assume that $a_1=0=b_1$ in both cases, and $a_i, b_i< 0$ for $2\leq i \leq g$ in the second case.

			Let $e=1-e_{1}-e_{2}$, where $e_1,e_2$ are the idempotents corresponding to the leftmost two vertices in the quiver $Q(g;m_0,\dotsc, m_u;0)$, see Definition \ref{Def:standardform}. Then $A'_{e}$ is of the standard form $(1;1;0)$ with $a_1=0=b_1$. Note that $eA'e$ is of the form $(g-1; m_0, \dotsc, m_u;0)$ with winding numbers $a_2,b_2, \dotsc, a_g, b_g <0$,  which cannot be of the standard form $(1;1;0)$ with winding numbers $0, 0$. Therefore, by Theorem \ref{Thm:existence_of_silting},  $\per(A'_{e})$ has no silting objects and $\per(eA'e)$ has a silting object. The latter induces a silting object, say $P$, in $\thick(eA')\subset \per(A')$ by Remark \ref{Rem:thm5.2b}. It follows from Theorem \ref{Thm:siltingreduction} (b) that $P$ cannot be partial silting since the set on the left hand side of the bijection is empty.
		\end{proof}
		
		Now we are ready to prove Theorem \ref{Thm:partialsiltingmainresult}.
		
		\begin{proof}[Proof of Theorem \ref{Thm:partialsiltingmainresult}]
			$(i) \Rightarrow (ii)$ is given by Lemma \ref{Lem:(a)to(b)}. 
			
			$(ii)\Rightarrow (i)$. Let $P$ be any pre-silting object in $\per(A)$. We need to show that $P$ is partial silting. By Remark \ref{Rem:siltingdissection}(1), $P$ gives rise to an admissible collection $\{\ell_i\}_{i=1}^n$ on $S$. Let $\Delta_P$ be any admissible dissection on $S$ extending  $\{\ell_i\}_{i=1}^n$ and let  $B$ be the graded gentle algebra corresponding to $\Delta_P$. Denote by  $e$ the sum of vertex idempotent in $B$ corresponding to  the arcs $\{\ell_i\}_{i=1}^n$. Then by \cite{HKK17} there is a triangle equivalence $F:\per(A)\simeq \per(B)$ such that $F(P)\cong eB$. Hence, $eB$ is a pre-silting object in $\per(B)$, and 
			to show $P$ is  partial silting in $\per(A)$, it  is equivalent to show that $eB$ is  partial silting in $\per(B)$. 
			For this, we claim that $\per(B_{e})$ has silting objects. Then the assertion follows from Theorem \ref{Thm:siltingreduction} (b).

			Let us prove this claim. Note that the cut surface $(S_{P}, M_{P}, \zD_{P}\backslash \{ \ell_{i}\}_{i=1}^{n})$   in  Theorem \ref{Thm:siltingreduction} (c) is a surface model of $B_{e}$, and that the genus of $S_{P}$ is not greater than the genus of $S$ (since cutting does not increase the genus of the surface).   We prove the claim by considering the following two cases.

			\textbf{Case 1}  \ \  The surface $S_{P}$ is of genus $0$. Then the claim follows directly from Lemma \ref{Lem:existsofsiltingforA_{e}} (b) since $eB$ is a pre-silting object in $\per(B)$. 
			
			\textbf{Case 2}  \ \  
			The  surface $S_{P}$ is of genus $1$. In this case $S$ has to be also of genus $1$ such that \eqref{equ:Anot=gcdofboudary} holds. By Lemma \ref{lemma:toruscutting}, we may assume that  $\Delta_P$ is such that the corresponding algebra $B$ is 
			given by the quiver of the following form
			\begin{align*}
				\xymatrix@1{ &Q= 1 \ar@<1pc>[r]^-{\za_{1}} \ar@<-1pc>[r]^-{\zg_{1}}& 2 \ar[l]_-{\zb_{1}}\ar[r]^-{\delta_1} & *+++[F.]{\widetilde Q}.
			} \end{align*}
			such that the vertices corresponding to $\{\ell_i\}_{i=1}^n$ are contained in $\widetilde Q_0$. Note that using formula \eqref{align:windingnumbersegments} we may calculate that 
			$w_{\eta}(s')=|\za_{1}|+|\zb_{1}|-1=a_1$ and $w_{\eta}(t')=|\zb_{1}|+|\zg_{1}|-1=b_1$, where $s', t'$ are the non-separating simple closed curves as in the proof of Lemma \ref{lemma:toruscutting}. In particular, they induce a symplectic basis of $\mathrm H_{1}(\bar{S})$. Then condition \eqref{equ:Anot=gcdofboudary} implies that $\gcd \{w_{\eta}(s'), w_{\eta}(t')\}\not=0$, see Remark \ref{Rem:thm5.6explanation}. So $(a_1, b_1) \neq (0,0).$

			Since  $\{\ell_i\}_{i=1}^n$ is contained in $\widetilde Q_{0}$, the quiver $Q_{e}$ of $B_{e}$ is of the following form (see Definition \ref{definition:left algebra})
			\begin{align*}
				\xymatrix{ &Q_e= 1 \ar@<1pc>[r]^-{\za_{1}} \ar@<-1pc>[r]^-{\zg_{1}}& 2 \ar[l]_-{\zb_{1}}\ar[r]^-{\delta_1} & *++[F.]{\widetilde Q_e}
				}
			\end{align*}
			Note that $B_e$ might not be proper but it satisfies the condition \eqref{conditionA} by Lemma \ref{Lem:existsofsiltingforA_{e}} (2). 
			Therefore, by Theorem \ref{Thm:existence_of_silting} $\per(B_{e})$ has silting objects by Theorem \ref{Thm:siltingreduction} (b). This proves the claim.
		\end{proof}

		\begin{proof}[Proof of Corollary \ref{Cor:prenotpartial}]
			$ (a) \Rightarrow (b)$ If there is a pre-silting object in $\per(A)$ which is not partial silting, then by Theorem \ref{Thm:partialsiltingmainresult}, the surface model of $A$ is either of genus $1$ with \begin{equation} 
				\w{\A}(\eta)=\gcd\{w_{\eta}(\partial_{0}S)+2, \dots, w_{\eta}(\partial_{b-1}S)+2\}, 
			\end{equation}
			or  $S$ is of genus strictly greater than $1$. By Remark \ref{Rem:furtherstandard}, $A$ is derived equivalent to a graded gentle algebra $B$ of standard form 
			$(g; m_0, \dotsc, m_u;0)$, such that the corresponding sum of gradings $a_{1}$ and $b_{1}$ defined in \eqref{align:numbersofthesumgradingsAn} are both equal to $1$ and $a_i, b_i< 1$ for $2\leq i \leq g$ in the second case. Let $e:=1-e_{1}-e_{2}\in B$. Then $B_{e}$ is of the standard form $(1;1;0)$ with $a_1=0=b_1$. Note that in this case $\per(eBe)$ cannot be of the form $(1;1;0)$ with $(a_1,b_1)=(0,0)$ and hence it has silting objects.   
			
			$(b) \Rightarrow (a)$  We only need to show that there is a pre-silting object in $\per(B)$ which is not partial silting.  This follows directly from 
			Theorem \ref{Thm:partialsiltingmainresult} if the surface $S$ associated  to $B$ is of genus strictly greater than $1$.  
			We are left to consider the case where $S$ is of genus $1$.  By assumption, there is an idempotent $e\in B$ such that a connected component of $B_{e}$ is  of the form $(1;1;0)$ with $(a_1,b_1)=(0,0)$ and thus $\per(B_e)$ has no silting objects by Theorem~\ref{Thm:existence_of_silting}.   On the other hand, since the surface model of $B_e$ is of genus one it follows that the arcs corresponding to $e$ are not isotopic to non-separating simple closed curves and thus $eBe$ is of genus $0$. By Theorem \ref{Thm:existence_of_silting} there is a silting object $P\in \thick(eB)\simeq \per(eBe)$. Then $P$ is a pre-silting object in $\per(B)$ which is not partial silting. 
		\end{proof}

		We end this section by showing that almost complete pre-silting objects are always partial silting. 
		
		\begin{Prop}\label{Cor:n-1prepartial}
			Suppose that  $A=\Bbbk Q/I$ is a homologically smooth and proper graded gentle algebra with $|Q_0|=n$. Let $P$ be a basic pre-silting object with $n-1$ indecomposable direct  summands. Then $P$ is partial silting.
		\end{Prop}
		
		\begin{proof}
			Let $(S, M, \eta, \Delta)$ be a  surface model of $A$. By Proposition~\ref{Prop:behaviorsiltingobjects}, $P$ gives rise to an admissible collection on $S$, which may be extended to an admissible dissection $\Delta'$ by adding one arc $\ell'.$ Denote by $A'=\Bbbk Q'/I'$ the graded gentle algebra corresponding to $\Delta'$. Then by \cite{HKK17} we have a triangle equivalence $F \colon \per(A) \to \per(A')$ which sends $P$ to $eA'$ for $e = 1-e'$, where $e'$ corresponds to the new arc $\ell'$. 
			
			By Theorem \ref{Thm:siltingreduction}, it suffices to show that $\per(A'_{e})$ has silting objects. For this, since by \cite[Lemma 4.7]{CJS}, $A_{e}$ is a homologically smooth graded gentle algebra  with only one vertex, we have either $A_{e}\cong \Bbbk$ or $A_{e}\cong \Bbbk[x]$. By Lemma \ref{Lem:existsofsiltingforA_{e}} (b) the algebra $A_e$ satisfies the condition  \eqref{conditionA}. Namely, $|x|< 0$ for the second case. Therefore, $A_e$ is a silting object in $\per(A_e)$.
		\end{proof}

		\appendix

		\section{Primitive closed curves have nonzero self-extensions}
		In this appendix, we will show the following result (see also \cite[Theorem 3.3]{OPS18}), which is used in the proof of Proposition \ref{Prop:behaviorsiltingobjects}. 
		
		\begin{Prop}\label{Prop:simpleclosedcurves}
			Let $A$ be a homologically smooth graded gentle algebra.
			Let $X$ be an indecomposable object in $\per(A)$ corresponding to a primitive closed curve.  Then $\Hom_{\per(A)}(X, X[1]) \neq 0$.
		\end{Prop}

		We fix the following convention. 
		
		\begin{Convention}\label{Convention:mappath}
			Let $i, j$ be two vertices in $Q_0$. Any path $p\notin I$ from $i$ to $j$ induces a nonzero morphism of right $A$-modules 
			\[
			p^*\colon P_j \to P_i, \quad  u \mapsto pu
			\]
			Here $P_i= e_i A$ (resp.\ $P_j=e_jA$) is the indecomposable projective right dg $A$-module corresponding to the vertex $i$ (resp.\ $j$). By abuse of notation, we shall identify the path $p$ with its corresponding map $p^*$. 
		\end{Convention}
		
		\begin{proof}[Proof of Proposition \ref{Prop:simpleclosedcurves}]

			By \cite[Subsection 4.1]{HKK17}, see also \cite{OPS18}, the indecomposable object $X$ (corresponding to a primitive closed curve) can  be read as a twisted complex 
			\[
			X=\left(\bigoplus_{i=1}^n P_i[k_i] \otimes V, \beta_0 \otimes T + \sum_{i=1}^{n-1} \beta_i \otimes \mathrm{Id}_V\right).
			\]
			Here, $P_i$ is the projective indecomposable at vertex $i$, $\beta_0$ is a path of degree $k_n-k_1+1$ from vertex $n$ to vertex $1$, which induces a dg  $A$-module morphism of degree one from $P_1[k_1]$ to $P_n[k_n]$ (compare Convention \ref{Convention:mappath}) and $V$ is a finite dimensional indecomposable $\Bbbk [T]$-module such that $T$ induces an automorphism of $V$.  Furthermore, $\beta_i$ is a path either from $i$ to $i+1$ or a path from $i+1$ to $i$ and the $k_i$ satisfy the following equation
			\[     
			k_{i+1} = \begin{cases}
				k_i +|\beta_{i}|-1 & \textrm{if $\beta_{i}$ is a path $i+1$ to $i$} \\
				k_i - |\beta_{i}|+ 1 & \textrm{if $\beta_{i}$ is a path $i$ to $i+1$,}
			\end{cases}
			\]
			for each $i\in \{1,\ldots, n-1\}$.

			We  consider the case where $\dim_\Bbbk V = 1$, the case of $\dim_\Bbbk V > 1$ following a similar argument. In this case, $T$ corresponds to the multiplication by a nonzero scalar $\zl$ in $\Bbbk$. Then up to changing $A$ by to a derived equivalent graded gentle algebra $A'$,  $X$ can always be illustrated as the following (infinitely  repeating) diagram 
			\begin{align*}\footnotesize
				\xleftarrow{\beta_{2k}^1} \underbrace{X_1^0 \xrightarrow{\alpha_1^1} X_1^1 \xrightarrow{} \dotsc \xrightarrow{\alpha_1^{i_1}} X_1^{i_1}} \overbrace{ \xleftarrow{\beta_2^{i_2}} X_2^{i_2} \xleftarrow{} \dotsc \leftarrow X_2^2 \xleftarrow{\beta_2^1} X_2^1}  \underbrace{\xrightarrow{\alpha_3^1} X_3^1 \xrightarrow{} \dotsc \xrightarrow{\alpha_3^{i_3}} X_3^{i_3}} \dotsc \overbrace{ \xleftarrow{\beta_{2k}^{i_{2k}}} X_{2k}^{i_{2k}} \xleftarrow{} \dotsc \xleftarrow{\beta_{2k}^1} X_{2k}^1} \xrightarrow{\alpha_1^1} 
			\end{align*}
			Here $X_1^0=X_{2k}^1$,  for some $k \geq 1$. The $X_i^j$'s are (shifted) indecomposable projective dg $A$-modules (e.g.\ $P_i[k_i]$), and the  $\alpha_i^j$'s and $\beta_i^j$'s are paths in $Q$ by Convention \ref{Convention:mappath}. Note that the arrows in each brace have the same direction (e.g.\ $\alpha_{2j-1}^i$'s go from the left to the right and $\beta_{2j}^i$'s go from the right to the left) and the composition of any two consecutive arrows in the same direction is zero. 
			We may assume that $i_1, i_2, \dotsc, i_{2k} >0$. 
			
			The simple closed curves and the arcs in the admissible dissection are assumed  to be in {\it minimal} position, see e.g.\ \cite[Subsection 2.2]{OPS18}. In particular, this implies that,  for each $1\leq j \leq k$, the paths $\alpha_{2j-1}^{i_{2j-1}}$ and $\beta_{2j}^{i_{2j}}$ have distinct starting arrows with the same starting vertex, and similarly the paths $\alpha_{2j+1}^1$ and $\beta_{2j}^{1}$ have  distinct ending arrows with the same ending vertex ($\alpha_{2k+1}^1:=\alpha_1^1$). This observation plays an important role in the following argument.

			Note that $\alpha_1^1$ (the diagonal map in the diagram below) induces a cocycle of degree one when viewed as an element in the Hom complex  $\shHom_A(X,X)$. 
			
			\begin{align*}\footnotesize
				\xymatrix@R=2.5pc{
					X_{2k}^2 \ar@{.>}@/_.8pc/[drr]|(.3){h_{2k,1}^{2,1}}&  \ar[l]_-{\beta_{2k}^1} X_1^0\ar[d]|(.4){h_{1,1}^{0,0}} \ar[r]^-{\alpha_1^1}  \ar[rd]^-{\alpha_1^1} &  X_1^1\ar[d]^-{h_{1,1}^{1,1}} \ar[r]^-{\alpha_1^2} & X_1^2 \ar[d]^-{h_{1,1}^{2,2}}  \ar[r]^-{\alpha_1^3} &  \dotsb  \ar[r]& X_1^{i_1-1} \ar[d]^-{h_{1,1}^{i_1-1,i_1-1}} \ar[r]^-{\alpha_1^{i_1}} &   X_1^{i_1} \ar[d]^-{h_{1,1}^{i_1,i_1}}    & X_2^{i_2} \ar[d]^-{h_{2,2}^{i_2,i_2}} \ar[l]_-{\beta_2^{i_2}} &  \ar[l]_-{\beta_2^{i_2-1}} X_2^{i_2-1} &    \dotsc \ar[l] &  X_{2k}^1  \ar[l]_-{\beta_{2k}^1} \ar[d]_-{h_{2k,2k}^{1,1}} &  \\
					X_{2k}^2 &  \ar[l]   X_1^0 \ar[r]_-{\alpha_1^1}  &  X_1^1  \ar[r]_-{\alpha_1^2}  & X_1^2 \ar[r]^-{\alpha_1^3} &  \dotsb  \ar[r]& X_1^{i_1-1}  \ar[r]_-{\alpha_1^{i_1}} &  X_1^{i_1}      & X_2^{i_2}  \ar[l]^-{\beta_2^{i_2}} &  \ar[l]^-{\beta_2^{i_2-1}} X_2^{i_2-1} &    \dotsc \ar[l] &  X_{2k}^1  \ar[l]_-{\beta_{2k}^1} & 
				}
			\end{align*}
			We claim that $\alpha_1^1$ is not a coboundary. Then  $\alpha_1^1$  induces a nonzero morphism in $\Hom_{\per(A)}(X,X[1])= \h^0 \shHom_A(X,X[1])$, which implies the assertion.

			To show this claim,   we follow the argument in \cite[the proof of Proposition 4.8]{ALP16}. Assume that $\alpha_1^1$ is a coboundary. Then there exists 
			\[h:=\sum_{i, j}\sum_{p,q} h_{i, p}^{j, q} \in \shHom^0_A(X,X) =\bop_{i,j}\bop_{p,q}\shHom_{A}^{0}(X_{i}^{j},X_{p}^{q}),\]
			where  $h_{i, p}^{j, q} \in \shHom^0_A(X_i^j, X_p^q)$ such that 
			\begin{equation}\label{equ:deltah}
				\delta(h) =\alpha_1^1\in \shHom_{A}^{1}(X_{1}^{0}, X_{1}^{1})\subset \shHom^1_A(X,X).
			\end{equation}
			Here, $\delta$ is the differential of $\shHom_A(X,X)$.
			Note that by the definition of $\delta$, we have 
			\[ \zd(h)=\sum_{i,j,p,q}\zd(h_{i,p}^{j,q}). \]
			Restricting  the equation  \eqref{equ:deltah} to the component $\shHom^1_A(X_1^0,  X_1^1)$ we obtain 
			\begin{align}\label{align:homotopy}
				\alpha_1^1 h_{1,1}^{0,0} - h_{1,1}^{1,1}  \alpha_1^1-h_{2k,1}^{2,1} \beta_{2k}^1= \alpha_1^1.  
			\end{align}
			Since the paths $\alpha_1^1$ and $\beta_{2k}^1$ have distinct ending arrows by the observation mentioned above, we infer that $h_{2k,1}^{2,1} \beta_{2k}^1$ and $ \alpha^1_1$ are  linearly independent (whatever $h_{2k,1}^{2,1}$ is).
			It follows from the equation \eqref{align:homotopy} that $h_{2k,1}^{2,1} \beta_{2k}^1 = 0$ and  $h_{1,1}^{0,0},\;  h_{1,1}^{1,1}$ are just multiplies of the corresponding idempotents.  By abuse of notation we may just view them as elements in $\Bbbk$, so we have 
			\begin{align}\label{align:homotopy1}
				h_{1,1}^{0,0} - h_{1,1}^{1,1} = 1.
			\end{align}
			Without loss of generality, let us assume that $h_{1,1}^{1,1} \neq 0$. 
			Restricting the equation  \eqref{equ:deltah} to the component $\shHom^1_A(X_1^1,  X_1^2)$ we obtain
			\begin{align}\label{align:homotopy1.5}
				\alpha_{1}^2 h_{1,1}^{1,1}  - h_{1,1}^{2, 2} \alpha_1^2 = 0.
			\end{align}
			Since $0\neq h_{1,1}^{1,1}\in\Bbbk$, it follows from the equation \eqref{align:homotopy1.5} that $h_{1,1}^{2,2}$ must be also a multiple of the corresponding idempotent and moreover $h_{1,1}^{1,1} =  h_{1,1}^{2,2}$. Applying the above argument to the components $\shHom^1_A(X_1^{j-1},  X_1^{j})$ for $2 \leq j \leq i_1-1$ we obtain 
			\begin{align}\label{align:homotopy2}
				0\neq h_{1,1}^{1,1} = h_{1,1}^{2,2} = \dotsb = h_{1,1}^{i_1-1,i_1-1} \in \Bbbk.
			\end{align}
			
			Again, restricting $\delta(h) =  \alpha^1_1$ to the component $\Hom^1_A(X_1^{i_1-1},  X_1^{i_1})$ we obtain
			\[
			\alpha_1^{i_1} h_{1,1}^{i_1-1,i_1-1}  + \beta_2^{i_2} h_{1, 2}^{i_1-1, i_2} -  h_{1,1}^{i_1,i_1} \alpha_1^{i_1} = 0.
			\]
			Since the paths $\alpha_1^{i_1}$ and $\beta_2^{i_2}$ have distinct starting arrows by the observation mentioned above, we infer that $\beta_2^{i_2} h_{1, 2}^{i_1-1, i_2} $ and $\alpha_1^{i_1}$ are linearly independent and thus $  \beta_2^{i_2} h_{1, 2}^{i_1-1, i_2}  =0$. Since by \eqref{align:homotopy2}, we have  $0\neq h_{1,1}^{i_1-1,i_1-1}\in \Bbbk$ and it follows that $h_{1,1}^{i_1,i_1}$ is a multiple of the idempotent. Moreover, 
			\begin{align}\label{align:homotopy3}
				h_{1,1}^{i_1-1, i_1-1} = h_{1,1}^{i_1, i_1}.
			\end{align}
			
			Repeating the above argument we may show that 
			\begin{align}\label{align:homotopy4}
				h_{1,1}^{i_1,i_1} = h_{2,2}^{i_2,i_2}=\dotsb =h_{2,2}^{1,1}=\dotsb= h_{2k,2k}^{i_{2k},i_{2k}} = \dotsb = h_{2k,2k}^{1,1}.
			\end{align}
			Note that $h_{2k,2k}^{1,1}$ is the same as $h_{1,1}^{0,0}$ since $X_{2k}^1=X_1^0$. Combing this with \eqref{align:homotopy2}-\eqref{align:homotopy4} we obtain that $h_{1,1}^{0,0}= h_{1,1}^{1,1}$ contradicting \eqref{align:homotopy1}. This proves that $\alpha^1_1$ is not a coboundary. 
		\end{proof}

		\begin{Rem}\label{remark:appendix}
			Denote by $[\alpha_1^1]$ the nonzero element in $\Hom_{\per(A)}(X,X[1])$ corresponding to the cocycle $\alpha_1^1$ considered in the proof of Proposition \ref{Prop:simpleclosedcurves}. Note that $$[\alpha_1^1]^2 =[\alpha_1^1\circ \alpha_1^1] =0 \in \Hom_{\per(A)}(X,X[2]).$$ It follows that the graded endomorphism ring of $X$ in $\per(A)$ contains a nonzero nilpotent element. 
		\end{Rem}

		\section*{Acknowledgements}
		We would like to thank Martin Kalck, Yu Zhou and  Alexandra Zvonareva for useful discussions. The first-named author would like to thank Osamu Iyama for helpful comments regarding partial silting objects and he also thanks the Institute of Algebra and Number Theory at the University of Stuttgart for the hospitality while working on this paper.


	\end{document}